\documentclass{amsart}
\usepackage[margin=1.2in]{geometry}
\usepackage{amsmath,amsthm,amssymb,amsfonts}
\usepackage{xcolor}
\usepackage{mathrsfs}
\usepackage{verbatim}
\usepackage{float}
\usepackage{graphicx}
\usepackage{tikz-cd}
\usepackage{marvosym}
\usepackage{hyperref}
\hypersetup{
    colorlinks=true,
    }
\usepackage[capitalize]{cleveref}
\usepackage[shortlabels]{enumitem}

\newcommand{\Z}{\mathbb{Z}}
\newcommand{\R}{\mathbb{R}}
\newcommand{\N}{\mathbb{N}}
\newcommand{\C}{\mathbb{C}}
\newcommand{\CP}{\mathbb{CP}}
\DeclareMathOperator{\Mod}{Mod}
\DeclareMathOperator{\SMod}{SMod}
\DeclareMathOperator{\Diff}{Diff}
\DeclareMathOperator{\Homeo}{Homeo}
\DeclareMathOperator{\Id}{Id}

\DeclareMathOperator{\Aut}{Aut}
\DeclareMathOperator{\Symp}{Symp}
\DeclareMathOperator{\Sp}{Sp}

\DeclareMathOperator{\PD}{PD}

\theoremstyle{theorem}
\newtheorem{thm}{Theorem}[section]
\newtheorem{cor}[thm]{Corollary}
\newtheorem{prop}[thm]{Proposition}
\newtheorem{lem}[thm]{Lemma}

\newtheorem{question}[thm]{Question}

\crefname{thm}{theorem}{theorems}

\theoremstyle{definition}
\newtheorem{defn}[thm]{Definition}

\theoremstyle{remark}
\newtheorem{rmk}[thm]{Remark}

\begin{document}

\title{Infinitely many Lefschetz pencils on ruled surfaces}
\author{Seraphina Eun Bi Lee}
\address{Department of Mathematics, Harvard University}
\email{slee@math.harvard.edu}
\author{Carlos A. Serv\'an}
\address{Department of Mathematics, University of Chicago}
\email{cmarceloservan@uchicago.edu}
\maketitle
\vspace{-10mm}
\begin{abstract}
   We show that any ruled surface $X$ with $\chi(X) < 0$ admits infinitely many inequivalent Lefschetz pencils of
    fixed genus and number of base points. Our proof proceeds by building
  infinitely many inequivalent Lefschetz fibrations on a blow-up $X \# 4 \overline{\CP^2}$ of $X$ with constant
  fiber class, via a
  mechanism known as partial conjugation. Furthermore, there exists a symplectic form on $X$ compatible with all such
  pencils, and similarly for the fibrations in $X\#4\overline{\CP^2}$. This provides
  the first example of this phenomenon and makes progress on Problem 4.98 of the K3 list of problems in low-dimensional topology in the case of ruled surfaces.
\end{abstract}

\tableofcontents

\section{Introduction}

Let $X$ be a closed, symplectic $4$-manifold. The work of Donaldson~\cite{donaldson}
and Gompf~\cite{gompf-symp} builds a correspondence between topological properties of $X$ and properties
of the \emph{topological Lefschetz pencils} $(X,f)$: Any
symplectic $4$-manifold admits a pencil, and any $4$-manifold admiting a pencil is symplectic.
\begin{center}
  \begin{tikzcd}
    \{ \mbox{closed symplectic $4$-manifolds}\} \arrow[bend right=30,swap]{r}{\textrm{Donaldson}} &
    \arrow[bend right=30, swap]{l}{\textrm{Gompf}} \{ \mbox{Lefschetz pencils}\}
    \end{tikzcd}
\end{center}
\mbox{}\\
This correspondence has proved remarkably fruitful in the study
of symplectic manifolds (e.g.~\cite{donaldson-smith,stipsicz-chern-numbers,usher,fintushel-stern,baykur-hayano2,
baykur-korkmaz-exotic,baykur-hamada,ABKP,BHM}). Nevertheless,
there is not an a priori canonical Lefschetz pencil
on a given symplectic manifold $X$. A fundamental open question is then the following.

\begin{question}[{K3 \cite[Problem 4.98]{K3}; cf. Baykur--Hayano \cite[Question 6.4]{baykur-hayano2}}]\label{question:basic}
  Does every closed symplectic $4$-manifold admit inequivalent Lefschetz pencils with the same fiber genus $g$, for sufficiently large g? How
  about infinitely many?
\end{question}
In this work we make progress on this problem in the case of \emph{ruled surfaces}, or $S^2$-bundles over $\Sigma_g$ for some $g \geq 0$. In particular, we provide the first examples of infinitely many inequivalent pencils of fixed genus on the same symplectic $4$-manifold.

\begin{thm}\label{thm:infty-pencils-ruled}
  Let $X$ be a ruled surface with $\chi(X) = 4-4g < 0$. There exist
  infinitely many pairwise inequivalent Lefschetz pencils $\pi_n: X - B_n \to S^2$, $n \in \Z_{\geq 0}$, of genus $2g$ with fixed number of base points $\lvert B_n\rvert = 4$ and fixed homology class of regular fiber. There exists a symplectic form $\omega$ on $X$ such that the smooth loci of all fibers of $\pi_n$ for all $n \in \Z_{\geq 0}$ are symplectic submanifolds of $(X, \omega)$.
\end{thm}

Blowing down disjoint $(-1)$-sections of inequivalent Lefschetz fibrations produce inequivalent Lefschetz pencils (see e.g. \cite[Corollary 3.10]{baykur-hayano}). The following theorem forms an intermediate step in the proof of Theorem \ref{thm:infty-pencils-ruled}.

\begin{thm}\label{thm:infty-fibrations}
  Let $X$ be a ruled surface with $\chi(X) = 4-4g < 0$. There exist infinitely many pairwise inequivalent Lefschetz fibrations $\pi_n: X \# 4 \overline{\CP^2} \to S^2$, $n \in \Z_{\geq 0}$ of genus $2g$ and fixed homology class of regular fiber. There exists a symplectic form $\omega$ on $X \# 4 \overline{\CP^2}$ such that the smooth loci of all fibers of $\pi_n$ for all $n \in \Z_{\geq 0}$ are symplectic submanifolds of $(X \# 4 \overline{\CP^2}, \omega)$.
\end{thm}
Theorem \ref{thm:infty-fibrations} may be of independent interest, as it gives the first construction of infinitely many pairwise inequivalent Lefschetz fibrations of the same genus on a fixed smooth $4$-manifold. See Section \ref{sec:previous-work} for a summary of previously known constructions of finitely many such Lefschetz fibrations.

\begin{rmk}[On the fiber genus]
An essential point of Question \ref{question:basic} is the existence of inequivalent Lefschetz pencils of the
\emph{same} fiber genus. Otherwise, infinitely many Lefschetz pencils can arise from increasingly large powers of a given line bundle.

The smallest genus of the regular fibers appearing in the fibrations of \Cref{thm:infty-fibrations} is $4$. Currently, the smallest known genus of inequivalent but diffeomorphic Lefschetz fibrations is $3$, by work of Baykur--Hayano \cite[Theorem 6.2]{baykur-hayano2}. On the other hand, there are restrictions on the monodromy of potential genus-$2$ examples; see~\cite[Remark 6.3]{baykur-hayano2}. \Cref{thm:infty-fibrations} answers Question 6.4 of Baykur--Hayano~\cite{baykur-hayano2} positively for the case of (four-fold blowups of) ruled surfaces.
\end{rmk}

\begin{rmk}[The case of surface bundles over surfaces and mapping tori] Our result highlights a difference between Lefschetz fibrations and surface bundles. F.E.A. Johnson~\cite{johnson-surface} showed that if $\pi: M \to \Sigma_h$ is a $\Sigma_g$-bundle over $\Sigma_h$ with $g, h \geq 2$ then there are only finitely many ways to realize $M$ as the total space of a $\Sigma_{g'}$-bundle $M \to \Sigma_{h'}$ with $g', h' \geq 2$, and Salter~\cite{Salter} constructed $4$-manifolds admitting arbitrarily many fiberings. In dimension $3$, the set of surface bundles $Y \to S^1$ for a fixed $3$-manifold $Y$ is governed by the \emph{Thurston norm} \cite{thurston}. Thurston showed, for example, that any irreducible and atoroidal $3$-manifold fibers over $S^1$ with fixed genus in at most finitely many ways. Moreover, the isomorphism class of the fibering is determined by the homology class of the fiber~\cite[Lemma 1]{Tischler}.
\end{rmk}

\subsection{A road map to the proof of \Cref{thm:infty-pencils-ruled,thm:infty-fibrations}}
In this subsection we give a quick overview of each of the steps appearing in the proofs
of \Cref{thm:infty-pencils-ruled,thm:infty-fibrations}. We emphasize that
the proofs of \Cref{thm:infty-pencils-ruled,thm:infty-fibrations} in the
smooth category are essentially finished in \Cref{sec:diffeo}; see Corollaries \ref{cor:smooth-infty-fib} and \ref{cor:smooth-infty-pencil}.

Let $g \geq 2$ and $X = \Sigma_g \times S^2\#4\overline{\CP^2}$. The starting point
in our construction of infinitely many inequivalent Lefschetz fibrations on $X$
is a quite peculiar Lefschetz fibration on $\pi:X \to S^2$ called the (even) Matsumoto--Cadavid--Korkmaz (MCK) Lefschetz fibration~\cite{hamada}.

\medskip \noindent\textbf{The Matsumoto--Cadavid--Korkmaz Lefschetz fibration.}
The MCK fibration is induced by a factorization of an involution $\eta$ on $\Sigma_{2g}$ with two fixed points. Its vanishing cycles and the involution $\eta$ are shown in \Cref{fig:mck-vanishing-cycles-intro}. This Lefschetz fibration was originally
described by Matsumoto~\cite[Proposition 4.2]{matsumoto} for $g=1$ and later
generalized to higher genus by both Cadavid~\cite{cadavid} and Korkmaz~\cite{korkmaz}. The MCK fibration has been the starting point
of many interesting examples of Lefschetz fibrations (e.g. ~\cite{korkmaz,ozbaci-stipsicz}). It is a
hyperelliptic Lefschetz fibration, with the smallest number of singular fibers of a fibration of genus at least $6$ for manifolds with $b_2^+ =1$~\cite{stipsciz-singular,baykur-min}.

\begin{figure}
\centering
\includegraphics[width=0.4\textwidth]{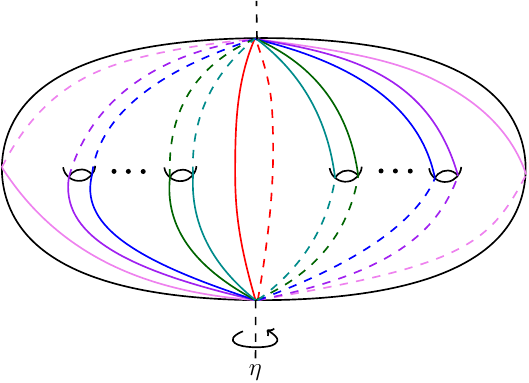}
\caption{Vanishing cycles of the MCK Lefschetz fibration of genus $2g$ and the involution $\eta$.}\label{fig:mck-vanishing-cycles-intro}
\end{figure}

For our purposes, the most important properties of the MCK fibration are
the following:
\begin{enumerate}
\item The total space is diffeomorphic to $\Sigma_g \times S^2\#4\overline{\CP^2}$.
  Furthermore, the total space is invariant under \emph{partial conjugations}
    by elements of the \emph{Torelli group} $\mathcal I_{2g} \trianglelefteq \Mod(\Sigma_{2g})$. (This relies heavily on results of Liu~\cite{liu}; see \Cref{sec:diffeo} for details).
  \item It is a hyperelliptic Lefschetz fibration.
\end{enumerate}
Our construction then proceeds as follows: Starting from the MCK fibration
$\pi:X \to S^2$, we find a suitable mapping class $f \in \mathcal I_{2g}$
and perform partial conjugations by the powers $f^n$ to obtain a
family of fibrations $\{\pi_n:X_n \to S^2\}$. We now briefly describe the partial conjugation construction.

\medskip\noindent \textbf{Navigating the set of Lefschetz fibrations.}
Let $(X,\pi)$ be a Lefschetz fibration of genus $h$.
A big part of the success in the study of Lefschetz fibrations is that
they admit simple combinatorial descriptions via factorizations of the identity in $\Mod(\Sigma_h)$ by right-handed Dehn twists~\cite{matsumoto,kas}. From this perspective the following operation,
called \emph{partial conjugation}~\cite{baykur-inequiv-1,auroux-pc}, is quite natural.
Let $W_1,W_2$ be products of positive Dehn twists such that
\[ W_1\cdot W_2 = 1 \in \Mod(\Sigma_h) .\]
Let $f \in \Mod(\Sigma_h)$ be a mapping class that commutes with $W_2$. Then because $f T_c f^{-1} = T_{f(c)}$ for any right-handed Dehn twist $T_c \in \Mod(\Sigma_h)$, we obtain a new factorization
\[
  W_1\cdot W_2^f = 1 \in \Mod(\Sigma_h)
\]
of the identity in $\Mod(\Sigma_h)$ by right-handed Dehn twists, where $W_2^f$ denotes the conjugation of the factors of $W_2$ by $f$. This new factorization defines a Lefschetz fibration $(X_f,\pi_f)$ on a manifold $X_f$ which has the same signature and Euler characteristic as the original manifold $X$.

\begin{rmk}
Partial conjugation can also be understood from a topological point of view
as a cutting and regluing operation; see \Cref{sec:background} for details. In this way, partial conjugation is a generalization of the fiber-sum operation. However, the Lefschetz fibrations we construct in the proof of
  \Cref{thm:infty-fibrations} are \emph{not minimal}.
  In particular, the partial conjugations that we employ are \emph{not} fiber-sums~\cite{usher}.
\end{rmk}
\begin{rmk}[Effectiveness of partial conjugations]
  Let $G_{W_2}$ be the group generated by the Dehn twists present in the word $W_2$.
  If $f \in G_{W_2}$, then $(X_f,\pi_f)$ is equivalent to $(X,\pi)$~\cite[Lemma 6]{auroux}. Thus, partial conjugation
  is more effective in changing the equivalence class of a fibration when $G_{W_2}$ is ``small.''
\end{rmk}

In order to distinguish the equivalence classes of the fibrations $\pi_n:X_n \to S^2$, we distinguish the monodromy groups of $\pi_n$ up to conjugacy.

\medskip \noindent \textbf{Monodromy group and the Johnson homomorphism.} An invariant
of a Lefschetz fibration is given by (the conjugacy class of) its \emph{monodromy group}; see \Cref{sec:background} for details. Let $G$ be the monodromy group of $(X,\pi)$ and $G_n$ the monodromy
group of $(X_n,\pi_n)$. Let $\mathcal I_{2g} \trianglelefteq \Mod(\Sigma_{2g})$
be the \emph{Torelli group}. If $G$ and $G_n$ are conjugate in $\Mod(\Sigma_{2g})$ then $G \cap \mathcal I_{2g}$ and $G_n\cap \mathcal I_{2g}$ are also conjugate in $\Mod(\Sigma_{2g})$ because $\mathcal I_{2g}$ is normal in $\Mod(\Sigma_{2g})$. Thus, it is enough to distinguish subgroups of $\mathcal I_{2g}$ up to conjugation by $\Mod(\Sigma_{2g})$.
This observation allows us to \emph{abelianize} the problem:
Let $H:=H_1(\Sigma_{2g};\Z)$ and
\[ \tau: \mathcal I_{2g} \to (\wedge^3 H)/H \]
be the \emph{Johnson homomorphism}~\cite[Section 6.6]{farb-margalit}. The Johnson homomorphism $\tau$ is $\Mod(\Sigma_{2g})$-equivariant, with respect to the conjugation action on $\mathcal I_{2g}$ and the induced action $(\wedge^3 H)/H$ via the symplectic representation of $\Mod(\Sigma_{2g})$ acting on $H$. Thus,
it is enough to distinguish $\Mod(\Sigma_{2g})$-orbits in $(\wedge^3 H)/H$, and the hyperellipticity of $(X,\pi)$ vastly simplifies this computation.

\begin{rmk}
  We employed the same idea in our previous work~\cite[Lemma 7.4]{LS-sections}.
  The role of $\mathcal I_{2g}$
  was played by the point-pushing subgroup $\pi_1(\Sigma_g) \trianglelefteq \Mod(\Sigma_{g,1})$, and the role of $\tau$ was played by the abelianization map $\pi_1(\Sigma_g) \to H_1(\Sigma;\Z)$.
\end{rmk}

\medskip
\noindent\textbf{From fibrations to pencils.} Hamada showed that the MCK fibration posseses four disjoint $(-1)$-sections by finding appropiate lifts of the corresponding factorization to $\Mod(\Sigma_{2g}^4)$~\cite{hamada}.
In order to relate \Cref{thm:infty-fibrations} to \Cref{thm:infty-pencils-ruled},
we find an appropriate lift of $f \in \mathcal I_{2g}$ and perform partial conjugations in $\Mod(\Sigma_{2g}^4)$ in order to preserve this number of disjoint $(-1)$-sections. Furthermore, by using two distinct lifts given by Hamada~\cite[Figure 21]{hamada}
the total space of the pencil can be either $\Sigma_g \times S^2$
or $\Sigma_g\tilde{\times} S^2$. The existence of four disjoint $(-1)$-sections
  also improves our understanding of the homology class of the fibers. In
  particular, we can show that all fibers of our pencils are representatives of a fixed
  homology class (\Cref{lem:nu-choice}).

\begin{rmk}
  In our previous work~\cite{LS-sections} we constructed examples, via partial conjugations,
  of infinitely many inequivalent factorizations in $\Mod(\Sigma_{g,1})$
  covering the \emph{same} factorization of the identity in $\Mod(\Sigma_g)$. Yet, these factorizations were \emph{not} associated to pencils: the starting point for the construction was a Lefschetz fibration with a section of self-intersection at most $-2$.
  In fact, our method in ~\cite{LS-sections} fails for sections of self-intersection $-1$ (see \Cref{rmk:prev-work}).
\end{rmk}

\noindent\textbf{Symplectic structures on ruled surfaces.}
To finish the proof of \Cref{thm:infty-pencils-ruled} and \Cref{thm:infty-fibrations},
we leverage the classification of symplectic structures on ruled surfaces of Lalonde--McDuff~\cite{lalonde-mcduff}. In
particular, it is enough to construct cohomologous symplectic forms compatible with the pencils. We achieve this via the Gompf--Thurston construction and by carefully choosing the
gluing maps associated to the respective partial conjugation; see \Cref{sec:symplectic} for details.

\begin{rmk}
  Lin--Wu--Xie--Zhang \cite[Theorem 1.5]{lin-wu-xie-zhang} recently showed the existence of infinitely many isotopy classes of diffeomorphic symplectic forms on ruled surfaces with negative Euler characteristic. While we show that all the symplectic forms we construct on a minimal ruled surface are diffeomorphic, we are unable to tell if they are isotopic.
\end{rmk}

\subsection{Lefschetz fibrations with homeomorphic total spaces}
In \Cref{sec:homeo} we show the broader applicability of our methods. Starting
with a simply-connected hyperelliptic Lefschetz fibration, we easily construct infinitely
many inequivalent but homeomorphic Lefschetz fibrations. We are unable to tell if these families of homeomorphic $4$-manifolds are diffeomorphic. In particular, our construction gives a broader source of potential examples of the
same phenomenon described by \Cref{thm:infty-fibrations}. 

\subsection{Previous work} \label{sec:previous-work}
Despite the great deal of interest in Lefschetz pencils and fibrations, Question \ref{question:basic} remains wide open in its full generality. The first (published) examples of manifolds admitting at least two inequivalent Lefschetz fibrations were given by Park--Yun~\cite{park-yun-1}, constructed using knot surgery operations of Fintushel--Stern~\cite{fintushel-stern} and distinguished by their monodromy groups. (Park--Yun also mentioned that Ivan Smith had an earlier example in his thesis of two inequivalent fibrations on $T^2 \times \Sigma_2\#9\overline{\CP^2}$.) See also further work of Park--Yun~\cite{park-yun-2} generalizing this construction to find arbitrarily large numbers of inequivalent but diffeomorphic fibrations. Later, Baykur--Hayano used monodromy substitutions to construct pairs of
inequivalent Lefschetz pencils on blow-ups of Calabi-Yau K3 surfaces~\cite[Theorem 6.2]{baykur-hayano2} whose fibers are ambiently homeomorphic. Using completely different methods, Baykur showed that any symplectic manifold $X$ admits at least two distinct pencils (up to equivalence and fibered Luttinger surgery) on a blowup of $X$, and arbitrarily many such pencils if $X$ is not rational or ruled~\cite{baykur-inequiv-1, baykur-inequiv-2}. Despite this progress, the methods used to distinguish the fibrations were inherently finite: Park--Yun ~\cite{park-yun-2} relied on a finite graph to distinguish the monodromy groups, and Baykur bounds the number
of inequivalent pencils by a number of a priori chosen points to blow-up~\cite{baykur-inequiv-1}. Further examples of pairs of inequivalent pencils were also given by Hamada~\cite[Theorem 7]{hamada} and Baykur--Hayano--Monden~\cite[Theorem 6.4]{BHM}.

\subsection{Questions} We end this introduction with a set of questions
prompted by our results.

In~\cite[Remark 3.5]{baykur-inequiv-1} Baykur wonders if the MCK fibration of genus $2g$ on
$\Sigma_{g} \times S^2\# 4\overline{\CP^2}$ is unique up to equivalence and fibered Luttinger surgery.
Although we are unable to tell if the partial conjugations we employ can be obtained by a sequence of fibered
Luttinger surgeries, our results show that Luttinger surgery is necessary in a potential uniqueness statement.
More generally, we ask the following refinement to \Cref{question:basic}:
\begin{question}
  Does there exist a $4$-manifold $X$ admitting infinitely many Lefschetz fibrations that are not related via sequences of partial conjugations?
\end{question}

A plausible starting point to finding more families of infinitely many inequivalent but diffeomorphic Lefschetz fibrations may be in determining the diffeomorphism types of the example described in \Cref{sec:homeo}.
\begin{question}
  Are the $4$-manifolds $Z_n$ and $Z_m$ diffeomorphic for all $n, m \in \Z_{\geq 0}$, where $\pi_n: Z_n \to S^2$ is the Lefschetz fibration constructed in Theorem \ref{thm:homeo-not-iso}?
\end{question}

Finally, we point out that we construct infinite families of \emph{non-hyperelliptic} Lefschetz fibrations of \emph{even} genus (starting with genus $4$) in this paper, leaving open the following questions.
\begin{question}
Do there exist infinitely many inequivalent Lefschetz fibrations of odd genus with diffeomorphic total spaces? Of genus $2$?
\end{question}

\begin{question}
Do there exist infinitely many inequivalent hyperelliptic Lefschetz fibrations of fixed genus and diffeomorphic total spaces?
\end{question}

\subsection{Organization of the paper}
In \Cref{sec:background}, we review the necessary background material about Lefschetz fibrations and mapping class groups of surfaces. In \Cref{sec:twisted-mck}, we describe in detail the Matsumoto--Cadavid--Korkmaz (MCK)
Lefschetz fibration $\pi:X \to S^2$ and define its twisted versions $\{\pi_n:X_n\to S^2\}$. In \Cref{sec:johnson} we show how to use
the Johnson homomorphism as a tool to detect inequivalent Lefschetz fibrations. In particular,
we show that the family of fibrations $\{\pi_n:X_n \to S^2\}$ is pairwise inequivalent.
\Cref{sec:diffeo} contains the proof of \Cref{thm:infty-pencils-ruled,thm:infty-fibrations}
in the smooth category. \Cref{sec:symplectic} describes the construction of the
symplectic forms $\omega_n$ compatible with $\pi_n:X_n \to S^2$ and finishes the proof of \Cref{thm:infty-pencils-ruled,thm:infty-fibrations}. In \Cref{sec:homeo}, we describe
a new family of examples of pairwise inequivalent Lefschetz fibrations with
homeomorphic total space. Appendix \ref{appendix-a} contains some routine computations in $\Mod(\Sigma_{2g}^4)$ and
associated mapping class groups. In particular, it contains the proof of \Cref{thm:hamada-sections}.
Appendix \ref{appendix-b} contains the proof of \Cref{prop:isotoping-hateta}, which allows us
to find nice representatives for the gluing maps employed in our partial conjugation construction.

\subsection{Acknowledgments}

We are grateful to our advisor Benson Farb for his constant support and invaluable guidance throughout this project, and
for his extensive comments on an earlier draft of this paper. We would like to thank \.{I}nanç Baykur for many
insightful discussions across wide-ranging topics around Lefschetz fibrations, including pointing us to Auroux's work
\cite{auroux} which spurred the initial key ideas of our current work and for his suggestions to push our preliminary
results on Lefschetz fibrations to constructing infinitely many symplectic Lefschetz pencils. We also thank Baykur for
sharing the statement of \cite[Problem 4.98]{K3} with us. We thank Faye Jackson and Nick Salter for comments on an
earlier draft that improved the exposition. The first-named author would like to thank Francesco Lin for a suggestion to
consider partial conjugations by elements in the Torelli group, Tye Lidman and Lisa Piccirillo for conversations about
diffeomorphism-types of Lefschetz fibrations arising from partial conjugation, and Denis Auroux for a conversation about
Lefschetz fibrations and symplectic surfaces. The second-named author would like to thank Mohammed Abouzaid and Danny Calegari for conversations about symplectic blowing down. 


\section{Lefschetz fibrations and mapping class groups}\label{sec:background}

In this section we recall the relationship between Lefschetz fibrations and mapping class groups and some tools for studying them.

\subsection{Lefschetz pencils, fibrations, and monodromy}

Let $M^4$ be a closed, oriented, smooth $4$-manifold. A \emph{Lefschetz pencil} is a smooth map $\pi: M - B \to S^2$ for some nonempty, finite set $B \subseteq M$ with finitely many critical points $p_1, \dots, p_r \in M - B$ such that
\begin{enumerate}[(a)]
\item for each \emph{base point} $p \in B$, there are a smooth, orientation-compatible chart $U \cong \C^2$ around $p \in M$ and a diffeomorphism $S^2 \cong \CP^1$ with respect to which $\pi$ takes the form
\[
	\pi(z, w) = [z:w] \in \CP^1,
\]
and
\item \label{defn:nodal} for each critical point $p_i$, there are smooth, orientation-compatible charts $U_i \cong \C^2$ around $p_i \in M$ and $V_i \cong \C$ around $\pi(p_i) \in S^2$ with respect to which $\pi$ takes the form
\[
	\pi(z, w) = z^2 + w^2.
\]
\end{enumerate}
A \emph{Lefschetz fibration} is a smooth, surjective map $\pi: M \to S^2$ with finitely many critical points $p_1, \dots, p_r \in M$, each satisfying \ref{defn:nodal} above. We assume that $\pi$ is injective on the set $\{p_1, \dots, p_r\}$ of its critical points and that Lefschetz fibrations are \emph{relatively minimal}, i.e. no fiber of $\pi$ contains an embedded $(-1)$-sphere. The \emph{genus} of a Lefschetz fibration $\pi: M \to S^2$ is the genus of any regular fiber $\pi^{-1}(b)$, $b \in S^2$. Two Lefschetz fibrations $\pi_1: M_1 \to S^2$ and $\pi_2: M_2 \to S^2$ are \emph{equivalent} if there exist diffeomorphisms $\Phi: M_1 \to M_2$ and $\varphi: S^2 \to S^2$ so that the following diagram commutes:
\begin{center}
\begin{tikzcd}
M_1 \arrow[r, "\Phi"] \arrow[d, "\pi_1"] & M_2 \arrow[d, "\pi_2"] \\
S^2 \arrow[r, "\varphi"]                 & S^2
\end{tikzcd}
\end{center}

A choice of a regular value $b \in S^2$ and a choice of a diffeomorphism $\Phi_b: \pi^{-1}(b) \xrightarrow{\sim} \Sigma_g$ determine the \emph{monodromy representation} of a Lefschetz fibration $\pi: M \to S^2$, which is an antihomomorphism of groups
\[
	\rho: \pi_1(S^2 - \{q_1, \dots, q_r\}, b) \to \Mod(\Sigma_g),
\]
where $q_1, \dots, q_r \in S^2$ denote the singular values of $\pi$ (cf. \cite[p. 291]{gompf-stipsicz}). The monodromy representation $\rho$ is characterized by the property that for any loop $\gamma \in \pi_1(S^2 - \{q_1, \dots, q_r\}, b)$ and any $\varphi \in \Diff^+(\Sigma_g)$ that represents the mapping class $\rho(\gamma)$, there is an isomorphism of $\Sigma_g$-bundles over $S^1$
\[
	\Phi: \gamma^* M \to M_\varphi := \Sigma_g \times [0, 1] / ((\varphi(x), 0) \sim (x, 1))
\]
between the pullback of $\pi$ along $\gamma$ and the \emph{mapping torus} $M_\varphi$. Viewing $\gamma$ as a map $[0, 1] \to S^2 - \{q_1, \dots, q_r\}$ with $\gamma(0) = \gamma(1) = b$, the restriction of $\Phi$ to $(\gamma^*\pi)^{-1}(0) = \pi^{-1}(b) \to \Sigma_{g} \times \{0\}$ agrees with the identification $\Phi_b: \pi^{-1}(b) \to \Sigma_g$ chosen in the definition of the monodromy representation.

If $\gamma_i \in \pi_1(S^2 - \{q_1, \dots, q_r\}, b)$ is a loop obtained from a small, counterclockwise loop around $q_i$ connected to $b$ by a path in $S^2 - \{q_1, \dots, q_r\}$, the monodromy $\rho(\gamma_i)$ is a right-handed Dehn twist $T_{\ell_i} \in \Mod(\Sigma_g)$ about a \emph{vanishing cycle} $\ell_i$, an isotopy class of some essential simple closed curve in $\Sigma_g$. Given an additional choice of generators $\gamma_1, \dots, \gamma_r \in \pi_1(S^2 - \{q_1, \dots, q_r\}, b)$ such that $\gamma_1 \gamma_2 \dots \gamma_r = 1$, the monodromy representation determines a \emph{monodromy factorization}, a relation in $\Mod(\Sigma_g)$ of the form
\[
	T_{\ell_r} \dots T_{\ell_1} = 1 \in \Mod(\Sigma_g).
\]

More generally, one can consider a \emph{positive factorization} of any mapping class $h \in \Mod(\Sigma_{g,n}^m)$, i.e. a factorization of $h$ consisting only of right-handed Dehn twists $T_\ell \in \Mod(\Sigma_{g,n}^m)$. Two positive factorizations of $h$ in $\Mod(\Sigma_{g,n}^m)$ are said to be \emph{Hurwitz equivalent} if they are related by a sequence of the following two types of moves:
\begin{enumerate}[(a)]
\item \emph{(Elementary transformation)} For any $1 \leq i \leq r-1$,
\[
	T_{\ell_r} \dots T_{\ell_{i+1}} T_{\ell_i} \dots T_{\ell_1} \leftrightarrow T_{\ell_r} \dots (T_{\ell_{i+1}} T_{\ell_i} T_{\ell_{i+1}}^{-1}) T_{\ell_{i+1}} \dots T_{\ell_1}.
\]
\item \emph{(Global conjugation)} For any $f \in \Mod(\Sigma_{g,n}^m)$ that commutes with $h \in \Mod(\Sigma_{g,n}^m)$
\[
	T_{\ell_r} \dots T_{\ell_1} \leftrightarrow (f T_{\ell_r} f^{-1}) \dots (f T_{\ell_1} f^{-1}).
\]
\end{enumerate}

The following theorem characterizes equivalence classes of Lefschetz fibrations by the Hurwitz equivalence classes of their monodromy factorizations.
\begin{thm}[{Kas \cite{kas}, Matsumoto \cite{matsumoto}}]\label{thm:matsumoto-kas}
Let $g \geq 2$. There is a bijection
\[
\left\{ \parbox{14em}{\centering genus-$g$ Lefschetz fibrations, up to equivalence} \right\} \xrightarrow{\sim} \left\{ \parbox{18em}{\centering positive factorizations of the identity in $\Mod(\Sigma_g)$, up to Hurwitz equivalence}\right\}
\]
given by monodromy factorizations.
\end{thm}
This bijection immediately implies the following criterion for two Lefschetz fibrations to be inequivalent.
\begin{cor}\label{cor:noniso-criterion}
Let $g \geq 2$ and let $\pi_1: M_1 \to S^2$ and $\pi_2: M_2 \to S^2$ be genus-$g$ Lefschetz fibrations with monodromy representations $\rho_1$ and $\rho_2$ respectively. If the images $\mathrm{im}(\rho_1)$ and $\mathrm{im}(\rho_2)$ are not conjugate as subgroups of $\Mod(\Sigma_g)$ then the Lefschetz fibrations $\pi_1$ and $\pi_2$ are inequivalent.
\end{cor}

Adding the data of disjoint sections of Lefschetz fibrations yields more refined information on the monodromy side. Fix $n$-many disjoint sections $\sigma_1, \dots, \sigma_n: S^2 \to M$ of a Lefschetz fibration $\pi: M \to S^2$. By considering the monodromy factorization of $\pi$ with respect to the sections $\sigma_1, \dots, \sigma_n$ in $\Mod(\Sigma_g^n)$ instead (where the $i$th boundary component $\delta_i$ corresponds to the boundary of a neighborhood of the $i$th section $\sigma_i(b)$ in $\pi^{-1}(b)$), we obtain a lift
\[
	T_{\tilde{\ell}_r}\dots T_{\tilde{\ell}_1} = T_{\delta_1}^{a_1} \dots T_{\delta_n}^{a_n} \in \Mod(\Sigma_g^n)
\]
of the monodromy factorization $T_{\ell_r} \dots T_{\ell_1} = 1 \in \Mod(\Sigma_g)$ of $\pi$ under the cappping and forgetful homomorphism $\Mod(\Sigma_g^n) \to \Mod(\Sigma_g)$. For all $1 \leq i \leq r$, the essential simple closed curve $\tilde \ell_i \subseteq \Sigma_g^n$ is isotopic to $\ell_i$ in $\Sigma_g$ (after capping off the boundary components) and $a_i \in \N$ satisfies $-a_i = [\sigma_i(S^2)]^2$ where $[\sigma_i(S^2)]^2$ denotes the self-intersection of $\sigma_i(S^2)$ in $M$ \cite[Lemma 2.3]{smith}. By considering only the sections $\sigma_1, \dots, \sigma_n$ themselves and not their normal neighborhoods in $M$, we obtain a similar lift
\[
	T_{\hat \ell_r} \dots T_{\hat \ell_1} = 1 \in \Mod(\Sigma_{g,n})
\]
of the monodromy factorization $T_{\ell_r} \dots T_{\ell_1} = 1 \in \Mod(\Sigma_g)$ under the forgetful map $\Mod(\Sigma_{g,n}) \to \Mod(\Sigma_g)$. Conversely, any lifts of the monodromy factorization of $\pi$ in $\Mod(\Sigma_g)$ to $\Mod(\Sigma_{g}^n)$ or $\Mod(\Sigma_{g,n})$ as above give rise to $n$-many disjoint sections $\sigma_1, \dots, \sigma_n: S^2 \to M$ of $\pi: M \to S^2$ of self-intersections $-a_1, \dots, -a_n$ respectively.

By blowing up the set $B$ (with $n = \lvert B\rvert$) of base points of any Lefschetz pencil $\pi: M - B \to S^2$, we obtain an associated Lefschetz fibration $\pi: M \# n \overline{\CP^2} \to S^2$ with $n$-many disjoint sections $\sigma_1, \dots, \sigma_n$ of self-intersection $-1$. Conversely, taking a Lefschetz fibration $\pi: M \to S^2$ and blowing down any set of $n$-many $(-1)$-sections $\sigma_1, \dots, \sigma_n$ yields a Lefschetz pencil structure on $N$ (where $M \to N$ is this blow down map so that $M \cong N \# n \overline{\CP^2}$). Two Lefschetz pencils $\pi_1: M_1 - B_1 \to S^2$ and $\pi_2: M_2 - B_2 \to S^2$ are \emph{equivalent} if there exist diffeomorphisms $\Phi: M_1 \to M_2$ and $\varphi: S^2 \to S^2$ so that $\Phi(B_1) = B_2$ and the following diagram commutes:
\begin{center}
\begin{tikzcd}
M_1-B_1 \arrow[r, "\Phi"] \arrow[d, "\pi_1"] & M_2-B_2 \arrow[d, "\pi_2"] \\
S^2 \arrow[r, "\varphi"]                 & S^2
\end{tikzcd}
\end{center}
One can detect when two Lefschetz pencils are inequivalent by considering their associated Lefschetz fibrations:
\begin{prop}[{cf. Baykur--Hayano \cite[Corollary 3.10]{baykur-hayano}}]\label{prop:noniso-pencils}
Let $\pi_1: M_1 - B_1 \to S^2$ and $\pi_2: M_2 - B_2 \to S^2$ be Lefschetz pencils with $n = \lvert B_1 \rvert = \lvert B_2 \rvert$. If the Lefschetz fibrations $M_1 \# n \overline{\CP^2} \to S^2$ and $M_2 \# n \overline{\CP^2} \to S^2$ associated to the Lefschetz pencils $\pi_1$ and $\pi_2$ respectively are inequivalent then the Lefschetz pencils $\pi_1$ and $\pi_2$ are inequivalent.
\end{prop}
\subsection{Partial conjugations}\label{sec:partial-conj}

One way to obtain new positive factorizations of the identity in $\Mod(\Sigma_{g,n})$ from old ones is via \emph{partial conjugation}. Given some positive factorization
\begin{equation}\label{eqn:original-sections}
	T_{\tilde \ell_r} \dots T_{\tilde \ell_{i+1}} T_{\tilde \ell_i} \dots T_{\tilde\ell_1} = 1 \in \Mod(\Sigma_{g,n})
\end{equation}
and any $\tilde f \in \Mod(\Sigma_{g,n})$ that commutes with the mapping class $T_{\tilde \ell_i} \dots T_{\tilde\ell_1}$ in $\Mod(\Sigma_{g,n})$, we form a new positive factorization
\begin{equation}\label{eqn:twisted-sections}
	T_{\tilde \ell_r} \dots T_{\tilde\ell_{i+1}} (\tilde f T_{\tilde \ell_i} \tilde f^{-1}) \dots (\tilde f T_{\tilde \ell_1} \tilde f^{-1}) = 1 \in \Mod(\Sigma_{g,n}).
\end{equation}

We now describe this operation in terms of Lefschetz fibrations and sections. Let $\pi: M \to S^2$ denote the Lefschetz fibration with sections $\sigma_1, \dots, \sigma_n: S^2 \to M$ corresponding to monodromy factorization (\ref{eqn:original-sections}). Let $b \in S^2$ denote the chosen regular value of $\pi$ and let $q_1, \dots, q_r \in S^2$ denote the singular values of $\pi$ so that $q_i$ corresponds to the vanishing cycle $\tilde \ell_i$ for all $1 \leq i \leq r$. We may then write the base $S^2$ as a union of two disks $S^2 = D_1 \cup_{\partial} D_2$ with $D_i \cong D^2$, where
\begin{enumerate}[(a)]
\item the base point $b \in S^2$ is contained in $\partial D_1 = \partial D_2$,
\item the chosen generators $\gamma_1, \dots, \gamma_i$ and singular values $q_1, \dots, q_i$ are contained in $D_1$ and the chosen generators $\gamma_{i+1}, \dots, \gamma_r$ and singular values $q_{i+1}, \dots, q_r$ are contained in $D_2$, and
\item \label{setup-monodromy} the based loop $(\partial D_1, b)$ can be oriented so that as an element of $\pi_1(S^2 - \{q_1, \dots, q_r\}, b)$,
\[
	[\partial D_1] = \gamma_1 \dots \gamma_i.
\]
\end{enumerate}
Let
\[
	\rho: \pi_1(S^2 - \{q_1, \dots, q_r\}, b) \to \Mod(\Sigma_{g, n})
\]
denote the monodromy representation corresponding to $\pi$ and $\sigma_1, \dots, \sigma_n$. Then \ref{setup-monodromy} implies that
\[
	\rho([\partial D_1]) = T_{\tilde\ell_i} \dots T_{\tilde\ell_1} \in \Mod(\Sigma_{g,n}).
\]
In other words, if $\tilde \eta \in \Diff^+(\Sigma_{g,n})$ is any representative of $T_{\tilde\ell_i} \dots T_{\tilde\ell_1} \in \Mod(\Sigma_{g,n})$ then there is an isomorphism of $\Sigma_g$-bundles over $S^1$
\[
	\pi^{-1}(\partial D_1) \to M_{\tilde\eta} = (\Sigma_g \times [0, 1])/((\tilde\eta(x), 0) \sim (x,1))
\]
that sends the section $\sigma_i$ to the section $s_i: S^1 \to M_{\tilde\eta}$ defined by the marked point $p_i$ of $\Sigma_{g,n}$ fixed by $\tilde\eta$.

Let $\varphi \in \Diff^+(\Sigma_{g,n})$ be any representative of $\tilde f \in \Mod(\Sigma_{g,n})$. Because $\tilde f$ and $T_{\tilde\ell_i} \dots T_{\tilde\ell_1}$ commute in $\Mod(\Sigma_{g,n})$, there is an isotopy $\varphi_t: \Sigma_{g,n} \times [0, 1] \to \Sigma_{g,n}$ with $\varphi_0 = \varphi$ and $\varphi_1 = \tilde\eta^{-1} \circ \varphi \circ \tilde\eta$. This isotopy induces an isomorphism $\mathcal F: M_{\tilde\eta} \to M_{\tilde\eta}$ of $\Sigma_g$-bundles
\[
	\mathcal F: (x, t) \mapsto (\varphi_t(x), t)
\]
that fixes each section $s_i: S^1 \to M_{\tilde\eta}$ for all $1 \leq i \leq n$. Let $\mathcal F$ also denote the corresponding map $\pi^{-1}(\partial D_1) \to \pi^{-1}(\partial D_2)$ under the identifications $\pi^{-1}(\partial D_i) \to M_{\tilde\eta}$. Then $\pi$ induces a new Lefschetz fibration
\begin{equation}
\pi^{-1}(D_1) \cup_{\mathcal F: \pi^{-1}(\partial D_1) \to \pi^{-1}(\partial D_2)} \pi^{-1}(D_2) \to S^2
\end{equation}
with $n$-many sections induced by the sections $\sigma_1, \dots, \sigma_n$ of $\pi$. This resulting Lefschetz fibration and sections have monodromy factorization (\ref{eqn:twisted-sections}) with respect to the original identification
\[
	\Phi_b: (\pi^{-1}(b), \sigma_1(b), \dots, \sigma_n(b)) \to \Sigma_{g,n},
\]
viewing $b$ as a point in $\partial D_2$.

\subsection{The Torelli group and the Johnson homomorphism}\label{sec:torelli-johnson-background}

Let $g \geq 2$. The \emph{Torelli group} $\mathcal I_g \leq \Mod(\Sigma_g)$ is the kernel of the symplectic representation
\[
	\Mod(\Sigma_g) \to \Aut(H_1(\Sigma_g; \Z)) \cong \Sp(2g, \Z).
\]
We will study the Torelli group via a surjective group homomorphism
\[
	\tau: \mathcal I_g \to \left(\wedge^3 H_1(\Sigma_g)\right)/H_1(\Sigma_g)
\]
called the \emph{Johnson homomorphism} \cite{johnson} (also see \cite[Section 6.6]{farb-margalit}). Here, the inclusion $H_1(\Sigma_g) \hookrightarrow \wedge^3 H_1(\Sigma_g)$ is given by
\[
	c \mapsto \left(\sum_{i=1}^g a_i \wedge b_i\right) \wedge c
\]
where $a_1, b_1, \dots, a_g, b_g \in H_1(\Sigma_g)$ is any standard symplectic basis.

One useful property of $\tau$ is its \emph{naturality property}: for any $h \in \Mod(\Sigma_g)$ and any $f \in \mathcal I_g$,
\begin{equation}\label{eqn:naturality}
	\tau(hfh^{-1}) = h_*(\tau(f))
\end{equation}
where on the right hand side, $\Mod(\Sigma_g)$ acts on $ \left(\wedge^3 H_1(\Sigma_g)\right)/H_1(\Sigma_g)$ via the symplectic representation \cite[Equation (6.1)]{farb-margalit}. We record one easy but important consequence of naturality below:

\begin{cor}\label{cor:hyperelliptic-kernel}
Fix a hyperelliptic involution $\iota \in \Mod(\Sigma_g)$ and let $\SMod(\Sigma_g)$ denote the hyperelliptic mapping class group of $\Sigma_g$, i.e. the centralizer of $\iota$ in $\Mod(\Sigma_g)$. There is an inclusion
\[
	\SMod(\Sigma_g) \cap \mathcal I_g \leq \ker(\tau).
\]
\end{cor}
This corollary follows from naturality and the fact that $\iota$ induces the negation map on a torsion-free abelian group $\wedge^3 H_1(\Sigma_g) / H_1(\Sigma_g)$.

\section{Partial conjugations of the Matsumoto--Cadavid--Korkmaz Lefschetz fibration}\label{sec:twisted-mck}

In this section we construct a genus-$2g$ Lefschetz fibration $\pi_n: X_n \to S^2$ for each $n \in \Z_{\geq 0}$ and for all $g \geq 2$ by considering partial conjugations of the Matsumoto--Cadavid--Korkmaz (MCK) Lefschetz fibration. We will prove later that the fibrations $\pi_n$ are pairwise inequivalent (Section \ref{sec:johnson}), that the $4$-manifolds $X_n$ are pairwise diffeomorphic (Section \ref{sec:diffeo}), and pairwise symplectomorphic (Section \ref{sec:symplectic}).

To define the MCK Lefschetz fibration of genus $2g$, first consider the isotopy classes of curves shown in Figure \ref{fig:mck-vanishing-cycles}.
Dehn twists about these curves form a factorization of an important involution $[\eta] \in \Mod(\Sigma_{2g})$, where $\eta \in \Diff^+(\Sigma_{2g})$ denotes the order-$2$ diffeomorphism depicted in Figure \ref{fig:mck-vanishing-cycles}.

\begin{figure}
\centering
\includegraphics[width=0.4\textwidth]{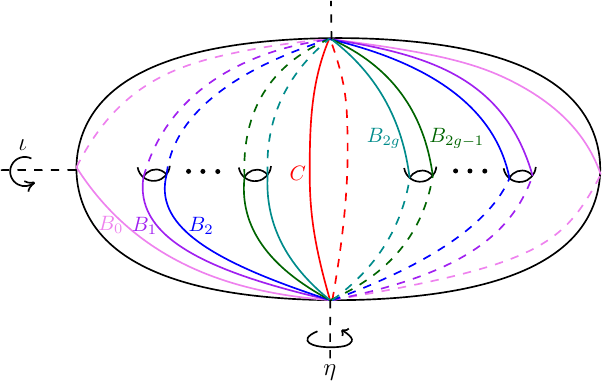}
\caption{Vanishing cycles of the MCK Lefschetz fibration of genus $2g$, the involution $\eta$, and a hyperelliptic involution $\iota$.}\label{fig:mck-vanishing-cycles}
\end{figure}

\begin{thm}[{Cadavid \cite[Theorem 5.1.1]{cadavid}, Korkmaz \cite[Theorem 3.4]{korkmaz}}]\label{thm:eta-factorization}
Let
\[
	h := T_{B_0} T_{B_1} \dots T_{B_{2g}} T_C \in \Mod(\Sigma_{2g}).
\]
Then $h = [\eta]$. In particular, $h^2 = 1 \in \Mod(\Sigma_{2g})$.
\end{thm}
With Lemma \ref{thm:eta-factorization} in hand, we are ready to define the MCK Lefschetz fibration of genus $2g$.
\begin{defn}\label{defn:mck}
The \emph{Matsumoto--Cadavid--Korkmaz Lefschetz fibration} (or \emph{MCK Lefschetz fibration}) of genus $2g$ is the Lefschetz fibration $\pi_0: X_0 \to S^2$ with monodromy factorization
\[
	T_{B_0} T_{B_1} \dots T_{B_{2g}} T_C T_{B_0} T_{B_1} \dots T_{B_{2g}} T_C = 1 \in \Mod(\Sigma_{2g}).
\]
\end{defn}

\begin{figure}
\includegraphics[width=0.9\textwidth]{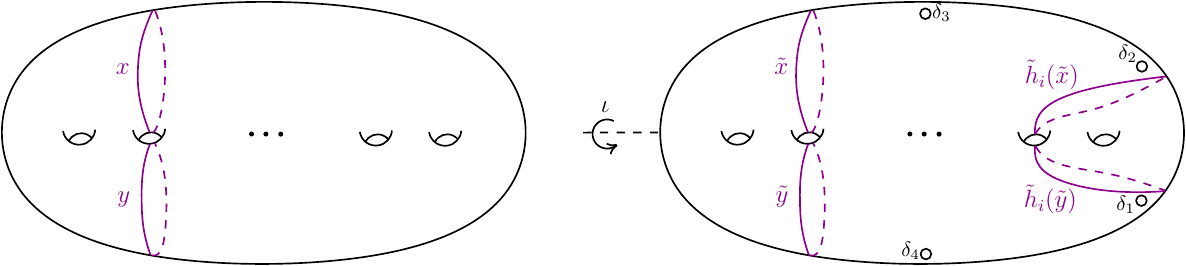}
\caption{Left: Curves used to define $f \in \mathcal I_{2g}$; Right: Lifts $\tilde x, \tilde y$ of $x, y$ to $\Sigma_{2g}^4$ and their images under $\tilde h_i \in \Mod(\Sigma_{2g}^4)$ defined in Lemma \ref{lem:bounding-pair-sections} for both $i = 1, 2$ (cf. Lemma \ref{lem:tilde-f-lemma}). A hyperelliptic involution $\iota$ acting on $\Sigma_{2g}^4$ permuting the boundary components.}\label{fig:bounding-pair}
\end{figure}

Hamada \cite{hamada} constructed multiple sets of four disjoint $(-1)$-sections of the MCK Lefschetz fibration. In this paper, we will consider two sets of sections given in \cite{hamada}.
\begin{figure}
\includegraphics[width=0.9\textwidth]{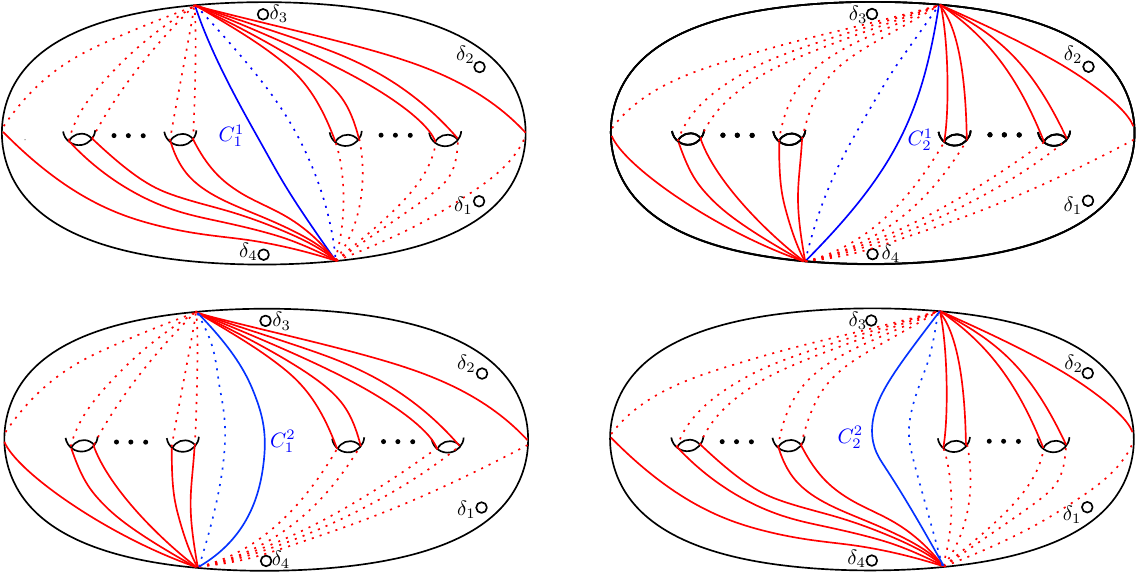}
\caption{Lifts of the vanishing cycles of the MCK Lefschetz fibration of genus $2g$ to $\Sigma_{2g}^4$ found by Hamada \cite{hamada}. The four boundary components of $\Sigma_{2g}^4$ are denoted by $\delta_1, \delta_2, \delta_3, \delta_4$. Top left: $B_{0, 1}^1, \dots, B_{2g,1}^1, C_1^1$; Top right: $B_{0, 2}^1, \dots, B_{2g,2}^1, C_2^1$; Bottom left: $B_{0,1}^2, \dots, B_{2g,1}^2, C_1^2$; Bottom right: $B_{0,2}^2, \dots, B_{2g,2}^2, C_2^2$}\label{fig:MCK-lifts}
\end{figure}
\begin{thm}[{Hamada \cite[Sections 3.4.1 -- 3.4.2]{hamada}}]\label{thm:hamada-sections}
For $i = 1, 2$, let $B_{0,1}^i, B_{1,1}^i, \dots, B_{2g, 1}^i, C_1^i$ and $B_{0,2}^i, B_{1,2}^i, \dots, B_{2g,2}^i, C_2^i$ be the isotopy classes of curves in $\Sigma_{2g}^4$ as shown in Figure \ref{fig:MCK-lifts}. Then
\[
	T_{B_{0,1}^i} T_{B_{1,1}^i} \dots T_{B_{2g,1}^i} T_{C_1^i}  T_{B_{0,2}^i} T_{B_{1,2}^i} \dots T_{B_{2g,2}^i} T_{C_2^i} = T_{\delta_1} T_{\delta_2} T_{\delta_3} T_{\delta_4} \in \Mod(\Sigma_{2g}^4).
\]
\end{thm}
For both completeness and convenience of the reader, we prove Theorem \ref{thm:hamada-sections} in Appendix \ref{appendix-a} and deduce Theorem \ref{thm:eta-factorization} from it. The following lemma constructs a positive factorization of $T_{\delta_1} T_{\delta_2} T_{\delta_3} T_{\delta_4} \in \Mod(\Sigma_{2g}^4)$ for each $n \in \Z$ via partial conjugation and Theorem \ref{thm:hamada-sections}.

\begin{lem}\label{lem:bounding-pair-sections}
Let $B_{0,1}^i, B_{1,1}^i, \dots, B_{2g, 1}^i, C_1^i$ and $B_{0,2}^i, B_{1,2}^i, \dots, B_{2g,2}^i, C_2^i$ for $i = 1, 2$ be the isotopy classes of curves in $\Sigma_{2g}^4$ as shown in Figure \ref{fig:MCK-lifts} and let $\tilde x, \tilde y$ be the isotopy classes of curves in $\Sigma_{2g}^4$ as shown in the right side of Figure \ref{fig:bounding-pair}. Let
\[
	\tilde h_i := T_{B_{0,2}^i} T_{B_{1, 2}^i} \dots T_{B_{2g,2}^i} T_{C_2^i} \in \Mod(\Sigma_{2g}^4).
\]
Then
\[
	\tilde f := T_{\tilde x} T_{\tilde y}^{-1} T_{\tilde h_1(\tilde x)} T_{\tilde h_1(\tilde y)}^{-1} = T_{\tilde x} T_{\tilde y}^{-1} T_{\tilde h_2(\tilde x)} T_{\tilde h_2(\tilde y)}^{-1} \in \Mod(\Sigma_{2g}^4)
\]
and for any $n \in \Z$ and for $i = 1, 2$,
\[
	T_{B_{0,1}^i} T_{B_{1,1}^i} \dots T_{B_{2g, 1}^i} T_{C_1^i} T_{\tilde f^n(B_{0,2}^i)} T_{\tilde f^n(B_{1, 2}^i)} \dots T_{\tilde f^n(B_{2g,2}^i)} T_{\tilde f^n(C_2^i)}  = T_{\delta_1} T_{\delta_2} T_{\delta_3} T_{\delta_4} \in \Mod(\Sigma_{2g}^4).
\]
\end{lem}
\begin{proof}
First, note that $T_{\tilde h_1(\tilde x)} = T_{\tilde h_2(\tilde x)}$ and $T_{\tilde h_1(\tilde y)} = T_{\tilde h_2(\tilde y)}$ in $\Mod(\Sigma_{2g}^4)$ by Lemma \ref{lem:tilde-f-lemma}\ref{lem:tilde-f-lemma-disjoint}, showing the first desired equality. Furthermore,
\[
	\tilde f = T_{\tilde x} T_{\tilde y}^{-1} T_{\tilde h_i(\tilde x)} T_{\tilde h_i(\tilde y)}^{-1}  = T_{\tilde h_i(\tilde x)} T_{\tilde h_i(\tilde y)}^{-1}T_{\tilde x} T_{\tilde y}^{-1}
\]
where the second equality follows by Lemma \ref{lem:tilde-f-lemma}\ref{lem:tilde-f-lemma-disjoint}. Now compute that
\[
	\tilde f \tilde h_i = T_{\tilde h_i(\tilde x)} T_{\tilde h_i(\tilde y)}^{-1}T_{\tilde x} T_{\tilde y}^{-1} \tilde h_i  = \tilde h_i T_{\tilde x} T_{\tilde y}^{-1} \tilde h_i^{-1}T_{\tilde x} T_{\tilde y}^{-1} \tilde h_i = \tilde h_i T_{\tilde x} T_{\tilde y}^{-1} T_{\tilde h_i^{-1}(\tilde x)} T_{\tilde h_i^{-1}(\tilde y)}^{-1} = \tilde h_i \tilde f,
\]
where the last equality follows from Lemma \ref{lem:tilde-f-lemma}\ref{lem:tilde-f-lemma-double}. Finally, the lemma follows from Theorem \ref{thm:hamada-sections} and the fact that $\tilde h_j$ and $\tilde f$ commute.
\end{proof}

By applying the capping and forgetful homomorphisms $\Mod(\Sigma_{2g}^4) \to \Mod(\Sigma_{2g})$, we now rephrase Lemma \ref{lem:bounding-pair-sections} in terms of $\Mod(\Sigma_{2g})$. Consider the curves $x, y \subseteq \Sigma_{2g}$ shown in Figure \ref{fig:bounding-pair} and define
\[
	f :=  T_x T_y^{-1} T_{h(x)} T_{h(y)}^{-1} \in \mathcal I_{2g},
\]
where we note that $f$ is contained in the Torelli group $\mathcal I_{2g}$ because $f$ is a composition of two bounding pair maps $T_x T_y^{-1}$ and $T_{h(x)} T_{h(y)}^{-1}$. The following corollary is then an immediate consequence of Lemma \ref{lem:bounding-pair-sections}.
\begin{cor}\label{cor:partial-MCKs}
Let $g \geq 2$. For any $n \in \Z$,
\[
	T_{B_0} T_{B_1} \dots T_{B_{2g}} T_C T_{f^n(B_0)} T_{f^n(B_1)} \dots T_{f^n(B_{2g})} T_{f^n(C)} = 1 \in \Mod(\Sigma_{2g}).
\]
\end{cor}

Finally, we define the Lefschetz fibrations of interest using Corollary \ref{cor:partial-MCKs}.
\begin{defn}
Let $g \geq 2$. For any $n \in \Z_{\geq 0}$, let $\pi_n: X_n \to S^2$ denote the Lefschetz fibration of genus $2g$ with monodromy factorization
\[
	T_{B_0} T_{B_1} \dots T_{B_{2g}} T_C T_{f^n(B_0)} T_{f^n(B_1)} \dots T_{f^n(B_{2g})} T_{f^n(C)} = 1 \in \Mod(\Sigma_{2g}).
\]
\end{defn}


\section{Distinguishing Lefschetz fibrations via the Johnson homomorphism}\label{sec:johnson}

The goal of this section is to prove the following theorem.
\begin{thm}\label{thm:non-isomorphic}
For any $n \neq m \in \Z_{\geq 0}$, the Lefschetz fibrations $\pi_n: X_n \to S^2$ and $\pi_m: X_m \to S^2$ are inequivalent.
\end{thm}

The main idea of the proof is to apply Corollary \ref{cor:noniso-criterion} by studying the images of the monodromy representations of $\pi_n$, intersected with the Torelli group $\mathcal I_{2g}$. As such, consider the following groups for any $n \in \Z_{\geq 0}$.
\begin{align*}
G_n &:= \langle T_{B_0}, T_{B_1}, \dots, T_{B_{2g}}, T_C, T_{f^n(B_0)}, T_{f^n(B_1)}, \dots, T_{f^n(B_{2g})}, T_{f^n(C)} \rangle, \\
G_n^{\mathcal I} &:= G_n\cap \mathcal I_{2g}, \\
A_n &:= \langle [T_{B_i}^{-1}, f^n], [T_C^{-1}, f^n] : 0 \leq i \leq 2g \rangle.
\end{align*}
By construction, $G_n$ is the image of the monodromy representation of $\pi_n: X_n \to S^2$. We also point out that $A_n$ is a subgroup of $G_n^{\mathcal I}$. To see this, write for any curve $c = C$ or $B_i$ for some $0 \leq i \leq 2g$ that
\[
	T_c^{-1} T_{f^n(c)}= [T_c^{-1}, f^n] = (T_c^{-1} f^n T_c) f^{-n} .
\]
The first equality shows that each generator $[T_{c}^{-1}, f^n]$ of $A_n$ is an element of $G_n$. The second equality shows that each $[T_{c}^{-1}, f^n]$ is an element of $\mathcal I_{2g}$ because both $f^{-n}$ and $T_c^{-1} f^n T_c$ are elements of $\mathcal I_{2g}$.

In what follows, we study the image $\tau(G_n^{\mathcal I})$ of $G_n^{\mathcal I}$ under the Johnson homomorphism (cf. Section \ref{sec:torelli-johnson-background})
\[
	\tau: \mathcal I_{2g} \to \left(\wedge^3 H\right)/H, \qquad H := H_1(\Sigma_{2g}).
\]
To do so, we first describe $G_n^{\mathcal I}$ in terms of $G_0^{\mathcal I}$ and $A_n$.
\begin{lem}\label{lem:GnI-generators}
For any $n \in \Z_{\geq 0}$, there is an equality of subgroups of $\mathcal I_{2g}$
\[
	G_n^{\mathcal I} = \langle G_0^{\mathcal I}, \, k A_n k^{-1} : k \in G_0 \rangle.
\]
\end{lem}
\begin{proof}
Recall that $A_n \leq G_n^{\mathcal I}$ and that $G_0 \leq G_n$. Therefore, it suffices to show the inclusion
\[
	G_n^{\mathcal I} \leq \langle G_0^{\mathcal I}, \, k A_n k^{-1} : k \in G_0 \rangle.
\]

By construction, there is an equality of subgroups of $\Mod(\Sigma_{2g})$
\[
	G_n = \langle G_0, \,A_n \rangle.
\]
Take any $k \in G_n^{\mathcal I} \leq G_n$ and write
\[
	k = \ell_1 m_1 \ell_2 m_2 \dots \ell_s m_s
\]
for some $\ell_i \in G_0, m_i \in A_n$ for $1 \leq i \leq s$. Applying the identity $ab = (aba^{-1})a$ repeatedly to the given factorization of $k$, write
\[
	k = \left(\prod_{i=1}^s r_i m_i r_i^{-1}\right)r
\]
for some $r \in G_0$ and $r_1, \dots, r_s \in G_0$. Because each $m_i \in A_n$ is contained in $\mathcal I_{2g}$ and because $k$ is contained in $\mathcal I_g$, the mapping class $r$ is contained in $\mathcal I_{2g}$, and hence in $G_0^{\mathcal I}$.
\end{proof}

The following lemma gives a first description of the image $\tau(G_n^{\mathcal I})$ using Lemma \ref{lem:GnI-generators}. Another key tool is the naturality (\ref{eqn:naturality}) of the Johnson homomorphism in conjunction with fact that $G_0^{\mathcal I}$ is hyperelliptic.
\begin{lem}\label{lem:GnI-divisibility}
For any $n \in \Z_{\geq 0}$, there is an inclusion of subgroups
\[
	\tau(G_n^{\mathcal I}) \leq n\left(\wedge^3 H\right)/H.
\]
\end{lem}
\begin{proof}
For any $\ell \in G_n^{\mathcal I}$, write using Lemma \ref{lem:GnI-generators}
\[
	\ell = \ell_1 (k_1 h_1 k_1^{-1}) \ell_2 (k_2 h_2 k_2^{-1}) \dots \ell_s (k_s h_s k_s^{-1})
\]
for some $\ell_i \in G_0^{\mathcal I}$, $k_i \in G_0$, and $h_i \in A_n$ for $1 \leq i \leq s$. Applying $\tau$ to both sides (and applying naturality (\ref{eqn:naturality})) yields
\[
	\tau(\ell) = \left(\sum_{i=1}^s \tau(\ell_i) \right) + \left(\sum_{i=1}^s k_i \cdot \tau(h_i) \right).
\]
On the other hand, recall that each generator of $G_0$ is hyperelliptic (see \Cref{fig:mck-vanishing-cycles}), and hence $G_0^{\mathcal I}$ is contained in the hyperelliptic mapping class group $\SMod(\Sigma_{2g})$. By Corollary \ref{cor:hyperelliptic-kernel}, $\tau(\ell_i) = 0$ for all $1 \leq i \leq s$.

Because $h_i$ is an element of $A_n$ for all $1 \leq i \leq s$ and because each $k_i$ preserves the subgroup $n \left(\wedge^3 H\right)/H$, it now suffices to show that $\tau(A_n)$ is contained in $n \left(\wedge^3 H\right)/H$. To see this, note that for any $k \in \Mod(\Sigma_{2g})$
\[
	\tau([k^{-1}, f^n]) = \tau(k^{-1} f^n k) - \tau(f^n) = n\left(\tau(k^{-1} f k) - \tau(f)\right) \in n \left(\wedge^3 H\right)/H. \qedhere
\]
\end{proof}

\begin{figure}
\includegraphics[width=0.4\textwidth]{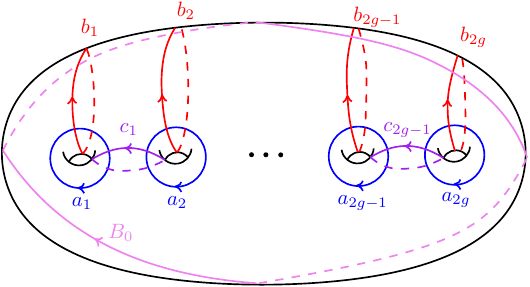}
\caption{Some homology classes in $H$}\label{fig:homology-curves}
\end{figure}

In order to further analyze the image $\tau(G_n^{\mathcal I})$, we need to verify that certain elements of $(\wedge^3 H)/H$ are nonzero and primitive. We record the following algebraic lemma for this purpose.
\begin{lem}\label{lem:nonzero-wedge}
Let $k \geq 2$ and let $\alpha_1, \beta_1, \dots, \alpha_k, \beta_k$ be a symplectic $\Z$-basis of $H_1(\Sigma_k)$. Let $\gamma_{2i-1} := \alpha_i$ and $\gamma_{2i} := \beta_i$ for all $1 \leq i \leq k$. The set
\[
	\{ \gamma_i \wedge \gamma_j \wedge \gamma_\ell : 1 \leq i < j < \ell \leq 2k, \, (i,j, \ell) \neq (i, 2k-1, 2k), \, (1, 2, 2k-1), \, (1, 2, 2k) \}
\]
forms a $\Z$-basis of the free abelian group $\left(\wedge^3 H_1(\Sigma_k)\right)/H_1(\Sigma_k)$.
\end{lem}
\begin{proof}
Let
\[
	\omega := \sum_{i=1}^k \alpha_i \wedge \beta_i = \sum_{i=1}^k \gamma_{2i-1} \wedge \gamma_{2i} \in \wedge^2 H_1(\Sigma_k).
\]
The wedge $\omega$ is independent of the choice of symplectic basis and so the inclusion $H_1(\Sigma_k) \hookrightarrow \wedge^3 H_1(\Sigma_k)$ is given by
\[
	c \mapsto \omega \wedge c.
\]
The group $\wedge^3 H_1(\Sigma_k)$ is torsion-free and has a $\Z$-basis
\[
	\{ \gamma_i \wedge \gamma_j \wedge \gamma_\ell : 1 \leq i < j < \ell \leq 2k \}.
\]
The set
\[
	\{ \gamma_i \wedge \gamma_j \wedge \gamma_\ell : 1 \leq i < j < \ell \leq 2k, \, (i,j, \ell) \neq (i, 2k-1, 2k), \, (1, 2, 2k-1), \, (1, 2, 2k) \} \cup \{\omega \wedge \gamma_i : 1 \leq i \leq 2k\}
\]
is another $\Z$-basis of $\wedge^3 H_1(\Sigma_k)$. To see this, note that the size of the new basis agrees with that of the old and that each element of the old basis can be expressed in terms of this new basis because
\begin{align*}
\gamma_i \wedge \gamma_{2k-1} \wedge \gamma_{2k} &= \omega \wedge \gamma_i - \sum_{j = 1}^{k-1} \gamma_{2j-1} \wedge \gamma_{2j} \wedge \gamma_i \qquad \text{ for }1 \leq i \leq 2k-2, \\
\gamma_1 \wedge \gamma_2 \wedge \gamma_{2k-1} &= \omega \wedge \gamma_{2k-1} - \sum_{j = 2}^{k-1} \gamma_{2j-1} \wedge \gamma_{2j} \wedge \gamma_{2k-1}, \\
\gamma_1 \wedge \gamma_2 \wedge \gamma_{2k} &= \omega \wedge \gamma_{2k} - \sum_{j = 2}^{k-1} \gamma_{2j-1} \wedge \gamma_{2j} \wedge \gamma_{2k}.
\end{align*}
Because $\{\omega \wedge \gamma_i : 1 \leq i \leq 2k\}$ is a $\Z$-basis of $H_1(\Sigma_k) \leq \wedge^3 H_1(\Sigma_k)$, the quotient $\left(\wedge^3 H_1(\Sigma_k)\right) / H_1(\Sigma_k)$ is torsion-free. Furthermore, the $\Z$-span of
\[
	\{ \gamma_i \wedge \gamma_j \wedge \gamma_\ell : 1 \leq i < j < \ell \leq 2k, \, (i,j, \ell) \neq (i, 2k-1, 2k), \, (1, 2, 2k-1), \, (1, 2, 2k) \}
\]
maps isomorphically onto the quotient $\wedge^3 H_1(\Sigma_k) \to \left(\wedge^3 H_1(\Sigma_k)\right) / H_1(\Sigma_k)$ and hence forms a $\Z$-basis of the quotient.
\end{proof}

The following lemma refines our description of $\tau(G_n^{\mathcal I})$ through a more detailed calculation of the Johnson homomorphism. In the proof, we use the notation of Figure \ref{fig:homology-curves} to compute in $\left(\wedge^3 H\right)/H$.
\begin{lem}\label{lem:primitive-img}
There exists a primitive $v \in \left(\wedge^3 H\right)/H$ so that $nv$ is contained in $\tau(G_n^{\mathcal I})$ for all $n \in \Z_{\geq 0}$.
\end{lem}
\begin{proof}
Let
\[
	v := \tau([T_{B_0}^{-1}, f]).
\]
To see that $nv$ is contained in $\tau(G_n^{\mathcal I})$ for all $n \in \Z_{\geq 0}$, recall that $A_n \leq G_n^{\mathcal I}$ and compute that
\[
	\tau([T_{B_0}^{-1}, f^n]) = \tau(T_{B_0}^{-1} f^n T_{B_0}) - \tau(f^n) = n(\tau(T_{B_0}^{-1} f T_{B_0}) - \tau(f)) = n v.
\]

It remains to show that $v$ is primitive. First, note that
\[
	\tau(T_xT_y^{-1}) = a_1 \wedge b_1 \wedge b_2 = a_1 \wedge b_1 \wedge c_1,
\]
where the first equality follows from \cite[p. 195-196]{farb-margalit} and the second equality follows because $b_2 = b_1 + c_1$ as elements of $H$. Therefore,
\begin{align*}
\tau(f) &= \tau(T_x T_y^{-1}) + \tau(T_{h(x)} T_{h(y)}^{-1}) \\
&= \tau(T_x T_y^{-1}) + h \cdot \tau(T_x T_y^{-1}) \qquad \text{by naturality (\ref{eqn:naturality}) of }\tau \\
&= a_1 \wedge b_1 \wedge c_1 + (-a_{2g}) \wedge (-b_{2g}) \wedge c_{2g-1} .
\end{align*}
Now compute in $\left(\wedge^3 H\right)/H$
\begin{align*}
v &= \tau(T_{B_0}^{-1} f T_{B_0}) - \tau(f) \\
&=T_{B_0}^{-1}\left(a_1 \wedge b_1 \wedge c_1 + a_{2g} \wedge b_{2g} \wedge c_{2g-1}\right) - \left(a_1 \wedge b_1 \wedge c_1 + a_{2g} \wedge b_{2g} \wedge c_{2g-1}\right) \quad \text{by naturality (\ref{eqn:naturality}) of }\tau \\
&= \left(a_1 \wedge (b_1-B_0) \wedge c_1 + a_{2g} \wedge (b_{2g}-B_0) \wedge c_{2g-1}\right) - \left(a_1 \wedge b_1 \wedge c_1 + a_{2g} \wedge b_{2g} \wedge c_{2g-1}\right) \\
&= (a_1 \wedge c_1 + a_{2g} \wedge c_{2g-1})\wedge B_0.
\end{align*}
Here and in the rest of this proof, we also denote by $B_0$ the homology class of $B_0$, oriented as in Figure \ref{fig:homology-curves}.

There exists a symplectic basis $\alpha_1, \beta_1, \dots, \alpha_{2g}, \beta_{2g}$ of $H_1(\Sigma_{2g})$ with
\[
	\alpha_1 = -a_1, \, \beta_1 = c_1, \, \alpha_2 = -a_{2g}, \, \beta_2 = c_{2g-1}, \, \alpha_3 = B_0.
\]
By Lemma \ref{lem:nonzero-wedge}, the set $\{ a_1 \wedge c_1 \wedge B_0, \, a_{2g}\wedge c_{2g-1} \wedge B_0 \}$ can be completed to a $\Z$-basis of $\left(\wedge^3 H_1(\Sigma_{2g})\right) / H_1(\Sigma_{2g})$. Therefore, the sum $a_1 \wedge c_1 \wedge B_0 + a_{2g} \wedge c_{2g-1} \wedge B_0$ is primitive.
\end{proof}

The following proposition forms the main computational tool in the proof of Theorem \ref{thm:non-isomorphic}.
\begin{prop}\label{prop:non-conj}
If $n \neq m \in \Z_{\geq 0}$ then $G_n$ and $G_m$ are not conjugate as subgroups of $\Mod(\Sigma_{2g})$.
\end{prop}
\begin{proof}
Suppose that there exists $k \in \Mod(\Sigma_{2g})$ so that
\[
	k G_n k^{-1} = G_m \leq \Mod(\Sigma_{2g}).
\]
Because $\mathcal I_{2g}$ is normal in $\Mod(\Sigma_{2g})$, this implies that
\[
	k G_n^{\mathcal I} k^{-1} = G_m^{\mathcal I} \leq \mathcal I_{2g}.
\]
By naturality (\ref{eqn:naturality}) of $\tau$ and Lemma \ref{lem:GnI-divisibility},
\[
	\tau(G_m^{\mathcal I}) = k \cdot \tau(G_n^{\mathcal I}) \leq n \left(\wedge^3 H\right)/H.
\]
By Lemma \ref{lem:primitive-img}, there exists a primitive class $v \in \left(\wedge^3 H\right)/H$ so that $mv$ is contained in $n \left(\wedge^3 H\right)/H$; in other words, $n$ divides $m$. By symmetry, $m$ also divides $n$, i.e. $n = m$.
\end{proof}

The main theorem of this section now follows as an immediate consequence.
\begin{proof}[Proof of Theorem \ref{thm:non-isomorphic}]
If $n \neq m \in \Z_{\geq 0}$ then the $G_n$ and $G_m$ are not conjugate as subgroups of $\Mod(\Sigma_{2g})$. By construction, $G_n$ and $G_m$ are the images $\mathrm{im}(\rho_n)$ and $\mathrm{im}(\rho_m)$ of the monodromy representations $\rho_n$ and $\rho_m$ of $\pi_n$ and $\pi_m$ respectively. By Corollary \ref{cor:noniso-criterion}, $\pi_n$ and $\pi_m$ are inequivalent Lefschetz fibrations.
\end{proof}


\section{Ruled surfaces and their blowups}\label{sec:diffeo}

The goal of this section is to determine the diffeomorphism type of the Lefschetz fibrations $\pi_n: X_n \to S^2$ of Section \ref{sec:twisted-mck} (Proposition \ref{prop:Xn-diffeo}). Along the way, we will also determine the diffeomorphism type of the blowdown of $X_n$ of the four $(-1)$-sections found in Section \ref{sec:twisted-mck} (Proposition \ref{prop:diffeo-ruled}).

\subsection{Diffeomorphism type of $X_n$}

We first compute the algebraic topology invariants of $X_n$.
\begin{lem}\label{lem:Xn-alg-top}
For any $n \in \Z_{\geq 0}$,
\[
	\sigma(X_n) = -4, \quad \chi(X_n) = 8 - 4g, \quad b_1(X_n) = 2g, \quad b_2^+(X_n) = 1, \quad b_2(X_n) = 6
\]
\end{lem}
\begin{proof}
By Korkmaz's computations \cite[Section 5]{korkmaz}, the lemma holds in the case of $n = 0$, i.e. the MCK Lefschetz fibration $\pi_0: X_0 \to S^2$. (In fact, Korkmaz determines the diffeomorphism type of $X_0$ for $g \geq 3$; also see Proposition \ref{prop:Xn-diffeo}.)

To compute $\sigma(X_n)$ for $n \geq 1$, consider the Lefschetz fibration $Z \to D^2$ with monodromy factorization
\[
	T_{B_0} \dots T_{B_{2g}} T_C \in \Mod(\Sigma_{2g}).
\]
Then $X_n$ is formed by gluing two copies of $Z$ together along their boundaries by some diffeomorphism $\partial Z \to \partial Z$ (which varies with $n$) for any $n \geq 0$ (cf. Section \ref{sec:partial-conj}). By Novikov additivity,
\[
	\sigma(X_n) = 2\sigma(Z) = \sigma(X_0) = -4
\]
for all $n \geq 0$.

To compute $\chi(X_n)$ for all $n \in \Z_{\geq 0}$, note that $\pi_n: X_n \to S^2$ has $(4g+4)$-many vanishing cycles and that
\[
	\chi(X_n) = 4 - 8g + (4g+4) = 8 - 4g.
\]

To compute $b_1(X_n)$ for $n \geq 1$, recall first that the positive factorizations of $T_{\delta_1} T_{\delta_2} T_{\delta_3} T_{\delta_4}\in \Mod(\Sigma_{2g}^4)$ given in Lemma \ref{lem:bounding-pair-sections} are lifts of the positive factorization of the identity in $\Mod(\Sigma_{2g})$ given in Corollary \ref{cor:partial-MCKs} that defines the Lefschetz fibration $\pi_n$, and hence $\pi_n: X_n \to S^2$ admits sections. Therefore,
\[
	H_1(X_n) \cong H_1(\Sigma_{2g})/ \Z\{[B_0],\, \dots, \,[B_{2g}],\, [C], \, [f^n(B_0)],\, \dots, \,[f^n(B_{2g})], \, [f^n(C)] \}.
\]
Because $f^n$ acts trivially on $H_1(\Sigma_{2g})$, this implies that $b_1(X_n) = b_1(X_0) = 2g$ for all $n \geq 1$.

Finally, compute $b_2^+(X_n)$ and $b_2(X_n)$ by solving the system of equations
\begin{align*}
8 - 4g &= \chi(X_n) = 1 - 2g + (b_2^+(X_n) + b_2^-(X_n)) - 2g + 1, \\
-4 &= \sigma(X_n) = b_2^+(X_n) - b_2^-(X_n). \qedhere
\end{align*}
\end{proof}

The algebraic topology invariants of $X_n$ determine its diffeomorphism type by a theorem of Liu \cite{liu}.
\begin{prop}\label{prop:Xn-diffeo}
For any $n \in \Z_{\geq 0}$, the $4$-manifold $X_n$ is diffeomorphic to $\left(\Sigma_g \times S^2\right) \# 4\overline{\CP^2}$ and to $(\Sigma_g \tilde\times S^2) \# 4 \overline{\CP^2}$.
\end{prop}
\begin{proof}
By the Gompf--Thurston construction, the $4$-manifold $X_n$ admits a symplectic form. Let $M_n$ be a minimal, symplectic $4$-manifold so that $M_n \# k\overline{\CP^2}$ is diffeomorphic to $X_n$. Because $b_2(X_n) = 6$ and $b_2^+(X_n) = 1$ by Lemma \ref{lem:Xn-alg-top}, there is a bound $0 \leq k \leq 5$. By Lemma \ref{lem:Xn-alg-top},
\[
	\sigma(M_n) = -4 + k, \quad \chi(M_n) = 8 - 4g - k, \quad b_1(M_n) = 2g, \quad b_2^+(M_n) = 1.
\]
Then because $g \geq 2$,
\[
	c_1^2(M_n) = 2\chi(M_n) + 3\sigma(M_n) = 4 - 8g + k \leq 9 - 8g < 0.
\]
Liu's theorem \cite[Theorem A]{liu} shows that $M_n$ is an irrational ruled surface. Hence $M_n$ is an $S^2$-bundle over $\Sigma_g$ because $b_1(M_n) = 2g$, i.e. $M_n$ is diffeomorphic to $\Sigma_g \times S^2$ or to $\Sigma_g \widetilde{\times} S^2$, and $k = 4$ by Euler characteristic considerations. Since $\left(\Sigma_g \times S^2\right) \# \overline{\CP^2}$ and $\left(\Sigma_g \widetilde\times S^2\right) \# \overline{\CP^2}$ are diffeomorphic, the proposition follows.
\end{proof}

An immediate corollary is the smooth portion of Theorem \ref{thm:infty-fibrations}.
\begin{cor}\label{cor:smooth-infty-fib}
Let $X$ be a ruled surface with $\chi(X) = 4 - 4g < 0$. There exist infinitely many inequivalent Lefschetz fibrations $X \# 4 \overline{\CP^2} \to S^2$ of genus $2g$.
\end{cor}
\begin{proof}
By assumption, $X = \Sigma_g \times S^2$ or $\Sigma_g \tilde\times S^2$ with $g \geq 2$. By Proposition \ref{prop:Xn-diffeo}, there are diffeomorphisms $X_n \cong X \# 4 \overline{\CP^2}$ for all $n \in \Z_{\geq 0}$. By Theorem \ref{thm:non-isomorphic}, the Lefschetz fibrations $\pi_n: X \# 4 \overline{\CP^2} \to S^2$ of genus $2g$ are pairwise inequivalent for all $n \in \Z_{\geq 0}$.
\end{proof}

\subsection{Lefschetz pencils on ruled surfaces}\label{sec:pencils-ruled}

In this section we prove Theorem \ref{thm:infty-pencils-ruled} by studying the topology of $X_n$ after blowing down certain sets of disjoint $(-1)$-sections.

\begin{defn}\label{defn:sections-monodromy}
Fix any $n \in \Z_{\geq 0}$. For $i = 1$ or $2$, let $\sigma_{1, i}^n$, $\sigma_{2, i}^n$, $\sigma_{3, i}^n$, $\sigma_{4, i}^n$ denote the four disjoint $(-1)$-sections of $\pi_n: X_n \to S^2$ defined by the positive factorization (Lemma \ref{lem:bounding-pair-sections})
\begin{equation}\label{eqn:sections-monodromy}
	T_{B_{0,1}^i} T_{B_{1,1}^i} \dots T_{B_{2g,1}^i} T_{C_1^i} T_{\tilde f^n(B_{0, 2}^i)} T_{\tilde f^n(B_{1,2}^i)} \dots T_{\tilde f^n(B_{2g,2}^i)} T_{\tilde f^n(C_2^i)} = T_{\delta_1} T_{\delta_2} T_{\delta_3} T_{\delta_4} \in \Mod(\Sigma_{2g}^4)
\end{equation}
where the curves $B_{0, 1}^i, \dots, B_{2g, 1}^i, C_1^i, B_{1,2}^i, \dots, B_{2g,2}^i, C_2^i$ in $\Sigma_{2g}^4$ are as shown in Figure \ref{fig:MCK-lifts}.
\end{defn}
Similarly as in the proof of Proposition \ref{prop:Xn-diffeo}, let $M_n^i$ denote the $4$-manifold obtained by blowing down the $(-1)$-sections $\sigma_{1,i}^n$, $\sigma_{2,i}^n$, $\sigma_{3, i}^n$, and $\sigma_{4, i}^n$ in $X_n$.
\begin{lem}\label{lem:parity-ruled}
Let
\[
	Q_{M_n^i}: H_2(M_n^i; \Z) \times H_2(M_n^i; \Z) \to \Z
\]
denote the intersection form of $M_n^i$. The lattice $(H_2(M_n^i), Q_{M_n^i})$ is even if $i = 1$ and is odd if $i = 2$.
\end{lem}
\begin{proof}
First, we determine the intersection form of $M_n^i$. The blowdown $X_n \to M_n^i$ decomposes the lattice $(H_2(X_n), Q_{X_n})$ as an orthogonal direct sum
\[
	(H_2(X_n), Q_{X_n}) \cong (H_2(M_n^i), Q_{M_n^i}) \oplus (\Z\{\sigma_{1, i}^n, \sigma_{2, i}^n, \sigma_{3, i}^n, \sigma_{4, i}^n\}).
\]

Consider the reducible singular fiber $F \subseteq X_n$ of $\pi_n$ corresponding to the vanishing cycle $C_1^i$ in the monodromy factorization of $\pi_n$ as in (\ref{eqn:sections-monodromy}). Then $F$ is a union of two genus-$g$ surfaces $F_1$ and $F_2$ intersecting transversely once, and
\[
	Q_{X_n}([F_1], [F_1]) = Q_{X_n}([F_2], [F_2]) = -1, \qquad Q_{X_n}([F_1], [F_2]) = 1.
\]

Below, we determine the parity of the lattice $(H_2(M_n^i), Q_{M_n^i})$ depending on the index $i$. Note that $H_2(M_n^i)$ is torsion-free for both $i = 1, 2$ because $H_2(X_n)$ is, by Proposition \ref{prop:Xn-diffeo}.
\begin{enumerate}
\item If $i = 1$ then for some $\alpha_1, \alpha_2 \in H_2(M_n^1)$
\[
	[F_1] = \alpha_1 - [\sigma_{4,1}^n], \qquad [F_2] = \alpha_2 - [\sigma_{1,1}^n] - [\sigma_{2,1}^n] - [\sigma_{3,1}^n]
\]
as homology classes in $H_2(X_n)$ with respect to the orthogonal sum decomposition given above, up to possibly permuting the indices of $F_1, F_2$. Then because the sections $\sigma_{j, 1}^n$ have self-intersection $-1$,
\[
	Q_{M_n^1}(\alpha_1, \alpha_1) = 0, \qquad Q_{M_n^1}(\alpha_2, \alpha_2) = 2, \qquad Q_{M_n^1}(\alpha_1, \alpha_2) = 1.
\]
The restriction of $Q_{M_n^1}$ to the $\Z$-span $\Z\{\alpha_1, \alpha_2\}$ is unimodular. Moreover, $H_2(M_n^1) \cong \Z^2$ as a group by Lemma \ref{lem:Xn-alg-top}, and so $\Z\{\alpha_1, \alpha_2\}$ has finite index in $H_2(M_n^1)$. These two facts together imply that $\Z\{\alpha_1, \alpha_2\} = H_2(M_n^1)$. Therefore, the lattice $(H_2(M_n^1), Q_{M_n^1})$ is even because $Q_{M_n^1}(\alpha_j, \alpha_j)$ is even for both $j = 1, 2$.

\item If $i = 2$ then $[F_1] \in H_2(X_n)$ is orthogonal to the sections $\sigma_{1, 2}^n, \sigma_{2,2}^n, \sigma_{3,2}^n, \sigma_{4,2}^n$ and so $[F_1]$ is contained in $H_2(M_n^2)$ under the orthogonal sum decomposition given above, and
\[
	Q_{M_n^2}([F_1], [F_1]) = -1.
\]
Because there exists an element of $H_2(M_n^2)$ with odd self-intersection, the lattice $(H_2(M_n^2), Q_{M_n^2})$ is odd. \qedhere
\end{enumerate}
\end{proof}

The following proposition uses Liu's theorem \cite[Theorem A]{liu} and the parity of the intersection form of $M_n^i$ to determine its diffeomorphism type.
\begin{prop}\label{prop:diffeo-ruled}
There are diffeomorphisms
\[
	M_n^i \cong \begin{cases}
	\Sigma_g \times S^2 & \text{ if } i = 1, \\
	\Sigma_g \tilde\times S^2& \text{ if } i = 2.
	\end{cases}
\]
\end{prop}
\begin{proof}
For both $i = 1, 2$, the $(-1)$-sections $\sigma_{1,i}^n, \sigma_{2,i}^n, \sigma_{3,i}^n, \sigma_{4,i}^n$ are disjoint. By the Gompf--Thurston construction \cite[Theorem 10.2.18]{gompf-stipsicz}, there exists a symplectic form on $X_n$ turning the sections $\sigma_{1,i}^n, \sigma_{2,i}^n, \sigma_{3,i}^n, \sigma_{4,i}^n$ into symplectic submanifolds. Blowing down these sections then yields a symplectic form on $M_n^i$. Using Lemma \ref{lem:Xn-alg-top}, compute that
\[
	\sigma(M_n^i) = 0, \quad \chi(M_n^i) = 4-4g, \quad b_2^+(M_n^i) = 1, \quad b_1(M_n^i) = 2g, \quad b_2(M_n^i) = 2, \quad c_1^2(M_n^i) = 8-8g.
\]
Therefore, $c_1^2(M_n^i) < 0$ because $g \geq 2$.

Let $N_n^i$ be a minimal symplectic manifold so that $M_n^i$ is a symplectic blowup of $N_n^i$. If $M_n^i$ is not minimal then because $b_2^-(M_n^i) = 1$, we can compute that
\[
	b_2^+(N_n^i) = b_2 = 1, \qquad c_1^2(N_n^i) = 9 - 8g < 0.
\]
By Liu's theorem \cite[Theorem A]{liu}, $N_n^i$ is an irrational ruled surface, which is a contradiction because $\sigma(N_n^i) \neq 0$. Therefore, $M_n^i$ is a minimal symplectic $4$-manifold to which Liu's theorem \cite[Theorem A]{liu} applies, i.e. $M_n^i$ is an irrational ruled surface with $b_1(M_n^i)  = 2g$. This means that $M_n^i$ is diffeomorphic to either $\Sigma_g \times S^2$ or $\Sigma_g \tilde\times S^2$. Because $\Sigma_g \times S^2$ has even intersection form and $\Sigma_g \tilde\times S^2$ has odd intersection form, Lemma \ref{lem:parity-ruled} gives the desired conclusion.
\end{proof}

With Proposition \ref{prop:diffeo-ruled}, we are able to define the Lefschetz pencils of interest on ruled surfaces.
\begin{defn}\label{defn:infty-pencils}
For any $n \in \Z_{\geq 0}$ and $i = 1, 2$, let $B_n^i \subseteq M_n^i$ denote the image of the four sections $\sigma_{1,i}^n$, $\sigma_{2,i}^n$, $\sigma_{3,i}^n$, and $\sigma_{4, i}^n$ under the blowdown $X_n \to M_n^i$. Let
\[
	\pi_{n,i}: M_n^i - B_n^i \to S^2
\]
denote the the Lefschetz pencil induced by the Lefschetz fibration $\pi_n: X_n \to S^2$.
\end{defn}
The following corollary proves the smooth portion of Theorem \ref{thm:infty-pencils-ruled}.
\begin{cor}\label{cor:smooth-infty-pencil}
Let $X$ be a ruled surface with $\chi(X) = 4 - 4g < 0$. There are infinitely many pairwise inequivalent Lefschetz pencils $X - B \to S^2$ of genus $2g$ and $\# B = 4$.
\end{cor}

\begin{proof}
By assumption, $X = \Sigma_g \times S^2$ or $\Sigma_g \tilde\times S^2$ with $g \geq 2$. By Proposition \ref{prop:diffeo-ruled}, there exists $i \in \{1, 2\}$ such that $X$ is diffeomorphic $M_n^i$ for all $n \in \Z_{\geq 0}$. We may assume that $B_n^i = B$ under these diffeomorphisms for all $n \in \Z_{\geq 0}$.

Because the Lefschetz fibrations $\pi_n: X_n \to S^2$ are pairwise inequivalent by Theorem \ref{thm:non-isomorphic}, Proposition \ref{prop:noniso-pencils} implies that the pencils $\pi_{n,i}: X - B \to S^2$ and $\pi_{m,i}: X - B \to S^2$ are inequivalent Lefschetz pencils if $n \neq m$.
\end{proof}

\begin{rmk}[{Comparison with \cite{LS-sections}}]\label{rmk:prev-work}
In \cite[Corollary 1.5, Remark 3.4]{LS-sections} we constructed a fiber-sum indecomposable Lefschetz fibration $M_g \to S^2$ for every genus $g \geq 2$ that admits infinitely many homologously distinct sections of equal self-intersection. One may hope to blow down these sections to obtain infinitely many pairwise inequivalent Lefschetz pencils. However, this requires the sections to have self-intersection $-1$. Li \cite[Corollary 3]{li-symplectic-spheres} showed that any closed, symplectic $4$-manifold admitting infinitely many homology classes represented by smoothly embedded spheres of self-intersection $-1$ is rational or ruled; on the other hand, the symplectic manifold $M_g$ constructed is neither rational nor ruled because $b_2^+(M_g) \geq 2$ \cite[Proposition 6.6]{LS-sections}.

Even aside from the specific examples $\pi: M_g \to S^2$ above, the constructions of \cite{LS-sections} always yield sections $\sigma$ of self-intersection $[\sigma]^2 \leq -2$. Below, we prove this upper bound using an idea of Smith \cite{smith} to study the action of $\Mod(\Sigma_{g,1})$ on $\partial \mathbb H^2 \cong S^1$.

Let $p \in \Sigma_{g,1}$ denote the marked point and fix some lift $\tilde p \in \mathbb H^2$ of $p \in \Sigma_{g,1}$ under the covering $\mathbb H^2 \to \Sigma_g$. There is a well-defined homomorphism \cite[Section 5.5.4]{farb-margalit}
\[
	\partial: \Mod(\Sigma_{g,1}) \hookrightarrow \Homeo^+(\partial \mathbb H^2)
\]
where $\partial h \in \Homeo^+(\partial \mathbb H^2)$ is the homeomorphism induced by the lift $\tilde \varphi \in \Diff^+(\mathbb H^2, \tilde p)$ of any representative $\varphi \in \Diff^+(\Sigma_{g,1})$ of $h$. For any Dehn twist $T_\ell \in \Mod(\Sigma_{g,1})$, the homeomorphism $\partial T_{\ell}$ fixes countably many points of $\partial \mathbb H^2$ and moves all other points of $\partial \mathbb H^2$ clockwise \cite[Proposition 2.1]{smith}.

Consider any sequence $T_{\ell_1}, \dots, T_{\ell_r} \in \Mod(\Sigma_{g,1})$ of Dehn twists such that
\[
	T_{\ell_r} \dots T_{\ell_1} = 1 \in \Mod(\Sigma_{g,1})
\]
and let $\pi: M \to S^2$ and $\sigma: S^2 \to M$ denote the Lefschetz fibration and section corresponding to this positive factorization. Because each $\partial T_{\ell_i} \in \Homeo^+(\partial \mathbb H^2)$ fixes points, the sequence $\partial T_{\ell_1}, \dots, \partial T_{\ell_r}$ of homeomorphisms admits a well-defined rotation number $c \in \N$. According to \cite[Lemma 2.3]{smith}, the rotation number $c$ is equal to $-[\sigma]^2$, where $[\sigma]^2$ denotes the self-intersection of $\sigma$.

Consider a Lefschetz fibration $\pi: M \to S^2$ and sections $\sigma_k: S^2 \to M$ for any $k \in \Z$ constructed in \cite{LS-sections}. The monodromy factorization of $\sigma_k$ is of the form \cite[Theorem 2.1(b)]{LS-sections}
\[
	T_{\ell_{r_1+r_2}} \dots T_{\ell_{r_1+1}} T_{P_\gamma^k(\ell_{r_1})} \dots T_{P_\gamma^k(\ell_{1})} =1 \in \Mod(\Sigma_{g,1})
\]
for some point-push mapping class $P_\gamma \in \pi_1(\Sigma_g, p) \trianglelefteq \Mod(\Sigma_{g,1})$ commuting with the product $T_{\ell_{r_1}} \dots T_{\ell_1} \in \Mod(\Sigma_{g,1})$ such that for some $1 \leq i \leq r_1$ and for some $r_1+1 \leq j \leq r_1+r_2$,
\[
	\hat i([\gamma], [\ell_i]) \neq 0 \qquad\text{ and } \qquad \hat i([\gamma], [\ell_j]) \neq 0
\]
where $\hat i$ denotes the (algebraic) intersection form of $\Sigma_g$. In particular, no power of $T_{\ell_i}(\gamma)$ or $T_{\ell_j}(\gamma)$ is contained in $\langle \gamma \rangle \leq \pi_1(\Sigma_g, p)$.

The homeomorphism $\partial P_\gamma$ of $\partial \mathbb H^2$ coincides with the action of a deck transformation of $\mathbb H^2 \to \Sigma_g$ \cite[p. 150-151]{farb-margalit} and hence has exactly two fixed points, one attracting point $p_1$ and one repelling point $p_2$. Because $P_\gamma$ and $T_{\ell_{r_1}} \dots T_{\ell_1}$ commute, the homeomorphisms $\partial(T_{\ell_{r_1}} \dots T_{\ell_1})$ and $\partial (T_{\ell_{r_1+r_2}} \dots T_{\ell_{r_1+1}})$ must also fix each point $p_1, p_2 \in \partial \mathbb H^2$. Because no power of $T_{\ell_i} P_\gamma T_{\ell_i}^{-1}$ is contained in $\langle P_\gamma \rangle$, the homeomorphism $\partial (T_{\ell_i} P_\gamma T_{\ell_i}^{-1})$ moves at least one point $p_1$ or $p_2$. In other words, $\partial T_{\ell_i}$ moves at least one point $p_1$ or $p_2$ clockwise, and similarly for $\partial T_{\ell_j}$. Hence each sequence $\partial T_{\ell_1}, \dots, \partial T_{\ell_{r_1}}$ and $\partial T_{\ell_{r_1}}, \dots, \partial T_{\ell_{r_1+r_2}}$ moves the point $p_1$ around $\partial \mathbb H^2$ clockwise at least once. Therefore, the rotation number $c\in \N$ of the sequence $\partial T_{\ell_1}, \dots, \partial T_{\ell_{r_1+r_2}}$ is at least $2$, and $[\sigma_k]^2 = -c \leq -2$.
\end{rmk}


\section{Symplectic forms}\label{sec:symplectic}

Throughout this section, fix $g \geq 2$ and let $M^1$ and $M^2$ denote the $4$-manifolds
\[
	M^1 := \Sigma_g \times S^2, \qquad M^2 := \Sigma_g \tilde\times S^2.
\]
The goal of this section is to prove that the Lefschetz pencils $\pi_{n,i}: M_n^i - B_n^i \to S^2$ defined in Definition \ref{defn:infty-pencils} are symplectic for a common symplectic form on $M^i$ for all $n \in \Z_{\geq 0}$ with respect to some diffeomorphism $M_n^i \cong M^i$.
\begin{thm}\label{thm:symplectic-LP}
Let $i = 1$ or $i = 2$. There exist a symplectic form $\omega$ on $M^i$ and diffeomorphisms $\Psi_n^i: M^i_n \to M^i$ for all $n \in \Z_{\geq 0}$ such that the smooth locus of the irreducible components of all fibers of $\pi_{n,i}: M^i_n - B_n^i \to S^2$ are all symplectic submanifolds of $(M^i_n, (\Psi_n^i)^*\omega)$. Moreover, the regular fibers $F_n$ of $\pi_{n,i}$ are all homologous in $H_2(M^i; \Z)$ under $\Psi_n^i$, i.e.
\[
	(\Psi_n^i)_*([F_n]) = (\Psi_m^i)_*([F_m]) \in H_2(M^i; \Z)
\]
for all $n, m \in \Z_{\geq 0}$.
\end{thm}

The following theorem about the Lefschetz fibrations $\pi_n: X \to S^2$ will also follow from our proof of Theorem \ref{thm:symplectic-LP}.
\begin{thm}\label{thm:symplectic-LF}
There exist a symplectic form $\omega$ on $(\Sigma_g \times S^2) \# 4\overline{\CP^2}$ and diffeomorphisms $\Phi_n: X_n \to (\Sigma_g \times S^2) \# 4\overline{\CP^2}$ for each $n \in \Z_{\geq 0}$ such that the smooth locus of the irreducible components of all fibers of $\pi_n: X_n \to S^2$ are all symplectic submanifolds of $(X_n, \Phi_n^*\omega)$. Moreover, the regular fibers $F_n$ of $\pi_n$ are all homologous in $H_2((\Sigma_g \times S^2) \# 4\overline{\CP^2}; \Z)$ under $\Phi_n$, i.e.
\[
	(\Phi_n)_*([F_n]) = (\Phi_m)_*([F_m]) \in H_2((\Sigma_g \times S^2) \# 4\overline{\CP^2}; \Z)
\]
for all $n, m \in \Z_{\geq 0}$.
\end{thm}

We begin with an overview of the proofs of Theorems \ref{thm:symplectic-LP} and \ref{thm:symplectic-LF}.
\begin{enumerate}
\item The existence of the diffeomorphism $\Phi_n:X_n \to M^i\#4\overline{\CP^2}$
  sending regular fibers (and in fact the irreducible components of reducible singular fibers) to the same homology class
follows from a more careful analysis of the diffeomorphism found in \Cref{prop:diffeo-ruled}.
This is done in \Cref{lem:nu-choice}, relying on work of Liu~\cite{liu} and minimal genus computations in ruled surfaces due to Li--Li~\cite{li-li-minimal-genus}. Blowing down produces the diffeomorphisms $\Psi_n^i$.

\item Since $X_n$ is the blow-up of a ruled surface, it suffices to construct appropriate symplectic forms $\omega_n$ of $X_n$ such that $[(\Phi_n^{-1})^*(\omega_n)] \in H^2(M^i \# 4 \overline{\CP^2}; \Z)$ is independent of $n \in \Z_{\geq 0}$ by work of Lalonde--McDuff~\cite{lalonde-mcduff} and McDuff~\cite{mcduff1996symplectic}.

\item The first step is to build a suitable symplectic form $\omega_0$ in $X_0$ by the Gompf--Thurston construction. Realize $X_0$ (with four sections $\sigma_{1,i}^0$, $\sigma_{2,i}^0$, $\sigma_{3, i}^0$, $\sigma_{4,i}^0$) as the fiber sum of two fibrations over the two hemispheres on $S^2$ glued over the equator $E$, and let $A$ be an annulus neighborhood of $E$. An understanding of the (co)homology of $\pi_0^{-1}(A)$ allows us to control $\omega_0$ on $A$. Furthermore, we construct $\omega_0$ to be a standard symplectic form near the four $(-1)$-sections $\sigma_{1,i}^0$, $\sigma_{2,i}^0$, $\sigma_{3,i}^0$, $\sigma_{4,i}^0$ to ensure that the fibers remain symplectic in the blow-down $M_0^i$. This is all achieved in \Cref{lem:theta-cohomology,prop:GT-X0}.

\item Finally, to construct the forms $\omega_n$ we use the description of $X_n$ via partial conjugation: It is enough to find a representative $\varphi$ of the twisting mapping class $f \in \mathcal I_{2g}$ that induces a bundle symplectomorphism of $(\pi_0^{-1}(A),\omega_0|_A)$ acting by the identity near the sections $\sigma_{1,i}^0$, $\sigma_{2,i}^0$, $\sigma_{3,i}^0$, $\sigma_{4,i}^0$. In fact, we construct both the ``fiber part'' of $\omega_0|_A$ and $\varphi$ simultaneously in \Cref{prop:isotoping-hateta} (and discussion thereafter). The key observation is that the mapping class $\tilde{h}_i \in \Mod(\Sigma_{2g}^4)$ described in \Cref{thm:hamada-sections} and \Cref{lem:bounding-pair-sections} is \emph{reducible}.
\end{enumerate}

\subsection{(Co)homology computations}\label{sec:cohom}

The following lemma determines the homology classes of various fiber classes and section classes of $\pi_n: X_n \to S^2$ in $H_2(X_n; \Z)$. We will also specify a cohomology class $[\nu_i] \in H^2(X_0; \R)$ that will be a key input to the Gompf--Thurston construction of a symplectic form $\omega_0$ on $X_0$. Below, recall that $\sigma_{j, i}^n: S^2 \to X_n$ for $1 \leq j \leq 4$ and for $i = 1 ,2$ is the $(-1)$-section of $\pi_n$ specified in Definition \ref{defn:sections-monodromy}.
\begin{lem}\label{lem:nu-choice}
For any $n \in \Z_{\geq 0}$, let $F_1^n \cup F_2^n$ and $F_3^n \cup F_4^n$ denote the two reducible singular fibers of $\pi_n: X_n \to S^2$  corresponding to the vanishing cycles $C_1^i$ and $\tilde f^n(C_2^i)$ respectively. Let $F_1^n, F_3^n$ and $F_2^n, F_4^n$ denote the irreducible components depicted on the left and right side respectively of the corresponding vanishing cycle in Figure \ref{fig:MCK-lifts}. Let $F^n$ denote any regular fiber of $\pi_n$.
\begin{enumerate}[(a)]
\item For each $i = 1, 2$, there exists a diffeomorphism
\[
	\Phi_n^i: X_n \to M^i \# 4 \overline{\CP^2}
\]
such that
\[
	(\Phi_n^i)_*([F^n]) = (\Phi_0^i)_*([F^0]),  \qquad (\Phi_n^i)_*([F_j^n]) = (\Phi_0^i)_*([F_j^0]), \qquad (\Phi_n^i)_*([\sigma_{j,i}^n]) = E_j
\]
for all $1 \leq j \leq 4$, where $E_1, \dots, E_4 \in H_2(M^i \# 4 \overline{\CP^2}; \Z)$ denote the exceptional divisors coming from each summand $ \overline{\CP^2}$. \label{lem:phi-n-i}
\item For each $i = 1, 2$, the homology classes $[F_1^n]$, $[F_2^n]$, $[\sigma_{1,i}^n]$, $[\sigma_{2,i}^n]$, $[\sigma_{3,i}^n]$, and $[\sigma_{4,i}^n]$ span $H_2(X_n; \R)$, i.e.
\[
	H_2(X_n; \R) = \R\{[F_1^n], [F_2^n], [\sigma_{1,i}^n], [\sigma_{2,i}^n], [\sigma_{3,i}^n], [\sigma_{4,i}^n]\} \cong \R^6.\label{lem:homology-xn}
\]
\item For each $i = 1, 2$, there exists a cohomology class $[\nu_i] \in H^2(X_0; \R)$ such that
\[
	\langle [\nu_i], [F^0] \rangle = 1, \qquad \langle [\nu_i],  [F_j^0] \rangle > 0, \qquad \langle [\nu_i], [\sigma_{j,i}^0] \rangle = 0
\]
for all $1 \leq j \leq 4$. \label{lem:nu-i}
\end{enumerate}
\end{lem}
\begin{proof}
Fix $n\in \Z_{\geq 0}$ and $i = 1, 2$. Recall from Proposition \ref{prop:diffeo-ruled} that blowing down the sections $\sigma_{1,i}^n$, $\sigma_{2,i}^n$, $\sigma_{3,i}^n$, and $\sigma_{4,i}^n$ in $X_n$ yields a diffeomorphism
\[
	\Phi_n^i: X_n \to M^i \# 4 \overline{\CP^2}
\]
with $(\Phi_n^i)([\sigma_{j,i}^n]) = E_j$ for all $1 \leq j \leq 4$. The connected sum structure determines an orthogonal decomposition
\[
	H_2(M^i \# 4 \overline{\CP^2}; \Z) \cong H_2(M^i; \Z) \oplus \Z\{E_1, E_2, E_3, E_4\}.
\]

There exists a diffeomorphism $M^i \to M^i$ acting by negation on $H_2(M^i; \Z)$ by \cite[Theorem 3]{li-li-minimal-genus}. Therefore, there exists a diffeomorphism $\Psi^i \in \Diff^+(M^i \# 4 \overline{\CP^2})$ such that
\[
	\Psi^i_* = -\Id \oplus \Id: H_2(M^i; \Z) \oplus \Z\{E_1, E_2, E_3, E_4\} \to H_2(M^i; \Z) \oplus \Z\{E_1, E_2, E_3, E_4\}
\]
by \cite[Lemma 2]{wall}.

Irreducible components of reducible fibers of Lefschetz fibrations have self-intersection $-1$. Because $F_1^n$ and $F_2^n$ intersect once transversely (similarly, $F_3^n$ and $F_4^n$),
\begin{equation}\label{eqn:fiber-intersections}
	Q_{X_n}([F_j^n], [F_j^n]) = -1, \qquad Q_{X_n}([F_1^n], [F_2^n]) = Q_{X_n}([F_3^n], [F_4^n]) = 1,
\end{equation}
for all $1 \leq j \leq 4$, and $Q_{X_n}([F_j^n], [F_k^n]) = 0$ for all other pairs $1 \leq j, k \leq 4$. The intersection numbers $Q_{X_n}([F_j^n], [\sigma_{k,i}^n])$ for $1 \leq j, k \leq 4$ are specified by Figure \ref{fig:MCK-lifts}. By computing with these intersection numbers, we prove \ref{lem:phi-n-i}, \ref{lem:homology-xn}, \ref{lem:nu-i} separately in each case $i = 1$ and $i = 2$.

\medskip
\noindent
\emph{Case 1: $i = 1$.} Suppose $i = 1$ and consider $M^1 = \Sigma_g \times S^2$. There exist $\alpha_1, \alpha_2, \alpha_3, \alpha_4 \in H_2(M^1; \Z)$ such that
\begin{align*}
	(\Phi_n^1)_*[F^n_1] &= \alpha_{1} - E_4, \quad &(\Phi_n^1)_*[F^n_2] = \alpha_2 - E_1 - E_2 - E_3, \\
	(\Phi_n^1)_*[F^n_3] &= \alpha_{3} - E_3, \quad &(\Phi_n^1)_*[F^n_4] = \alpha_{4} - E_1 - E_2 - E_4.
\end{align*}
Then (\ref{eqn:fiber-intersections}) implies that
\[
	Q_{M^1}(\alpha_1, \alpha_1) = 0, \qquad Q_{M^1}(\alpha_2, \alpha_2)= 2, \qquad Q_{M^1}(\alpha_1, \alpha_2) = 1.
\]

Let $[\Sigma_g], [S^2] \in H_2(M^1; \Z)$ denote the classes of each factor. The only pairs of classes $\alpha_1, \alpha_2 \in H_2(M^1; \Z)$ satisfying the prescribed intersection pattern are
\begin{align*}
(\alpha_1, \alpha_2) = \pm ([\Sigma_g], [\Sigma_g] + [S^2]), \, \pm ([S^2], [\Sigma_g] + [S^2]).
\end{align*}

By blowing down the sections $\sigma_{j,1}^n$ for $1 \leq j \leq 4$, the diffeomorphism $\Phi_n^1$ induces a diffeomorphism $M_n^1 \to M^1 = \Sigma_g \times S^2$ sending $[F^n] \in H_2(M_n^1; \Z)$ to $\alpha_1 + \alpha_2 \in H_2(M^1; \Z)$, where we also denote by $F^n \subseteq M_n^1$ the image of a regular fiber $F^n \subseteq X_n$ under the blowdown $X_n \to M_n^1$. Because $F^n \subseteq M_n^1$ is a fiber of a genus-$2g$ Lefschetz pencil structure on $M_n^1$, there is a symplectic form on $M_n^1$ turning $F^n$ into a symplectic submanifold.\footnote{One can make a compatible symplectic form $\omega_n$ on $X_n$ standard near the sections, in a way that is compatible with the fibration (see \Cref{prop:GT-X0}\ref{standard-sections}). The form on $M_n^1$ is the blow-down of $\omega_n$~\cite[Section 5]{mcduff-polterovich}.} By the resolution of the symplectic Thom conjecture~\cite[Theorem 1.1]{ozsvath-szabo}, the smooth minimal genus of $[F^n] \in H_2(M_n^1; \Z)$ (and hence of $\alpha_1 + \alpha_2 \in H_2(M^1; \Z)$) is $2g$. However, the smooth minimal genus of $\pm (2[S^2] + [\Sigma_g]) \in H_2(M^1; \Z)$ is $g$ by work of Li--Li \cite[Theorem 1]{li-li-minimal-genus}. Therefore, $(\alpha_1, \alpha_2)$ is not equal to $\pm ([S^2], [\Sigma_g] + [S^2])$. By the same argument applied to $(\alpha_3, \alpha_4)$ and by noting that $[F^n] = [F^n_1] + [F^n_2] = [F^n_3] + [F^n_4]$, we conclude that
\[
	(\alpha_1, \alpha_2) = (\alpha_3, \alpha_4) = \pm ([\Sigma_g], [\Sigma_g] + [S^2]).
\]
After possibly replacing $\Phi_n^1$ with $\Psi^1 \circ \Phi_n^1$, 
\[
	(\Phi_n^1)_*([F_j^n]) = (\Phi_0^1)_*([F_j^0])
\]
for all $1 \leq j \leq 4$. This proves \ref{lem:phi-n-i} for $i = 1$.

The classes $\alpha_1$ and $\alpha_2$ form an $\R$-basis of $H_2(M^1; \R) \cong \R^2$, and so
\[
	\{(\Phi_n^1)_*([F^n_1]), (\Phi_n^1)_*([F^n_2]), E_1, E_2, E_3, E_4\}
\]
forms an $\R$-basis of $H_2(M^1 \# 4 \overline{\CP^2}; \R)$. Noting that $(\Phi_n^1)_*([\sigma_{j,1}^n]) = E_j$ for all $1 \leq j \leq 4$ concludes the proof of \ref{lem:homology-xn} for $i = 1$.

Finally, let $\beta \in H^2(M^1 \# 4 \overline{\CP^2}; \R)$ be any class such that
\[
	\langle \beta, \alpha_1\rangle = \langle \beta, \alpha_3 \rangle = \frac 12, \qquad \langle \beta, \alpha_2 \rangle = \langle \beta, \alpha_4 \rangle = \frac 12, \qquad \langle \beta, E_j \rangle = 0 \text{ for all }1 \leq j \leq 4
\]
which exists because $\alpha_1, \alpha_2, E_1, \dots, E_4$ are linearly independent in $H_2(M^1 \# 4 \overline{\CP^2}; \R)$. Letting $[\nu_1] = (\Phi_0^1)^*(\beta)$, compute for all $1 \leq j \leq 4$ that
\[
	\langle [\nu_1], [F^0_j] \rangle = \langle [\nu_1], [F^0_j] \rangle = \frac 12, \qquad \langle [\nu_1], [F^0] \rangle = \langle [\nu_1], [F^0_1] + [F^0_2] \rangle = 1, \qquad \langle [\nu_1], [\sigma_{j,1}^0] \rangle = 0,
\]
which proves \ref{lem:nu-i} for $i = 1$.

\medskip
\noindent
\emph{Case 2: $i = 2$.} Suppose $i = 2$ and consider $M^2 = \Sigma_g \tilde\times S^2$. There exist $\alpha_1, \alpha_2, \alpha_3, \alpha_4 \in H_2(M^2; \Z)$ such that
\begin{align*}
	(\Phi_n^2)_*[F^n_1] &= \alpha_{1}, \quad &(\Phi_n^2)_*[F^n_2] = \alpha_2 - E_1 - E_2 - E_3 - E_4, \\
	(\Phi_n^2)_*[F^n_3] &= \alpha_{3} - E_3 - E_4, \quad &(\Phi_n^2)_*[F^n_4] = \alpha_{4} - E_1 - E_2,
\end{align*}
Then (\ref{eqn:fiber-intersections}) implies that
\begin{align*}
	Q_{M^2}(\alpha_1, \alpha_1) &= -1, \qquad Q_{M^2}(\alpha_2, \alpha_2) = 3, \qquad Q_{M^2}(\alpha_1, \alpha_2) = 1, \\
	Q_{M^2}(\alpha_3, \alpha_3) &= 1, \qquad Q_{M^2}(\alpha_4, \alpha_4) = 1, \qquad Q_{M^2}(\alpha_3, \alpha_4) = 1.
\end{align*}
Let $[\Sigma_g] \in H_2(M^2; \Z)$ denote the class of a section of the $S^2$-bundle $M^2 \to \Sigma_g$ of self-intersection $1$ and let $[S^2] \in H_2(M^2; \Z)$ denote the class of the fiber. The only pairs of classes $\alpha_1, \alpha_2 \in H_2(M^2; \Z)$ and $\alpha_3, \alpha_4 \in H_2(M^2; \Z)$ satisfying the prescribed intersection pattern are
\begin{align*}
	(\alpha_1, \alpha_2) &= \pm ([\Sigma_g] - [S^2], -3[\Sigma_g] + [S^2]), \, \pm ([\Sigma_g] - [S^2], [\Sigma_g] + [S^2]) \\
	(\alpha_3, \alpha_4) &= \pm ([\Sigma_g], [\Sigma_g])
\end{align*}

By blowing down the sections $\sigma_{j,2}^n$ for $1 \leq j \leq 4$, the diffeomorphism $\Phi_n^2$ induces a diffeomorphism $M_n^2 \to M^2 = \Sigma_g \tilde\times S^2$ sending $[F^n_2]$ to $\alpha_2$, where we also denote by $F^n_2 \subseteq M_n^2$ the image of the genus-$g$ surface $F^n_2 \subseteq X_n$ under the blowdown $X_n \to M_n^2$. By work of Li--Li \cite[Theorem 2]{li-li-minimal-genus}, the smooth minimal genus of $\pm (-3 [\Sigma_g] +[S^2]) \in H_2(M^2; \Z)$ is $3g-1$. Therefore, $\alpha_2 \neq \pm (-3 [\Sigma_g] +[S^2])$ because $\alpha_2$ is represented by a smooth genus-$g$ surface and $g < 3g-1$ for all $g \geq 2$. Combining with the fact that $[F^n] = [F^n_1]+[F^n_2] = [F^n_3] + [F^n_4]$, we conclude that
\[
	(\alpha_1, \alpha_2, \alpha_3, \alpha_4) = \pm ([\Sigma_g] -[S^2], [\Sigma_g] + [S^2], [\Sigma_g], [\Sigma_g]).
\]
Therefore, after possibly replacing $\Phi_n^2$ with $\Psi^2 \circ \Phi^2_n$,
\[
	(\Phi_n^2)_*[F_j^n] = (\Phi_0^2)_*[F_j^0]
\]
for all $1 \leq j \leq 4$. This proves \ref{lem:phi-n-i} for $i = 2$.

The classes $\alpha_1$ and $\alpha_2$ form an $\R$-basis of $H_2(M^2; \R) \cong \R^2$, and so
\[
	\{(\Phi_n^2)_*[F^n_1], (\Phi_n^2)_*[F^n_2], E_1, E_2, E_3, E_4\}
\]
forms an $\R$-basis of $H_2(M^2 \# 4 \overline{\CP^2}; \R)$. Noting that $(\Phi_n^2)_*[\sigma_{j,2}^n] = E_j$ for all $1 \leq j \leq 4$ concludes the proof of \ref{lem:homology-xn} for $i = 2$.

Let $\beta \in H^2(M^2 \# 4 \overline{\CP^2}; \R)$ be any class such that
\[
	\langle \beta, \alpha_1 \rangle = \frac 14, \qquad \langle \beta, \alpha_3 \rangle = \langle \beta, \alpha_4 \rangle= \frac 12, \qquad \langle \beta, E_j \rangle = 0 \text{ for all }1 \leq j \leq 4
\]
which exists because $\alpha_1, \alpha_3, E_1, \dots, E_4$ are linearly independent in $H_2(M^2\#4 \overline{\CP^2};\R)$. Letting $[\nu_2] = (\Phi^2_0)^*(\beta)$, compute for all $1 \leq j \leq 4$ that
\begin{align*}
	\langle [\nu_2], [F_1^0] \rangle &= \frac 14, \qquad \langle [\nu_2], [F_2^0] \rangle = \frac 34, \qquad
	\langle [\nu_2], [F_3^0] \rangle = \langle [\nu_2], [F_4^0] \rangle = \frac 12, \\
	\langle [\nu_2], [F^0] \rangle &= \langle [\nu_2], [F_1^0] + [F_2^0] \rangle = 1,\qquad \langle [\nu_2], [\sigma_{j,2}^0] \rangle = 0,
\end{align*}
which proves \ref{lem:nu-i} for $i = 2$.
\end{proof}

\begin{rmk}[\textbf{Capping and forgetting points}]\label{capping}
  Recall the mapping classes in $\Mod(\Sigma_{2g}^4)$
\[
	\tilde h_i = T_{B_{0,2}^i} T_{B_{1,2}^i} \dots T_{B_{2g,2}^i} T_{C_2^i}, \qquad \tilde f = T_{\tilde x} T_{\tilde y}^{-1} T_{\tilde h_i(\tilde x)} T_{\tilde h_i(\tilde y)}^{-1}
\]
defined in Lemma \ref{lem:bounding-pair-sections}. We also denote by $\tilde h_i, \tilde f \in \Mod(\Sigma_{2g,4})$ the images under the homomorphism $\Mod(\Sigma_{2g}^4) \to \Mod(\Sigma_{2g, 4})$ that caps the four boundary components of $\Sigma_{2g}^4$ with disks with a marked point. Throughout the rest of this section we fix $i = 1$ or $i = 2$ and drop the index from our notation, i.e. let
\begin{align*}
	M &:= M^i, \quad [\nu] := [\nu_i]\in H^2(X_0; \R), \quad \sigma_j^n := \sigma_{j,i}^n \text{ for }1 \leq j \leq 4, \qquad \tilde h := \tilde h_i \in \Mod(\Sigma_{2g, 4}).
\end{align*}

It is sometimes more convenient to consider the image
\begin{equation}\label{eqn:hat-h}
	\hat h \in \Mod(\Sigma_{2g,2})
\end{equation}
of $\tilde h \in \Mod(\Sigma_{2g, 4})$ under the forgetful map $\Mod(\Sigma_{2g, 4}) \to \Mod(\Sigma_{2g, 2})$ that forgets the marked points $p_3, p_4$. We observe that $\hat h \in \Mod(\Sigma_{2g, 2})$ is independent of $i$ and has order $2$; see Corollary \ref{cor:puncture-involution} in Appendix \ref{appendix-a}.
\end{rmk}

Recall that the Lefschetz fibration $\pi_0$ was defined by the monodromy factorization
\[
	T_{B_0} T_{B_1} \dots T_{B_{2g}} T_C T_{B_0} T_{B_1} \dots T_{B_{2g}} T_C = 1 \in \Mod(\Sigma_{2g})
\]
in Definition \ref{defn:mck}. Let $q_{4g+4}, \dots, q_{1} \in S^2$ denote the singular values of $\pi_0: X_0 \to S^2$ corresponding to the vanishing cycles $B_0, B_1, \dots, B_{2g}, C, B_0, B_1, \dots, B_{2g}, C$ respectively. Let $\gamma_1, \dots, \gamma_{4g+4} \in \pi_1(S^2 -\{q_1, \dots, q_{4g+4}\}, b)$ be the chosen generators, each $\gamma_i$ a loop around $q_i$, that yields the above monodromy factorization for $\pi_0$.

As in Section \ref{sec:partial-conj}, write $S^2 = D_1 \cup_{\partial} D_2$ such that the basepoint $b \in S^2$ is contained in $E := \partial D_1$ and such that the loops $\gamma_1, \dots, \gamma_{2g+2}$ are contained in $D_1$, the loops $\gamma_{2g+3}, \dots, \gamma_{4g+4}$ are contained in $D_2$, and
\[
	[E] = \gamma_1 \dots \gamma_{2g+2}\in \pi_1(S^2 -\{q_1, \dots, q_{4g+4}\}, b).
\]
In other words, there is an isomorphism of $\Sigma_{2g}$-bundles over $S^1$
\[
	\pi_0^{-1}(E) \xrightarrow{\sim} M_{\eta}
\]
for any representative $\eta \in \Diff^+(\Sigma_{2g})$ of $h = T_{B_0} T_{B_1} \dots T_{B_{2g}} T_C \in \Mod(\Sigma_{2g})$ (cf. Theorem \ref{thm:eta-factorization}).

We first record an elementary computation of the homology of $\pi_0^{-1}(E)$.
\begin{lem}\label{lem:transfer}
Let $\eta \in \Diff^+(\Sigma_{2g})$ be any diffeomorphism of order $2$ representing $h \in \Mod(\Sigma_{2g})$. For any set of simple closed curves $\ell_{1}, \dots, \ell_{2g} \subseteq \Sigma_{2g}$ which together form an $\R$-basis of $H_1(\Sigma_{2g}; \R)^{\langle \eta \rangle}$, the homology classes
\[
	[\Sigma_{2g}], \quad [T_{\ell_1}], \dots, [T_{\ell_{2g}}]
\]
form an $\R$-basis of $H_2(M_{\eta}; \R)$, where $[\Sigma_{2g}]$ denotes the fiber class of $M_{\eta} \to S^1$ and $T_{\ell_i}$ is the torus
\[
	T_{\ell_i} = \ell_i \times [0, 1] \cup \eta(\ell_i) \times [0, 1]
\]
for each $i = 1, \dots, 2g$. In particular,
\[
	H_2(M_{\eta}; \R) \cong \R^{2g+1}.
\]
\end{lem}
\begin{proof}
We first show that $H_1(\Sigma_{2g}; \R)^{\langle \eta \rangle}$ is $2g$-dimensional. Consider the description of $h \in \Mod(\Sigma_{2g})$ in Theorem \ref{thm:eta-factorization} as the mapping class of the involution shown in Figure \ref{fig:mck-vanishing-cycles}. This description shows for all $1 \leq i \leq g$ that
\begin{align*}
	h(a_i - a_{2g+1-i}) &= a_i - a_{2g+1-i}, \qquad h(a_i + a_{2g+1-i}) = -(a_i + a_{2g+1-i}),\\
	h(b_i - b_{2g+1-i}) &= b_i - b_{2g+1-i}, \, \, \qquad h(b_i + b_{2g+1-i}) = -(b_i + b_{2g+1-i}),
\end{align*}
where the homology classes $a_i, b_i \in H_1(\Sigma_{2g}; \Z)$ are as shown in Figure \ref{fig:homology-curves}. The set
\[
	\{a_i - a_{2g+1-i},\,  a_i + a_{2g+1-i}, \, b_i - b_{2g+1-i}, \, b_i + b_{2g+1-i} : 1 \leq i \leq g \}
\]
is an $\R$-basis of $H_1(\Sigma_{2g}; \R)$, and hence
\[
	H_1(\Sigma_{2g}; \R)^{\langle h \rangle} = \R\{a_i - a_{2g+1-i}, \,b_i - b_{2g+1-i} : 1 \leq i \leq g \} \cong \R^{2g}.
\]

Next, we compute $H_2(M_{\eta}; \R)$. Consider the double cover
\[
	p: \Sigma_{2g} \times \R/\Z \to M_{\eta}
\]
which is the quotient of the order-$2$ action on $\Sigma_{2g} \times \R/\Z$ given by
\[
	\eta_0 : (x, t) \mapsto \left(\eta(x), t+\frac 12\right) \in M_{\eta}= (\Sigma_{2g} \times [0, 1])/((\eta(x),0) \sim (x, 1)).
\]
By transfer, the restriction
\[
	p_*: H_2(\Sigma_{2g} \times \R/\Z;\R)^{\langle \eta_0 \rangle} \to H_2(M_{\eta}; \R)
\]
is an isomorphism. The fixed subspace $H_2(\Sigma_{2g}\times \R/\Z; \R)^{\langle \eta_0 \rangle}$ has an $\R$-basis
\[
	\{[\Sigma_{2g}], [\ell_i \times \R/\Z] : 1 \leq i \leq 2g\}.
\]
Compute that $p_*([\Sigma_{2g}]) = [\Sigma_{2g}] \in H_2(M_{\eta}; \R)$ and
\[
	p_*([\ell_i \times \R/\Z]) = [T_{\ell_i}] \in H_2(M_{\eta}; \R). \qedhere
\]
\end{proof}

Using Lemma \ref{lem:transfer}, the following lemma computes the map on cohomology induced by the inclusion $\pi_0^{-1}(E) \hookrightarrow X_0$.
\begin{lem}\label{lem:tilde-eta-homology}
Let $\iota: \pi_0^{-1}(E) \hookrightarrow X_0$ denote the inclusion. The image of $\iota^*: H^2(X_0; \R) \to H^2(\pi_0^{-1}(E); \R)$ is $\R\{\PD([\sigma(E)])\}$, where $\sigma: S^2 \to X_0$ is any section of $\pi_0: X_0 \to S^2$.
\end{lem}
\begin{proof}
Write $X_0 = Y_1 \cup_{\pi_0^{-1}(E)} Y_2$ with
\[
	Y_1 = \pi_0^{-1}(D_1), \quad Y_2 = \pi_0^{-1}(D_2).
\]
Because $Y_i$ is a genus-$2g$ Lefschetz fibration over $D^2$ for each $i = 1, 2$, it is a finite union of $2$-handles $D^2 \times D^2$ attached to $\Sigma_{2g} \times D^2$, each with attaching regions $D^2 \times S^1$ \cite[Section 8.2]{gompf-stipsicz}. There is a Mayer--Vietoris sequence of the form
\[
	H^2(\sqcup_{j=1}^m D^2 \times S^1) \to H^3(Y_i) \to H^3(\Sigma_{2g} \times D^2) \oplus H^3(\sqcup_{j=1}^m (D^2 \times D^2))
\]
for some $m\geq 0$, which implies that $H^3(Y_i; \R) = 0$ for each $i = 1, 2$. Moreover, $H^3(X_0; \R) \cong \R^{2g}$ because $X_0$ is diffeomorphic to $(\Sigma_g \times S^2) \# 4\overline{\CP^2}$ by Proposition \ref{prop:Xn-diffeo}, and $H^2(\pi_0^{-1}(E); \R) \cong \R^{2g+1}$ by Lemma \ref{lem:transfer}. Combining these three computations with the Mayer--Vietoris sequence for $X_0 = Y_1 \cup_{\pi_0^{-1}(E)} Y_2$ yields
\[
	H^2(X_0; \R) \xrightarrow{(j_1)^* \oplus (j_2)^*} H^2(Y_1; \R) \oplus H^2(Y_2; \R) \xrightarrow{(k_1)^* - (k_2)^*}\underbrace{H^2(\pi_0^{-1}(E); \R)}_{\cong \R^{2g+1}} \to \underbrace{H^3(X_0; \R)}_{\cong \R^{2g}} \to 0
\]
where $j_i: Y_i \hookrightarrow X_0$ and $k_i: \pi_0^{-1}(E) \hookrightarrow Y_i$ are inclusion maps for $i = 1, 2$ such that $\iota = j_i \circ k_i$ for $i = 1, 2$. Therefore, the image of $(k_1)^*$ (and hence the image of $\iota^*$) is at most $1$-dimensional.

Observe that
\[
	\iota^*(\PD([\sigma(S^2)])) = \PD([\sigma(E)]),
\]
where $\PD([\sigma(S^2)])$ and $\PD([\sigma(E)])$ denote the Poincar\'e duals of $[\sigma(S^2)]$ and $[\sigma(E)]$ in $X_0$ and $\pi_0^{-1}(E)$ respectively. The class $\PD([\sigma(E)])$ is nonzero (e.g. because $\sigma(E)$ intersects a fiber $F$ in $\pi_0^{-1}(E)$ once transversely), and hence $\PD([\sigma(E)])$ spans the image of $\iota^*$, i.e.
\[
	\mathrm{im}(\iota^*) = \R\{\PD([\sigma(E)])\} \subseteq H^2(\pi_0^{-1}(E); \R). \qedhere
\]
\end{proof}

\subsection{Gompf--Thurston construction for $(X_n, \omega_n)$}

The goal of this subsection is to prove the following proposition which builds a suitable symplectic form $\omega_n$ on $X_n$. A key point of the construction below is the good control over the cohomology classes $[\omega_n] \in H^2(X_n; \R)$ for all $n \in \Z_{\geq 0}$.
\begin{prop}\label{prop:GT-Xn}
For any $n \in \Z_{\geq 0}$, let $F_1^n \cup F_2^n$ and $F_3^n \cup F_4^n$ denote the two reducible singular fibers of $\pi_n: X_n \to S^2$ with $F_1^n, F_2^n, F_3^n, F_4^n$ denoting the irreducible components as defined in Lemma \ref{lem:nu-choice}. For each $n \in \Z_{\geq 0}$, there exists a symplectic form $\omega_n$ on $X_n$ satisfying the following properties:
\begin{enumerate}[(a)]
\item The smooth loci of the irreducible components of the fibers of $\pi_n: X_n \to S^2$ are all symplectic submanifolds of $(X_n, \omega_n)$. \label{GT-Xn-fibers}
\item For all $n \in \Z_{\geq 0}$ and for all $1 \leq j \leq 4$,
\[
	\langle [\omega_n], [F_j^n] \rangle = \langle [\omega_0], [F_j^0] \rangle.
\]
In particular, $\langle [\omega_n], [F_j^n] \rangle$ is independent of $n$. \label{GT-Xn-areas}

\item \emph{(Standard near the sections)} For each $1 \leq j \leq 4$, there exists an open tubular neighborhood $N_j^n$
  of $\sigma_j^n(S^2)$ such that $\pi_n|_{N^n_j}: N^n_j \to S^2$ is isomorphic to the tautological line bundle $\pi_L: L \to \CP^1$ via an orientation-preserving diffeomorphism $\phi^n_j: N^n_j \to L$. Furthermore, for a small enough neighborhood $\sigma^n_j(S^2)\subseteq V^n_j \subseteq N^n_j$,
\[
	\omega_n|_{V^n_j} = (\phi^n_j)^*(a_j \pi_{\C^2}^* \omega_{\C^2} + \pi^*_L \omega_{S^2})
\]
where $\pi_{\C^2}: L \to \C^2$ is the standard blow-down map, $\omega_{\C^2}$ is the standard symplectic form on $\C^2$, $\omega_{S^2}$ is the Fubini--Study form on $S^2$, and $a_j$ is some positive constant \emph{independent} of $n$.
In particular, all sections $\sigma_j^n$ are symplectic and
\[ \langle [\omega_n], [\sigma_j^n(S^2)]\rangle = \pi \]
  is independent of $n$.\label{standard-alln}
\end{enumerate}
\end{prop}

In order to explicitly build a suitable form around $\pi_0^{-1}(E) \subseteq X_0$ and around each section $\sigma^n_j$, we will need to find nice representative diffeomorphisms of the monodromy $\hat h \in \Mod(\Sigma_{2g, 2})$ and $\tilde h \in \Mod(\Sigma_{2g,4})$, as well as compatible representatives of the conjugating mapping class $\tilde f \in \Mod(\Sigma_{2g,4})$. This is achieved in the following proposition, whose proof is given in Appendix~\ref{appendix-b}.

\begin{prop}\label{prop:isotoping-hateta}
Let $\alpha$, $\beta$, $\gamma$, and $\delta$ be disjoint curves in $\Sigma_{2g, 4}$ representing the isotopy classes $\tilde x, \tilde y, \tilde h(\tilde x), \tilde h(\tilde y)$ respectively. There exist
\begin{itemize}
\item a unit-area symplectic form $\theta$ on $\Sigma_{2g}$,
\item symplectomorphisms $\varphi \in \Diff^+(\Sigma_{2g, 4})$ and $\hat\eta\in \Diff^+(\Sigma_{2g, 2})$ of $(\Sigma_{2g}, \theta)$ with
\[
	[\varphi] = \tilde f \in \Mod(\Sigma_{2g, 4}), \qquad [\hat\eta] = \hat h \in \Mod(\Sigma_{2g, 2}), \qquad\text{ and }\qquad \hat\eta^2 = \Id_{\Sigma_{2g, 2}},
\]
\item an isotopy $\kappa_t: \Sigma_{2g, 2} \times [0, 1] \to \Sigma_{2g, 2}$ with $\kappa_0 = \Id_{\Sigma_{2g, 2}}$,
\item a disjoint union $W = W_\alpha \sqcup W_\beta \sqcup W_\gamma \sqcup W_\delta$ of tubular neighborhoods of $\alpha$, $\beta$, $\gamma$, $\delta$ in $\Sigma_{2g,4}$, and
\item a disjoint union $O = O_1 \sqcup O_2 \sqcup O_3 \sqcup O_4$ of neighborhoods of $p_1,p_2,p_3$ and $p_4$
\end{itemize}
satisfying the following properties:
\begin{enumerate}[(a)]
\item The symplectomorphism $\varphi$ is supported on $W$, i.e.
\[
	\varphi|_{\Sigma_{2g, 4} - W} = \Id|_{\Sigma_{2g, 4} - W}. \label{lem:eta-construction-symplectic}
      \]
\item The symplectomorphisms $\varphi$ and $\hat \eta$ commute, i.e.
\[
	\varphi \circ \hat\eta = \hat\eta \circ \varphi.
\label{lem:eta-construction-commutes}
\]
\item The diffeomorphism $\tilde\eta := \hat\eta \circ \kappa_1$ fixes $O$ pointwise and satisfies
\[
	[\tilde\eta] = \tilde h \in \Mod(\Sigma_{2g, 4}).\label{lem:eta-construction-four-points}
\]
\item For all $t\in [0,1]$
  \[ \kappa_t(O) \cap W = \emptyset. \]
  In particular, $W \cap O = \emptyset$. \label{lem:eta-construction-nice-iso}
\end{enumerate}
\end{prop}
\noindent \textbf{Isomorphisms near the equator $E$.} For the remainder of this section, fix the notation of Proposition \ref{prop:isotoping-hateta}. Let $N\cong E \times (-2, 2)$ be a tubular neighborhood of $E$ in $S^2$ disjoint from the set of singular values of $\pi_0$. Shrinking $N$ yields another tubular neighborhood $A \cong E \times (-1, 1)$ of $E$ in $S^2$.

The four marked points $p_1, p_2, p_3, p_4$ fixed by $\tilde\eta \in \Diff^+(\Sigma_{2g, 4})$ define sections $\tilde s_i$ of $M_{\tilde \eta} \times (-2, 2) \to S^1 \times (-2, 2)$ for $1 \leq i \leq 4$ by
\begin{equation}\label{eqn:constant-sections-equator}
	\tilde s_i: (t, s) \mapsto ((p_i,t), s) \in M_{\tilde\eta} \times (-2, 2),
\end{equation}
where we recall that the mapping torus $M_{\tilde\eta}$ is identified with
\[
  M_{\tilde\eta} = (\Sigma_{2g, 4} \times [0, 1])/ ((\tilde\eta(x), 0) \sim (x, 1)).
\]

There exists an isomorphism of $\Sigma_{2g}$-bundles over $S^1 \times (-2, 2)$
\begin{equation}\label{eqn:equator-isomorphism}
	\tilde{\mathcal G}: \pi_0^{-1}(N) \to M_{\tilde\eta} \times (-2, 2)
\end{equation}
that sends each section $\sigma_i$, $1 \leq i \leq 4$ of $\pi_0$ to the section $\tilde s_i$ of $M_{\tilde\eta} \times (-2, 2) \to S^1 \times (-2, 2)$.

The isotopy $\kappa_t: \Sigma_{2g, 2} \times [0, 1] \to \Sigma_{2g, 2}$ defines a map of $\Sigma_{2g}$-bundles
\begin{align*}
	\mathcal K: M_{\tilde\eta} \times (-2, 2) &\to M_{\hat\eta}, \\
	((x, t), s) &\mapsto (\kappa_t(x), t).
\end{align*}
Finally, define a map of $\Sigma_{2g}$-bundles $\hat{\mathcal G}$ to be the composition
\begin{equation}\label{eqn:hat-G}
	\hat{\mathcal G} := \mathcal K \circ \tilde{\mathcal G}: \pi_0^{-1}(N) \to M_{\hat\eta}
\end{equation}
Because $\kappa_t$ fixes the points $p_1, p_2$ for all $t \in [0, 1]$, the map $\hat{\mathcal G}$ sends each section $\sigma_1$ and $\sigma_2$ of $\pi_0$ to the section $\hat s_1$ and $\hat s_2$ of $M_{\hat \eta} \to S^1$ respectively, where $\hat s_i$ is the section defined by the point $p_i$ fixed by $\hat\eta \in \Diff^+(\Sigma_{2g, 2})$ for $i = 1, 2$.

The pullback $\mathrm{pr}_1^*\theta$ of $\theta\in \Omega^2(\Sigma_{2g})$ under the projection $\mathrm{pr}_1: \Sigma_{2g} \times \R \to \Sigma_{2g}$ is invariant under deck transformations of $\Sigma_{2g} \times \R \to M_{\hat \eta}$, and so induces a closed $2$-form $\hat\theta$ on $M_{\hat\eta}$ that restricts to the symplectic form $\theta$ on each fiber of $M_{\hat\eta} \to S^1$.

In order to use the form $\hat\theta$ in the Gompf--Thurston construction, we must determine its cohomology class in $H^2(M_{\hat\eta}; \R)$.
\begin{lem}\label{lem:theta-cohomology}
The cohomology class of the form $\hat\theta$ is
\[
	[\hat\theta] = \PD([\hat s_1(S^1)]) \in H^2(M_{\hat\eta}; \R).
\]
\end{lem}
\begin{proof}
In order to determine the class $[\hat\theta]$, we evaluate $[\hat\theta]$ on $H_2(M_{\hat\eta}; \R)$. Let $\ell_1, \dots, \ell_{2g} \subseteq \Sigma_{2g}$ be simple closed curves which together span $H_1(\Sigma_{2g}; \R)^{\langle \hat \eta \rangle}$. We may assume that each curve $\ell_i$ is disjoint from the marked point $p_1 \in \Sigma_{2g}$ which is fixed by $\hat\eta$ and defines the section $\hat s_1: S^1 \to M_{\hat\eta}$. For each $1 \leq i \leq 2g$, consider the torus $T_{\ell_i}$ found in Lemma \ref{lem:transfer}. We claim that $\hat \theta$ vanishes on $T_{\ell_i}$ for all $1 \leq i \leq 2g$. Indeed, $T_{\ell_i} \subseteq M_{\hat\eta}$ is the image of
\[
	\ell_i \times \R \subseteq \Sigma_{2g} \times \R
\]
under the cover $\Sigma_{2g} \times \R \to M_{\hat\eta}$. Because $\mathrm{pr}_1(\ell_i \times \R) = \ell_i \subseteq \Sigma_{2g}$ is $1$-dimensional,
\[
	(\mathrm{pr}_1^*\theta)|_{\ell_i \times \R} =0.
\]
Therefore $\hat\theta|_{T_{\ell_i}} = 0$ as well because $\mathrm{pr}_1^*\theta$ descends to $\hat\theta$ and $\ell_i \times \R$ descends to $T_{\ell_i}$ under $\Sigma_{2g} \times \R \to M_{\hat\eta}$. In other words,
\[
	\langle [\hat \theta], [T_{\ell_i}] \rangle = 0
\]
for all $1 \leq i \leq 2g$.

On the other hand,
\[
	\langle [\hat\theta], [\Sigma_{2g}] \rangle = \int_{\Sigma_{2g}} \hat\theta|_{\Sigma_{2g}} = \int_{\Sigma_{2g}} \theta = 1
\]
for any fiber $\Sigma_{2g}$ of $M_{\hat\eta}\to S^1$.

Now we analyze $\PD([\hat s_1(S^1)])$ evaluated on $H_2(M_{\hat\eta}; \R)$. Observe that $T_{\ell_i} \subseteq M_{\hat\eta}$ and $\hat s_1(S^1) \subseteq M_{\hat\eta}$ are disjoint for all $1 \leq i \leq 2g$ because $p_1 \in \Sigma_{2g}$ and $\ell_i \subseteq \Sigma_{2g}$ are disjoint. Therefore,
\[
	\langle \PD([\hat s_1(S^1)]), [T_{\ell_i}] \rangle = 0
\]
for all $1 \leq i \leq 2g$.

Similarly, $\Sigma_{2g}$ and $\hat s_1(S^1)$ intersect positively once for any fiber $\Sigma_{2g}$ of $M_{\hat \eta} \to S^1$ because $\hat s_1$ is a section. Therefore,
\[
	\langle \PD([\hat s_1(S^1)]), [\Sigma_{2g}] \rangle = 1.
\]

Finally, Lemma \ref{lem:transfer} shows that the classes $[T_{\ell_1}], \dots, [T_{\ell_{2g}}]$ and $[\Sigma_{2g}]$ form a basis of $H_2(M_{\hat\eta}; \R)$. The forms $[\hat\theta]$ and $\PD([\hat s_1(S^1)])$ agree on this basis of $H_2(M_{\hat\eta}; \R)$, and
\[
	[\hat\theta] = \PD([\hat s_1(S^1)]) \in H^2(M_{\hat \eta}; \R). \qedhere
\]
\end{proof}

Below, we first build a symplectic form on $X_0$ whose restriction to a neighborhood of $\pi_0^{-1}(E)$ interacts well with the map $\varphi \in \Symp(\Sigma_{2g}, \theta)$ which will be used in the partial conjugation construction to build $X_n$. In order for our construction to also interact nicely with the symplectic blowing down procedure, we \emph{standardize} our symplectic form near the sections. 
\begin{prop}[{Gompf--Thurston construction for $\pi_0: X_0 \to S^2$}]\label{prop:GT-X0}
There exists a symplectic form $\omega_0$ on $X_0$ such that the smooth loci of the irreducible components of the fibers of $\pi_0: X_0 \to S^2$ and the sections $\sigma_1^0, \sigma_2^0, \sigma_3^0, \sigma_4^0$ of $\pi_0: X_0\to S^2$ are all symplectic submanifolds of $(X_0, \omega_0)$. Moreover, $\omega_0$ satisfies the following properties.
\begin{enumerate}[(1)]
\item \emph{(Standard near the sections)} For each $1 \leq j \leq 4$, let $U_j$ be disjoint
  neighborhoods of $\sigma_j^0$. Then,
  there exists an open tubular neighborhood $N_j\subset U_j$ of $\sigma_j^0$ such that $\pi_0|_{N_j}: N_j \to S^2$ is isomorphic to the tautological line bundle $\pi_L: L \to \CP^1$ via an orientation-preserving diffeomorphism $\phi_j: N_j \to L$. Furthermore, for a small enough neighborhood $\sigma_j^0\subseteq N_j' \subseteq N_j$,
\[
	\omega_0|_{N_j'} = \phi_j^*(a_j \pi_{\C^2}^* \omega_{\C^2} + \pi^*_L \omega_{S^2})
\]
where $\pi_{\C^2}: L \to \C^2$ is the standard blow-down map, $\omega_{\C^2}$ is the standard symplectic form on $\C^2$, $\omega_{S^2}$ is the Fubini--Study form on $S^2$, and $a_j$ is some positive constant. \label{standard-sections}
\item \emph{(Standard over the twisting region)} The restriction of $\omega_0$ to $\pi_0^{-1}(A)- \bigcup_{j=1}^4 N_j$ takes the form
\[
	\omega_0|_{\pi_0^{-1}(A) - \bigcup_{j=1}^4 N_j} = \varepsilon \hat{\mathcal G}^*\hat\theta|_{\pi_0^{-1}(A) - \bigcup_{j=1}^4 N_j}+ \pi_0^* \omega_{A}
\]
for some $\varepsilon > 0$ and some symplectic form $\omega_{A}$ on $A \subseteq S^2$.  \label{standard-twisting}
\end{enumerate}
\end{prop}
\begin{proof}
Recall that a class $[\nu] \in H^2(X_0; \R)$ was chosen in Lemma \ref{lem:nu-choice}\ref{lem:nu-i} so that $[\nu] \in H^2(X_0; \R)$ pairs positively with the irreducible components of all fibers of $\pi_0: X_0 \to S^2$ and so that $\langle [\nu], [F] \rangle = 1$ for any smooth fiber $F$ of $\pi_0$. Given such a class $[\nu]$, the Gompf--Thurston construction (\cite[Proof of Theorem 10.2.18]{gompf-stipsicz} or \cite[Theorem 2.7]{gompf-symp}) yields a closed $2$-form $\zeta$ on $X_0$ with $[\zeta] = [\nu]$ such that
\begin{enumerate}[(a)]
\item $\zeta|_{S}$ is symplectic for $S$ the smooth locus of an irreducible component of any fiber of $\pi_0$, and
\item for any critical point $p \in X_0$ of $\pi_0$ and for some charts $U_p \subseteq \C^2$ and $V_{\pi_0(p)} \subseteq \C$ on which $\pi_0$ takes the form $(z, w) \mapsto z^2 + w^2$, the restriction $\zeta|_{U_p}$ is the standard symplectic form $dx_1 \wedge dy_1 + dx_2 \wedge dy_2$.
\end{enumerate}
Furthermore, recall that $[\nu] \in H^2(X_0; \R)$ was chosen in Lemma \ref{lem:nu-choice}\ref{lem:nu-i} so that $\langle [\nu], [\sigma_j^0] \rangle = 0$ for all $1 \leq j \leq 4$.

\medskip\noindent
\emph{Step 1: Standardizing over the annulus.} We will modify the form $\zeta$ over the tubular neighborhood $N\subseteq S^2$ of $E$ by the form $\hat{\mathcal G}^* \hat\theta$. Let $V_0 := N$. For each $i = 1, 2$, consider the open disks $V_i := D_i -(\bar A \cap D_i)$ in $S^2$. Then
\[
	S^2 = V_0 \cup V_1 \cup V_2
\]
is an open cover of $S^2$. Let $\{\rho_0, \rho_1, \rho_2\}$ be a partition of unity subordinate to this open cover. Consider closed $2$-forms $\xi_i$ on $V_i$ for each $i = 0, 1, 2$, defined by
\[
	\xi_0 := \hat{\mathcal G}^* \hat\theta, \qquad \xi_1 = \zeta|_{\pi_0^{-1}(V_1)}, \qquad \xi_2 = \zeta|_{\pi_0^{-1}(V_2)}.
\]
Then
\[
	[\xi_1] = [\nu|_{\pi_0^{-1}(V_1)}] \in H^2(\pi_0^{-1}(V_1); \R) \qquad \text{ and } \qquad [\xi_2] = [\nu|_{\pi_0^{-1}(V_2)}] \in H^2(\pi_0^{-1}(V_2); \R)
\]
by the Gompf--Thurston construction. On the other hand, $[\nu|_{\pi_0^{-1}(E)}] = a \PD([\sigma_1^0(E)]) \in H^2(\pi_0^{-1}(E); \R)$ for some $a \in \R$ by Lemma \ref{lem:tilde-eta-homology}, and $a = 1$ because $\langle [\nu|_{\pi_0^{-1}(E)}], [F] \rangle = 1$ for any fiber $F$ of $\pi_0^{-1}(E)$. By Lemma \ref{lem:theta-cohomology},
\[
	[(\hat{\mathcal G}^*\hat\theta)|_{\pi_0^{-1}(E)}] = \PD([\sigma_1^0(E)]) \in H^2(\pi_0^{-1}(E); \R)
\]
because $\hat{\mathcal G}$ sends the section $\sigma_1^0(E)$ to the section $\hat s_1(S^1)$ of $M_{\hat\eta}$. Finally because $H^2(\pi_0^{-1}(E); \R) \cong H^2(\pi_0^{-1}(N); \R)$,
\[
	[\xi_0] = [\nu|_{\pi_0^{-1}(V_0)}] \in H^2(\pi_0^{-1}(V_0); \R).
\]

For each $i = 0, 1, 2$, there exists a $1$-form $\beta_i$ on $\pi_0^{-1}(V_i)$ such that
\[
	\xi_i = \nu|_{\pi_0^{-1}(V_i)} + d\beta_i.
\]
Define the closed $2$-form $\xi$ on $X_0$ by
\[
	\xi := \nu + \sum_{i=0}^{2} d((\rho_i\circ \pi_0)\beta_i) = \sum_{i=0}^2 d(\rho_i\circ \pi_0) \wedge \beta_i + (\rho_i \circ \pi_0) \xi_i.
\]
We claim that $\xi$ restricts to a symplectic form on smooth loci $S$ of irreducible components of any fiber of $\pi_0$. To see this, note that $d(\rho_i \circ \pi_0)$ vanishes on $S$ for each $i = 0, 1, 2$, so
\[
	\xi|_S = \sum_{i=0}^2 (\rho_i \circ \pi_0)\xi_i |_{S}.
\]
If $S \subseteq \pi_0^{-1}(V_i)$ then $\xi_i|_S$ is symplectic for $i = 1, 2$ by the Gompf--Thurston construction. For $i = 0$, recall that the form $\hat\theta$ restricts to a symplectic form on each fiber of $M_{\tilde\eta} \to S^1$ by construction. Because $\hat{\mathcal G}$ maps a fiber of $\pi_0^{-1}(N) \to N$ diffeomorphically onto a fiber of $M_{\tilde\eta} \to S^1$, the pullback $\hat{\mathcal G}^*\hat\theta$ restricted to a fiber of $\pi_0^{-1}(N) \to N$ is still symplectic. Therefore, $\xi|_S$ is symplectic on $S$ because it is a positive linear combination of volume forms on $S$.

\medskip\noindent
\emph{Step 2: Standardizing near the sections.} Observe that $\pi_0: X_0 \to S^2$ is a submersion on each section $\sigma_j^0$. Therefore, for a small enough tubular neighborhood $N_j \subseteq U_j$ of $\sigma_j^0$, the restriction $\pi_0|_{N_j}: N_j \to S^2$ is isomorphic\footnote{In fact, $N_j \to S^2$ can be chosen to be isomorphic
  to the bundle $TF \to \sigma_j^0$, where $TF \subset TU_j$ is the subbundle \emph{parallel} to the fibers of $\pi_0$.} to a rank $2$ vector-bundle over $S^2$. Since $[\sigma_j^0]^2 = -1$, there exists a $\C$-line bundle isomorphism $\phi_j: N_j \to L$. Define
\[
	\lambda_j := (\phi_j^{-1})^*\xi|_{N_j} \in \Omega^2(L).
\]
Since $\phi_j$ is orientation-preserving on fibers, the form $\lambda_j$ tames the standard complex structure $i$ of $L$ in the fiber direction of $\pi_L: L \to \CP^1$. Let
\[
	L(1) := \{v \in L: \lvert \pi_{\C^2}(v) \rvert^2 \leq 1\}
\]
where $\lvert \cdot \rvert$ denotes the standard, K\"ahler metric compatible with $\omega_{\C^2}$. The following is a modification of a lemma of McDuff--Polterovich \cite[Lemma 5.5.B]{mcduff-polterovich}.
\begin{lem}[{cf. McDuff--Polterovich \cite[Lemma 5.5.B]{mcduff-polterovich}}]
There exists a $2$-form $\lambda_j' \in \Omega^2(L(1))$ satisfying the following properties:
\begin{enumerate}[(1)]
\item $\lambda_j'$ agrees with $\lambda_j$ near the boundary of $L(1)$.
\item $\lambda_j' = c_j \pi^*_{\C^2} \omega_{\C^2}$ near the zero section of $\pi_L: L\to \CP^1$ for some constant $c_j > 0$.
\item $\lambda_j'$ tames $i$ in the fiber direction of $\pi_L: L \to \CP^1$.
\end{enumerate}
\end{lem}
\begin{proof}
This is the same computation as \cite[Lemma 5.5.B]{mcduff-polterovich}; for completeness, we recall the proof. Let $B(r) \subseteq \C^2$ denote the open ball of radius $r$ in $\C^2$ with respect to the K\"ahler metric compatible with $\omega_{\C^2}$. By McDuff--Polterovich \cite[Proof of Lemma 5.5.B]{mcduff-polterovich}, there exists a K\"ahler form $\tau_k$ on $B(1)$ such that $\tau_k = \varepsilon^2 \omega_{\C^2}$ near the boundary of $B(1)$ anda $\tau_k = k^2 \omega_{\C^2}$ on $B(\varepsilon/(2k))$, for any $k > 1$ and $0 < \varepsilon < 1$.

Because $[\xi] = [\nu] \in H^2(X_0; \R)$ and because $\langle [\xi|_{N_j}], [\sigma_j^0] \rangle = 0$ by construction of $\nu$ in Lemma \ref{lem:nu-choice}\ref{lem:nu-i}, the form $\lambda_j \in \Omega^2(L)$ is exact and $\lambda_j = d\beta_j$ for some $\beta_j \in \Omega^1(L)$. Let $0 < \varepsilon < 1$ be such that $\lambda_j - \varepsilon^2 \pi_{\C^2}^* \omega_{\C^2}$ tames $i$ in the fiber direction of $\pi_L: L(1) \to \CP^1$. Consider a bump function $\rho: \C^2 \to \R_{\geq 0}$ supported in $B(1) \subseteq \C^2$ such that $\rho \equiv 1$ near $0 \in B(1)$ and $\rho\equiv 0$ near $\partial B(1)$. Define $\rho_k: L \to \R_{\geq 0}$ to be $\rho_k(z) = \rho((2k/\varepsilon) \pi_{\C^2}(z))$. Furthermore, define
\[
	\lambda_j' = \lambda_j + \pi_{\C^2}^* (\tau_k - \varepsilon^2 \omega_{\C^2}) - d(\rho_k \beta_j).
\]
For $k \gg 1$ large enough, the lemma follows.
\end{proof}

Since the $U_j$ are disjoint, it follows that we can modify $\xi$ only over $\bigsqcup_{j=1}^4N_j$ such that $\xi$ remains symplectic on the fibers and
\[
	\xi|_{N_j'} = \phi_j^*(c_j\pi_{\C^2}^* \omega_{\C^2})
\]
on a small enough neighborhood $\sigma_j^0 \subseteq N_j' \subseteq N_j$.

\medskip
\noindent
\emph{Step 3: Finishing the proof.} Let $\omega_{S^2}$ be the Fubini--Study form on $S^2$. For any $\varepsilon > 0$ small enough,
\[
	\omega_0 = \varepsilon\xi + \pi_0^* \omega_{S^2}
\]
is symplectic and the sections $\sigma_j^0$ are symplectic by compactness \cite[Proposition 10.2.20]{gompf-stipsicz}. Moreover, $\omega_0$ restricts to a symplectic form on the smooth locus $S$ of any fiber of $\pi_0$ because $\varepsilon\xi$ does and $\pi_0^*\omega_{S^2}|_S = 0$. Near each section $\sigma_j^0$,
\[
	\omega_0|_{N_j'} = \varepsilon \xi|_{N_j'} + \pi_0^* \omega_{S^2}|_{N_j'} = \varepsilon\phi_j^*(c_j \pi_{\C^2}^*\omega_{\C^2}) + \phi_j^*(\pi_L^*\omega_{S^2})
\]
where the last equality follows because $\pi_L \circ \phi_j = \pi_0$ restricted to $N_j'$. This proves \ref{standard-sections}, with $a_j = \varepsilon c_j$.

Finally, observe that $\rho_1|_A = \rho_2|_A \equiv 0$ and $\rho_0|_A \equiv 1$ because $A$ is disjoint from both $V_1$ and $V_2$. Therefore,
\[
	\xi|_{\pi_0^{-1}(A)} = \xi_0|_{\pi_0^{-1}(A)} = \hat{\mathcal G}^* \hat\theta|_{\pi_0^{-1}(A)},
\]
and so $\omega_0|_{\pi_0^{-1}(A) - \bigcup_{j=1}^4 N_j} = \varepsilon \hat{\mathcal G}^*\hat\theta|_{\pi_0^{-1}(A)- \bigcup_{j=1}^4 N_j} + \pi_0^*\omega_A$ where $\omega_A := \omega_{S^2}|_{A}$. This proves \ref{standard-twisting}.
\end{proof}

\begin{proof}[Proof of Proposition \ref{prop:GT-Xn}]
We will build an appropriate symplectic form $\omega_n$ on $X_n$ by a partial conjugation construction (Section \ref{sec:partial-conj}).

\medskip\noindent
\emph{Step 1: Explicitly constructing $\pi_n: X_n \to S^2$ and $\sigma_j^n: S^2 \to X_n$.} To find the gluing diffeomorphism used in the partial conjugation construction, consider the symplectomorphism $\varphi \in \Symp(\Sigma_{2g}, \theta)$ fixing the four
marked points $p_1, p_2, p_3, p_4$ with $[\varphi] = \tilde f \in \Mod(\Sigma_{2g, 4})$ found in Proposition \ref{prop:isotoping-hateta}. Recall
that there exist disk neighborhoods $O_j$ of each $p_j$ on which both $\varphi$ and $\tilde \eta$ act trivially.
In particular, for each $1\leq j \leq 4$ the set $O_j \times [0,1] \times (-2,2)$ induces a neighborhood
$\widetilde{O}_j$ of the section $\tilde s_i: E \times (-2,2) \to M_{\tilde\eta} \times (-2,2)$ defined in (\ref{eqn:constant-sections-equator}).

With the isotopy $\kappa_t$ found in Proposition \ref{prop:isotoping-hateta}, build a new isotopy
\[
	\kappa_t^{-1} \circ \varphi \circ \kappa_t: \Sigma_{2g, 4} \times [0, 1] \to \Sigma_{2g, 4},
\]
where we note that $\kappa_t^{-1} \circ \varphi \circ \kappa_t$ acts trivially on $O_1,O_2,O_3,O_4$ for all $t \in [0, 1]$ because $\kappa_t(O_j)$ is not contained in the support of $\varphi$ for all $t \in [0, 1]$ and for all $1 \leq  j \leq 4$. Moreover, compute at $t = 0, 1$ that
\begin{align*}
\kappa_0^{-1} \circ \varphi \circ \kappa_0 &= \Id_{\Sigma_{2g}}^{-1} \circ \varphi \circ \Id_{\Sigma_{2g}} = \varphi \\
\kappa_1^{-1} \circ \varphi \circ \kappa_1 &= (\hat\eta \circ \tilde\eta)^{-1} \circ \varphi \circ (\hat\eta \circ \tilde\eta) = \tilde\eta^{-1} \circ \varphi \circ \tilde \eta,
\end{align*}
where the last equality uses the fact that $\hat\eta$ and $\varphi$ commute. This isotopy induces a $\Sigma_{2g}$-bundle automorphism
\begin{align*}
	\tilde{\mathcal F} : M_{\tilde \eta} \times (-2, 2) &\to M_{\tilde \eta} \times (-2, 2) \\
	((x, t), s) &\mapsto ((\kappa_t^{-1} \circ \varphi \circ \kappa_t(x), t), s)
\end{align*}
that fixes pointwise each neighborhood $\widetilde O_j$ of each section $\tilde s_j$ of $M_{\tilde \eta} \times (-2,2)$.

Recall the isomorphism $\tilde{\mathcal G}: \pi_0^{-1}(N) \to M_{\tilde \eta} \times (-2, 2)$ fixed in (\ref{eqn:equator-isomorphism}). Define the bundle isomorphism $\mathcal F: \pi_0^{-1}(N) \to \pi_0^{-1}(N)$ by $\tilde{\mathcal F}$, using the identification $\tilde{\mathcal G}$
\[
	\mathcal F :=\tilde{\mathcal G}^{-1} \circ \tilde{\mathcal F} \circ \tilde{\mathcal G}: \pi_0^{-1}(N) \to \pi_0^{-1}(N).
\]
By construction, $\mathcal F$ fixes each section $\sigma_j^0: \pi_0^{-1}(N) \to \pi_0^{-1}(N)$ for $1 \leq j \leq 4$.

Let $U_1 := D_1 \cup A \subseteq S^2$ and $U_2 := D_2 \cup A \subseteq S^2$ be two open disks covering $S^2$ such that
\[
	U_1 \cap U_2 = A.
\]

By identifying any $x \in \pi_0^{-1}(A) \subseteq \pi_0^{-1}(U_1)$ with $\mathcal F^n(x) \in \pi_0^{-1}(A) \subseteq \pi_0^{-1}(U_2)$, form the Lefschetz fibration
\[
	\pi_n: X_n := \pi_0^{-1}(U_1) \cup_{\mathcal F^n: \pi_0^{-1}(A) \to \pi_0^{-1}(A)} \pi_0^{-1}(U_2) \to S^2
\]
where $\pi_n$ is defined to agree with $\pi_0$ on each $\pi_0^{-1}(U_1)$, $\pi_0^{-1}(U_2)$. Let $\sigma_j^n$ be the section of $\pi_n$ defined to agree with $\sigma_j^0$ on each $U_1$, $U_2$, for $1 \leq j \leq 4$. By the partial conjugation construction (Section~\ref{sec:partial-conj}), the Lefschetz fibration $\pi_n$ and sections $\sigma_1^n, \dots, \sigma_4^n$ have monodromy factorization given in (\ref{eqn:sections-monodromy}) as desired.

\medskip \noindent \emph{Step 2: A common neighborhood for $\sigma_j^n$:} Via $\tilde{\mathcal G}$ the neighborhoods
$\widetilde O_j$ induce neighborhoods $\mathfrak O_j$ of each section $\sigma_j^0$ over $\pi_0^{-1}(A)$ for each $1\leq j\leq 4$
on which $\mathcal F$ acts trivially. By \Cref{prop:GT-X0}\ref{standard-sections}, for each $1\leq j \leq 4$
there exists a tubular neighborhood $N_j|_{\pi^{-1}_0(A)} \subset \mathfrak O_j$ of
$\sigma_j^0$ over which the symplectic form $\omega_0$ is standard near $\sigma_j^0$. In particular, $N_j$ \emph{embeds} in  $X_n$ for all $n$ and $1\leq j \leq 4$; denote this natural embedding by $g_j^n:N_j \to X_n$. Furthermore, note that $\pi_n \circ g^n_j = \pi_0$, so $g^n_j$ is a map of bundles. Denote by $N_j^n$ the embedded copy of $N_j$ in $X_n$.

\medskip\noindent
\emph{Step 3: Constructing $\omega_n$ on $X_n$.} Consider the symplectic forms
\[
	\omega_0|_{\pi_0^{-1}(U_1)} \in \Omega^2(\pi_0^{-1}(U_1)), \qquad \omega_0|_{\pi_0^{-1}(U_2)} \in \Omega^2(\pi_0^{-1}(U_2))
\]
found in Proposition \ref{prop:GT-X0}. We claim that
\[
	\omega_n := \begin{cases}
	\omega_0|_{\pi_0^{-1}(U_1)}&\text{ on }\pi_0^{-1}(U_1) \subseteq X_n \\
	\omega_0|_{\pi_0^{-1}(U_2)}&\text{ on }\pi_0^{-1}(U_2) \subseteq X_n
	\end{cases}
\]
is a well-defined symplectic form on $X_n$. It suffices to check that $\omega_n$ is well-defined on $\pi_0^{-1}(A)$, i.e. that
\[
	\mathcal F^*\omega_0|_{\pi_0^{-1}(A)} = \omega_0|_{\pi_0^{-1}(A)}.
\]
This follows from \Cref{prop:isotoping-hateta}, \Cref{prop:GT-X0}\ref{standard-twisting}, and the definition of $\mathcal F$.

\medskip\noindent
\emph{Step 4: Restricting $\omega_n$ to the fibers of $\pi_n$.} Any fiber of $\pi_n$ is contained in $\pi_0^{-1}(U_1) \subseteq X_n$ or $\pi_0^{-1}(U_2) \subseteq X_n$. So for any $S$ the smooth locus of an irreducible component of any fiber of $\pi_n$,
\[
	\omega_n|_S = \omega_0|_S
\]
and so $S$ is a symplectic submanifold of $(X_n, \omega_n)$ by Proposition \ref{prop:GT-X0}. Furthermore, restricted to the components $F_1^n$, $F_2^n$, $F_3^n$, and $F_4^n$ of the reducible fibers of $\pi_n$,
\[
	\langle [\omega_n], [F_j^n] \rangle = \int_{F_j^n \subseteq X_n}  \omega_n|_{F_j^n} = \int_{F_j^0 \subseteq X_0} \omega_0|_{F_j^0} = \langle [\omega_0], [F_j^0] \rangle.
\]
This finishes the proof of \ref{GT-Xn-fibers} and \ref{GT-Xn-areas}.

\medskip \noindent
\emph{Step 5: Restricting $\omega_n$ near the sections $\sigma_j^n$.} Since $\mathcal F$ acts trivially on $N_j|_{\pi_0^{-1}(A)}$ it follows that
\[ g_j^n:(N_j,\omega_0) \to (N_j^n,\omega_n) \]
is a symplectomorphism. For each $1\leq j\leq 4$, let
$\phi_j:N_j \to L$ be the diffeomorphism found in \Cref{prop:GT-X0}\ref{standard-sections}, so that for $\sigma_j^0(S^2) \subseteq N_j' \subseteq N_j$ we have
  \[ \omega_0|_{N_j'} = \phi_j^*(a_j \pi_{\C^2}^* \omega_{\C^2} + \pi^*_L \omega_{S^2}).\]
  The map
$\phi_j^n:= \phi_j \circ (g^n_j)^{-1}$ is the desired isomorphism $N^n_j \to L$, with $V_j^n:= g_j^n(N_j')$, making
$\omega_n$ standard near $\sigma^n_j$. Note that under $\phi^n_j$, the section $\sigma_j^n$ corresponds to the zero section of $L$, which we denote by $S$. Thus $\sigma^n_j$ is symplectic and
  \[ \langle [\omega_n],[\sigma_j^n(S^2)] \rangle = \langle [\omega_{S^2}],[S] \rangle = \pi .\]
This finishes the proof of \ref{standard-alln}. 
\end{proof}

\subsection{On the cohomology classes $[\omega_n] \in H^2(X_n; \R)$}

Finally, we combine the results of the previous subsections and determine the cohomology classes $[\omega_n] \in H^2(X_n; \R)$.
\begin{lem}\label{lem:symplectic-form-cohomology}
For all $n \in \Z_{\geq 0}$, there exists a diffeomorphism
\[
	\Phi_n: X_n \to M \# 4 \overline{\CP^2}
\]
satisfying the following properties.
\begin{enumerate}[(a)]
\item On homology,
\[
	(\Phi_n)_*[\sigma_j^n] = E_j \qquad \text{ and } \qquad (\Phi_n)_*([F^n]) = (\Phi_0)_*([F^0]) \in H_2(M \# 4 \overline{\CP^2}; \Z)
\]
for all $1 \leq j \leq 4$, where $E_1, \dots, E_4 \in H_2(M \# 4 \overline{\CP^2}; \Z)$ denote the exceptional divisors coming from each summand $\overline{\CP^2}$ and $F^n$ denotes any regular fiber of $\pi_n: X_n \to S^2$.
\item The cohomology class $(\Phi_n^{-1})^*([\omega_n])$ is independent of $n$, i.e.
\[
	(\Phi_n^{-1})^*([\omega_n]) = (\Phi_0^{-1})^*([\omega_0]) \in H^2(M \# 4 \overline{\CP^2}; \R).
\]
\end{enumerate}
\end{lem}
\begin{proof}
For each $n \in \Z_{\geq 0}$, let $\Phi_n$ be the diffeomorphism found in Lemma \ref{lem:nu-choice}. Let $F_1^n \cup F_2^n$ and $F_3^n \cup F_4^n$ denote the two reducible singular fibers of $\pi_n: X_n \to S^2$ with $F_1^n$, $F_2^n$, $F_3^n$, and $F_4^n$ denoting the irreducible components as defined in Lemma \ref{lem:nu-choice}. By construction of $\Phi_n$,
\[
	(\Phi_n)_*([F_j^n]) = (\Phi_0)_*([F_j^0]), \qquad (\Phi_n)_*([\sigma_{j}^n]) = (\Phi_0)_*([\sigma_{j}^0]) = E_j
\]
for all $1 \leq j \leq 4$. It also follows that
\[
	(\Phi_n)_*([F^n]) = (\Phi_n)_*([F^n_1] + [F^n_2]) = (\Phi_0)_*([F^0_1] + [F^0_2]) = (\Phi_0)_*([F^0])
\]
as desired.

By Proposition \ref{prop:GT-Xn},
\begin{align*}
\langle (\Phi_n^{-1})^*([\omega_n]), (\Phi_n)_*([F_j^n]) \rangle = \langle [\omega_n], [F_j^n] \rangle &= \langle [\omega_0], [F_j^0] \rangle = \langle (\Phi_0^{-1})^*([\omega_0]), (\Phi_0)_*([F_j^0]) \rangle \\
\langle (\Phi_n^{-1})^*([\omega_n]), (\Phi_n)_*([\sigma_{j}^n]) \rangle = \langle [\omega_n], [\sigma_{j}^n] \rangle &= \langle [\omega_0], [\sigma_{j}^0] \rangle = \langle (\Phi_0^{-1})^*([\omega_0]), (\Phi_0)_*([\sigma_{j}^0]) \rangle
\end{align*}
for each $1 \leq j \leq 4$. By Lemma \ref{lem:nu-choice}\ref{lem:homology-xn}, the cohomology classes $(\Phi_n^{-1})^*([\omega_n])$ and $(\Phi_0^{-1})^*([\omega_0])$ agree on $H_2(M \# 4 \overline{\CP^2}; \R)$, and so
\[
	(\Phi_n^{-1})^*([\omega_n]) = (\Phi_0^{-1})^*([\omega_0]) \in H^2(M \# 4 \overline{\CP^2}; \R). \qedhere
\]
\end{proof}

\begin{proof}[{Proofs of Theorems \ref{thm:symplectic-LP} and \ref{thm:symplectic-LF}}]
By Proposition \ref{prop:GT-Xn}, the smooth loci of the irreducible components of the fibers and the sections $\sigma_1^n$, $\sigma_2^n$, $\sigma_3^n$, and $\sigma_4^n$ of $\pi_n: X_n \to S^2$ are all symplectic submanifolds of $(X_n, \omega_n)$. Let $\Phi_n: X_n \to M \# 4 \overline{\CP^2}$ be the diffeomorphisms found in Lemma \ref{lem:symplectic-form-cohomology} so that
\[
	(\Phi_n^{-1})^*([\omega_n]) = (\Phi_n^{-1})^*([\omega_0]) \in H^2(M \# 4 \overline{\CP^2}; \R)
\]
and
\[
	(\Phi_n)_*([F^n]) = (\Phi_0)_*([F^0])
\]
for all $n \in\Z_{\geq 0}$. Because each $\sigma_1^n$, $\sigma_2^n$, $\sigma_3^n$, and $\sigma_4^n$ are symplectic submanifolds of self-intersection $-1$, the form $\omega_n$ induces a symplectic form (which we also denote by $\omega_n$) on the blowdown $M_n$ of $\sigma_1^n$, $\sigma_2^n$, $\sigma_3^n$, and $\sigma_4^n$ in $X_n$.
Denote also by $\pi_n:M_n - B_n \to S^2$ the induced Lefschetz pencil. Note that since each $\omega_n$ is standard (see \Cref{prop:GT-Xn}\ref{standard-alln}) near each section $\sigma_j^n$,
we can ensure that the smooth loci of the irreducible components of the fibers of $\pi_n$ are symplectic submanifolds~\cite[Section 5]{mcduff-polterovich}.

Let $F_n\subset M_n$ denote a regular fiber of $\pi_n$. Because $\Phi_n$ sends $[\sigma_j^n]$ to $E_j$, it induces a diffeomorphism $\Psi_n: M_n \to M$ such that
\[
	(\Psi_n^{-1})^* ([\omega_n]) = (\Psi_0^{-1})^*([\omega_0]) \in H^2(M; \R)
\]
and
\[
	(\Psi_n)_*([F_n]) = (\Psi_n)_*([F_m]) \in H_2(M; \Z)
\]
for all $n \in \Z_{\geq 0}$. Work of Lalonde--McDuff \cite[Theorem 1.1]{lalonde-mcduff} shows that cohomologous symplectic forms on ruled surfaces are diffeomorphic, and so there is an equality of forms
\[
	(\Psi_n^{-1})^*\omega_n = (\Psi_0^{-1})^*\omega_0
\]
for all $n \in \Z_{\geq 0}$ after possibly post-composing $\Psi_n$ with a diffeomorphism of $M$. Any diffeomorphism of $M$ that fixes the cohomology class of a symplectic form acts by the identity on $H_2(M; \Z)$ (cf. \cite[Theorem 3]{li-li-minimal-genus}), and so the equality
\[
	(\Psi_n)_*([F_n]) = (\Psi_n)_*([F_m]) \in H_2(M; \R)
\]
still holds. Letting
\[
	\omega = (\Psi_0^{-1})^*\omega_0
\]
concludes the proof of Theorem \ref{thm:symplectic-LP}.

Recall by construction in Proposition \ref{prop:GT-Xn} that the sections $\sigma_j^n$ and $\sigma_j^0$ have equal area in $(X_n, \omega_n)$ and $(X_0, \omega_0)$ for all $n \in \Z_{\geq 0}$, i.e.
\[
	\langle [\omega_n], [\sigma_j^n] \rangle = [\omega_0], [\sigma_j^0] \rangle.
\]

Symplectic blowups of ruled surfaces are determined up to isotopy by the areas of the exceptional divisors by work of McDuff \cite[Corollary 1.3]{mcduff1996symplectic}. Therefore,

\[
	(\Phi_n^{-1})^*\omega_n = (\Phi_0^{-1})^*\omega_0
\]
for all $n \in \Z_{\geq 0}$ as symplectic forms, after possibly post-composing $\Phi_n$ with a diffeomorphism (isotopic to the identity) of $M \# 4 \overline{\CP^2}$. Finally, letting
\[
	\omega = (\Phi_0^{-1})^*\omega_0
\]
concludes the proof of the symplectic portion of Theorem \ref{thm:symplectic-LF}.
\end{proof}

Finally, Theorems \ref{thm:infty-pencils-ruled} and \ref{thm:infty-fibrations} follow as immediate corollaries.

\begin{proof}[Proof of Theorem \ref{thm:infty-pencils-ruled}]
If $X$ is a ruled surface with $\chi(X) < 0$, it is difeomorphic to either $M^1 = \Sigma_g \times S^2$ or $M^2 = \Sigma_g \tilde\times S^2$ for $g \geq 2$. Let $i = 1$ or $i = 2$ such that $X \cong M^i$ and fix the notation of Theorem \ref{thm:symplectic-LP}. For all $n \in \Z_{\geq 0}$, 
\[
	\pi_{n,i} \circ (\Psi_n^i)^{-1}: M^i - \Psi_n^i(B_n^i) \to S^2
\]
are all Lefschetz pencils of genus $2g$ with four base points that are compatible with $\omega$, whose regular fibers are all homologous in $H_2(M^i; \Z)$.
\end{proof}

\begin{proof}[Proof of Theorem \ref{thm:infty-fibrations}]
If $X$ is a ruled surface with $\chi(X) < 0$, it is difeomorphic to either $M^1 = \Sigma_g \times S^2$ or $M^2 = \Sigma_g \tilde\times S^2$ for $g \geq 2$, and $X \# 4 \overline{\CP^2}$ is diffeomorphic to $(\Sigma_g \times S^2) \# 4 \overline{\CP^2}$ in either case. Fix the notation of Theorem \ref{thm:symplectic-LF}. For all $n \in \Z_{\geq 0}$,
\[
	\pi_{n} \circ \Phi_n^{-1}: (\Sigma_g \times S^2) \# 4 \overline{\CP^2} \to S^2
\]
are all Lefschetz fibrations of genus $2g$ that are compatible with $\omega$, whose regular fibers are all homologous in $H_2((\Sigma_g \times S^2) \# 4\overline{\CP^2}; \Z)$.
\end{proof}


\section{Infinitely many homeomorphic Lefschetz fibrations}\label{sec:homeo}

Our construction of infinitely many pairwise inequivalent Lefschetz fibrations can apply to settings outside of ruled surfaces. In this section we demonstrate this flexibility through an example of infinitely many Lefschetz fibrations of every genus $g \geq 3$ that are pairwise inequivalent but are pairwise \emph{homeomorphic}.

Consider the isotopy classes $c_1, \dots, c_{2g} \subseteq \Sigma_g^1$ of curves shown in Figure \ref{fig:hyperelliptic-chain}. Let $\delta$ denote the boundary of $\Sigma_g^1$.
\begin{figure}
  \includegraphics[width=0.4\textwidth]{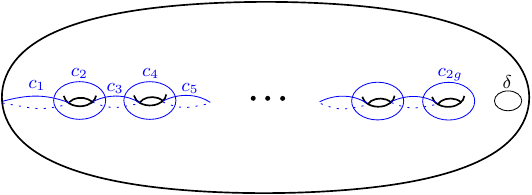}\qquad
  \includegraphics[width=0.4\textwidth]{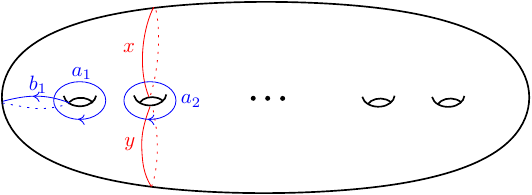}
  \caption{Left: A chain of curves in $\Sigma_g^1$; Right: Homologous curves $x,y \subseteq \Sigma_g$ and homology classes $a_1, a_2, b_1 \in H_1(\Sigma_g; \Z)$}.\label{fig:hyperelliptic-chain}
\end{figure}

\begin{lem}[{Chain Relation \cite[Proposition 4.12]{farb-margalit}}]\label{lem:chain}
For any $g \geq 2$,
\[
	(T_{c_{1}} \dots T_{c_{2g}})^{4g+2} = T_\delta \in \Mod(\Sigma_g^1).
\]
\end{lem}

Let $f = T_x T_y^{-1} \in \Mod(\Sigma_g)$ be a bounding pair of genus $1$, where $x, y \subseteq \Sigma_g$ are as shown in Figure \ref{fig:hyperelliptic-chain}. Conjugating the factorization given by Lemma \ref{lem:chain} by $f^n$ for any $n \in \Z_{\geq 0}$, shows that
\[
	(T_{f^n(c_{1})} \dots T_{f^n(c_{2g})})^{4g+2} = T_\delta \in \Mod(\Sigma_g^1).
\]
By combining the above with the factorization of Lemma \ref{lem:chain}, we define the genus-$g$ Lefschetz fibrations $\pi_n: Z_n \to S^2$.
\begin{defn}
Let $g \geq 3$. For any $n \in \Z_{\geq 0}$, let $\pi_n: Z_n \to S^2$ denote the Lefschetz fibration of genus $g$ with monodromy factorization
\[
	(T_{c_{1}} \dots T_{c_{2g}})^{2(4g+2)}  (T_{f^n(c_{1})} \dots T_{f^n(c_{2g})})^{4g+2} = 1 \in \Mod(\Sigma_g).
\]
\end{defn}

Note that $\pi_n: Z_n \to S^2$ has a section of self-intersection $-3$ for all $n \geq 0$ determined by the following lift of the monodromy factorization of $\pi_n$ to $\Mod(\Sigma_g^1)$ given by
\begin{equation}\label{eqn:monodromy-Zn-section}
	(T_{c_{1}} \dots T_{c_{2g}})^{2(4g+2)}  (T_{f^n(c_{1})} \dots T_{f^n(c_{2g})})^{4g+2} = T_\delta^3 \in \Mod(\Sigma_g^1).
\end{equation}

The following theorem is the main result of this section.
\begin{thm}\label{thm:homeo-not-iso}
For every $g \geq 3$, there exist genus-$g$ Lefschetz fibrations $\pi_n:Z_n \to S^2$ for every $n \in \Z_{\geq 0}$ such that that $\pi_n$ and $\pi_m$ are inequivalent if $n \neq m$ and such that $Z_n$ is homeomorphic to
\[
	(6g^2 - 2g + 1) \CP^2 \# (18g^2 + 10g + 1)\overline{\CP^2}
\]
for all $n \in \Z_{\geq 0}$.
\end{thm}

Similarly as in Section \ref{sec:johnson}, consider the following groups for any $n \in \Z_{\geq 0}$:
\begin{align*}
\hat G_n &:= \langle T_{c_i}, \, T_{f^n(c_i)} : 1 \leq i \leq 2g \rangle, \\
\hat G_n^{\mathcal I} &:= G_n \cap \mathcal I_g, \\
\hat A_n &:= \langle [T_{c_i}^{-1}, f^n] : 1 \leq i \leq 2g\rangle.
\end{align*}
We use the hat notation for groups in this section to distinguish them from the groups of Section \ref{sec:johnson}.

As in Section \ref{sec:johnson}, we first study the image $\tau(G_n^{\mathcal I})$ of $G_n^{\mathcal I}$ under the Johnson homomorphism. Let $a_1, a_2,b_1$ be as in \Cref{fig:hyperelliptic-chain} and define 
\[
	w := a_1 \wedge a_2 \wedge b_1 \in \left(\wedge^3 H\right)/H
\]
where $H := H_1(\Sigma_g; \Z)$.

\begin{lem}[{cf. Lemmas \ref{lem:GnI-generators}, \ref{lem:GnI-divisibility}, and \ref{lem:primitive-img}}]\label{lem:homeo-johnson}
For any $n \in \Z_{\geq 0}$, the image $\tau(\hat G_n^{\mathcal I})$ is contained in
\[
	\langle \Mod(\Sigma_g) \cdot (nw) \rangle \leq n\left(\wedge^3 H\right)/H,
\]
the group generated by the $\Mod(\Sigma_g)$-orbit of $nw \in \left(\wedge^3 H\right)/H$. Moreover, the class $nw$ is nonzero and is contained in $\tau(\hat G_n^{\mathcal I})$.
\end{lem}
\begin{proof}
For any $i \neq 4$, note that $[T_{c_i}^{-1}, f^n] = 1$ because $c_i$ is disjoint from $x$ and $y$. Therefore,
\[
	\hat A_n = \langle [T_{c_4}^{-1}, f^n] \rangle.
\]
By \cite[p. 195-196]{farb-margalit},
\[
	\tau(f) = \tau(T_x T_y^{-1}) = a_1 \wedge b_1 \wedge b_2,
\]
and compute using naturality (\ref{eqn:naturality}) of $\tau$ that
\[
	\tau([T_{c_4}^{-1}, f^n]) = n(T_{c_4}^{-1} \cdot \tau(f) - \tau(f)) = n ( a_1 \wedge b_1 \wedge T_{c_4}^{-1}(b_2) - a_1 \wedge b_1 \wedge b_2) = n(a_1 \wedge a_2 \wedge b_1) = nw.
\]
Because $[T_{c_4}^{-1}, f^n]$ is contained in $\hat G_n^{\mathcal I}$, the class $nw$ is contained in $\tau(\hat G_n^{\mathcal I})$. The fact that $nw \neq 0$ follows from the existence of a $\Z$-basis of the free abelian group $\left( \wedge^3 H_1(\Sigma_g) \right)/H_1(\Sigma_g)$ containing $w$, shown in Lemma \ref{lem:nonzero-wedge}.

On the other hand, the proof of Lemma \ref{lem:GnI-generators} shows that there is an equality of subgroups of $\mathcal I_g$
\[
	\hat G_n^{\mathcal I} = \langle \hat G_0^{\mathcal I}, k [T_{c_4}^{-1}, f^n] k^{-1} : k \in \hat G_0 \rangle.
\]
Each generator of $\hat G_0$ is hyperelliptic, and hence $\hat G_0^{\mathcal I}$ is contained in the hyperelliptic mapping class group $\SMod(\Sigma_g)$. By Corollary \ref{cor:hyperelliptic-kernel}, $\hat G_0^{\mathcal I}$ is contained in $\ker(\tau)$, and by naturality (\ref{eqn:naturality}) of $\tau$,
\[
	\tau(\hat G_n^{\mathcal I}) = \langle \hat G_0 \cdot \tau([T_{c_4}^{-1}, f^n]) \rangle \leq \langle \Mod(\Sigma_g) \cdot (nw) \rangle. \qedhere
\]
\end{proof}

The following non-conjugacy result forms the main part of the proof of Proposition \ref{prop:non-iso-Zn}.
\begin{prop}[{cf. Proposition \ref{prop:non-conj}}]
If $n \neq m \in \Z_{\geq 0}$ then $\hat G_n$ and $\hat G_m$ are not conjugate as subgroups of $\Mod(\Sigma_{g})$.
\end{prop}
\begin{proof}
Suppose that there exists $k \in \Mod(\Sigma_g)$ so that $k \hat G_n k^{-1} = \hat G_m$ as subgroups of $\Mod(\Sigma_g)$. Because $\mathcal I_g$ is normal in $\Mod(\Sigma_g)$, this implies that $k \hat G_n^{\mathcal I} k^{-1} = \hat G_m^{\mathcal I}$ as subgroups of $\mathcal I_g$. By naturality (\ref{eqn:naturality}) of $\tau$ and Lemma \ref{lem:homeo-johnson},
\[
\tau(\hat G_m^{\mathcal I}) = k \cdot \tau(\hat G_n^{\mathcal I}) \leq \langle\Mod(\Sigma_g) \cdot (nw) \rangle
\]
By Lemma \ref{lem:homeo-johnson}, the class $mw$ is nonzero and is contained in $\tau(\hat G_m^{\mathcal I})$. Since $w$ is primitive (\Cref{lem:nonzero-wedge}), this implies that $n$ divides $m$. By symmetry, $m$ also divides $n$, i.e. $n = m$.
\end{proof}

Rephrasing Proposition \ref{prop:non-conj} in terms of the monodromy of $\pi_n: Z_n \to S^2$ yields the following proposition.
\begin{prop}\label{prop:non-iso-Zn}
For any $n \neq m \in \Z_{\geq 0}$, the Lefschetz fibrations $\pi_n: Z_n \to S^2$ and $\pi_m: Z_m \to S^2$ are inequivalent.
\end{prop}
\begin{proof}
If $n \neq m \in \Z_{\geq 0}$ then $\hat G_n$ and $\hat G_m$ are not conjugate as subgroups of $\Mod(\Sigma_g)$. By construction, $\hat G_n$ and $\hat G_m$ are the images $\mathrm{im}(\rho_n)$ and $\mathrm{im}(\rho_m)$ of the monodromy representations $\rho_n$ and $\rho_m$ of $\pi_n: Z_n \to S^2$ and $\pi_m: Z_n\to S^2$ respectively. By Corollary \ref{cor:noniso-criterion}, $\pi_n$ and $\pi_m$ are inequivalent Lefschetz fibrations.
\end{proof}

It remains to show that the manifolds $Z_n$ are pairwise homeomorphic. To do so, we compute the algebraic topology invariants of $Z_n$.
\begin{lem}\label{lem:Zn-algtop}
For any $n \in \Z_{\geq 0}$, the manifold $Z_n$ is simply-connected and
\[
	b_2^+(Z_n) = 6g^2 - 2g + 1, \qquad b_2^-(Z_n) = 18g^2 + 10g + 1.
\]
\end{lem}
\begin{proof}
Because $\pi_n: Z_n \to S^2$ admits a section, there is an isomorphism
\[
	\pi_1(Z_n) \cong \pi_1(\Sigma_g)/N(c_1, \dots, c_{2g}, f^n(c_1), \dots, f^n(c_{2g}))
\]
where $N(c_1, \dots, f^n(c_{2g}))$ denotes the subgroup of $\pi_1(\Sigma_g)$ normally generated by the vanishing cycles of $\pi_n: Z_n \to S^2$. On the other hand, let $\alpha_i, \beta_i \in \pi_1(\Sigma_g)$ for $1 \leq i \leq g$ be the generators of $\pi_1(\Sigma_g)$ as shown below:
\begin{center}
\includegraphics[width=0.45\textwidth]{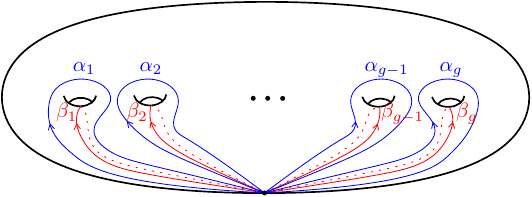}
\end{center}
The loop $\alpha_i \in \pi_1(\Sigma_g)$ is freely homotopic to $c_{2i}$ for all $1 \leq i \leq g$, the loop $\beta_1\in \pi_1(\Sigma_g)$ is freely homotopic to $c_{1}$, and the loop $\beta_{i}^{-1} \beta_{i+1} \in \pi_1(\Sigma_g)$ is freely homotopic to $c_{2i+1}$ for all $1 \leq i \leq g-1$. Therefore,
\[
	N(c_1, \dots, c_{2g}) = N(c_1, \dots, c_{2g}, f^n(c_1), \dots, f^n(c_{2g})) = \pi_1(\Sigma_{g}),
\]
and $Z_n$ is simply-connected for all $n \geq 0$.

There are $12g(2g+1)$-many vanishing cycles in $\pi_n: Z_n \to S^2$, and so
\[
	\chi(Z_n) = 4 - 4g + 12g(2g+1) = 24g^2 + 8g + 4.
\]
In other words, $b_2(Z_n) = 24g^2 + 8g + 2$.

In the case of $n = 0$, the Lefschetz fibration $\pi_0: Z_0 \to S^2$ is hyperelliptic with $12g(2g+1)$-many non-separating vanishing cycles. Therefore by \cite[Theorems 4.4(2), 4.8]{endo},
\[
	\sigma(Z_0) = -\left(\frac{g+1}{2g+1}\right) (12g(2g+1)) = -12g(g+1).
\]
To compute $\sigma(Z_n)$ for $n \geq 1$, consider the Lefschetz fibrations $Z' \to D^2$ and $Z'' \to D^2$ with monodromy factorizations
\[
	(T_{c_1} \dots T_{c_{2g}})^{2(4g+2)} \in \Mod(\Sigma_g), \qquad (T_{f^n(c_1)} \dots T_{f^n(c_{2g})})^{4g+2} \in \Mod(\Sigma_g)
\]
respectively. Then $Z_n$ is formed by gluing $Z'$ to $Z''$ along their boundaries by some diffeomorphism $\partial Z' \to \partial Z''$ (which varies with $n$) for any $n \geq 0$ (cf. Section \ref{sec:partial-conj}). By Novikov additivity,
\[
	\sigma(Z_n) = \sigma(Z') + \sigma(Z'') = \sigma(Z_0) = -12g(g+1)
\]
for all $n \geq 0$. Finally, compute for all $n \geq 0$ that
\[
	b_2^+(Z_n) = \frac{b_2(Z_n) + \sigma(Z_n)}{2} = 6g^2 - 2g + 1, \qquad b_2^-(Z_n) = \frac{b_2(Z_n) - \sigma(Z_n)}{2} = 18g^2 + 10g + 1. \qedhere
\]
\end{proof}

By Freedman's theorem \cite{freedman}, the algebraic topology invariants of $Z_n$ determines its homeomorphism type.
\begin{prop}\label{prop:homeo}
For any $n \in \Z_{\geq 0}$, the manifold $Z_n$ is homeomorphic to
\[
	(6g^2 - 2g + 1) \CP^2 \# (18g^2 + 10g + 1)\overline{\CP^2}.
\]
\end{prop}
\begin{proof}
By Lemma \ref{lem:Zn-algtop},
\[
	\sigma(Z_n) = -12g(g+1), \qquad b_2(Z_n) = 24g^2 + 8g + 2
\]
for all $n \geq 0$. In particular, the rank $b_2(Z_n)$ and signature $\sigma(Z_n)$ of the intersection form $Q_{Z_n}$ of $Z_n$ only depend on $g$ and are independent of $n$.

To determine the parity of $Q_{Z_n}$, recall that $\pi_n$ has a section of self-intersection $-3$ given by (\ref{eqn:monodromy-Zn-section}) for all $n \geq 0$. In other words, the intersection form $Q_{Z_n}$ is odd for all $n \geq 0$. Because $Q_{Z_n}$ is indefinite and unimodular, \cite[Theorem 1.2.21]{gompf-stipsicz} shows that
\[
	Q_{Z_n} \cong (6g^2 - 2g + 1) \langle 1 \rangle \oplus (18g^2 + 10g + 1)\langle -1 \rangle.
\]

Finally, Freedman's theorem \cite[Theorem 1.5]{freedman} says that the homeomorphism type of a simply-connected, smooth $4$-manifold is determined by its intersection form. Because $Z_n$ is simply-connected for all $n \geq 0$ by Lemma \ref{lem:Zn-algtop}, Freedman's theorem implies that $Z_n$ is homeomorphic to
\[
	(6g^2 - 2g + 1) \CP^2 \# (18g^2 + 10g + 1)\overline{\CP^2}
\]
for all $n \geq 0$.
\end{proof}

Combining Propositions \ref{prop:homeo} and \ref{prop:non-iso-Zn} concludes the proof of the main theorem.
\begin{proof}[{Proof of Theorem \ref{thm:homeo-not-iso}}]
For any $n \in \Z_{\geq 0}$, the manifold $Z_n$ is homeomorphic to $(6g^2 - 2g + 1) \CP^2 \# (18g^2 + 10g + 1)\overline{\CP^2}$ by Proposition \ref{prop:homeo}. If $n \neq m$ then the Lefschetz fibrations $\pi_n: Z_n \to S^2$ and $\pi_m: Z_m \to S^2$ are inequivalent by Proposition \ref{prop:non-iso-Zn}.
\end{proof}


\appendix \section{Calculations in $\Mod(\Sigma_{2g}^4)$}\label{appendix-a}

In this appendix we collect some routine calculations in $\Mod(\Sigma_{2g}^4)$, $\Mod(\Sigma_{2g,2})$, and $\Mod(\Sigma_{2g})$, including proofs of Theorems \ref{thm:eta-factorization} and \ref{thm:hamada-sections}. For $i = 1, 2$, consider the curves $B_{j, k}^i$ and $C_1^i, C_2^i$ as shown in Figure \ref{fig:MCK-lifts}. Let $c_1, \dots, c_{2g+1}$ and $d_1, \dots, d_{2g}$ denote isotopy classes of curves in $\Sigma_{2g}^4$ as shown in Figure \ref{fig:filling-curves}. Let $\alpha_1, \alpha_2, \alpha_3$ denote the isotopy classes of arcs in $\Sigma_{2g}^4$ as shown in Figure \ref{fig:filling-curves}.

\begin{figure}
\includegraphics[width=0.55\textwidth]{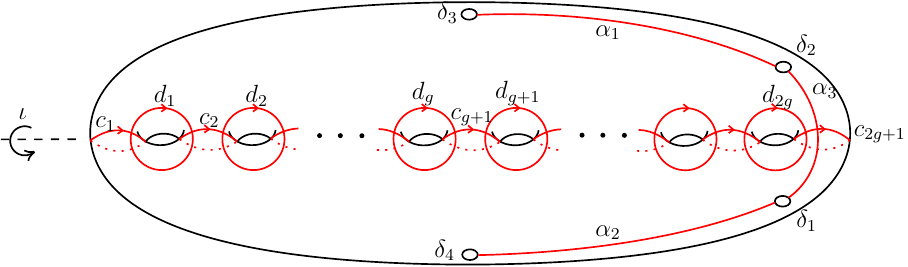}
\caption{Filling curves and arcs, and the hyperelliptic involution $\iota$.}\label{fig:filling-curves}
\end{figure}

The following two lemmas compute a composition of Dehn twists applied to the curves $c_j$ and $d_j$.
\begin{lem}\label{lem:eta-ci}
For $1 \leq j \leq 2g+1$ and $i = 1, 2$, there are equalities of isotopy classes of oriented curves
\[
	(T_{B_{0,1}^i} T_{B_{1,1}^i} \dots T_{B_{2g,1}^i} T_{C_1^i})(T_{B_{0,2}^i} T_{B_{1,2}^i} \dots T_{B_{2g,2}^i} T_{C_2^i})(c_j) = c_j, \qquad T_{B_{0,2}^i} T_{B_{1,2}^i} \dots T_{B_{2g,2}^i} T_{C_2^i}(c_j) = c_{2g+2-j}.
\]
\end{lem}
\begin{proof}
Suppose $j \neq g+1$. The curves $B_{2j-2,k}^i$ and $B_{2j-1, k}^i$ for $k, i = 1, 2$ and $c_j$, $c_{2g+2-j}$ are shown below:
\begin{center}
\includegraphics[width=0.9\textwidth]{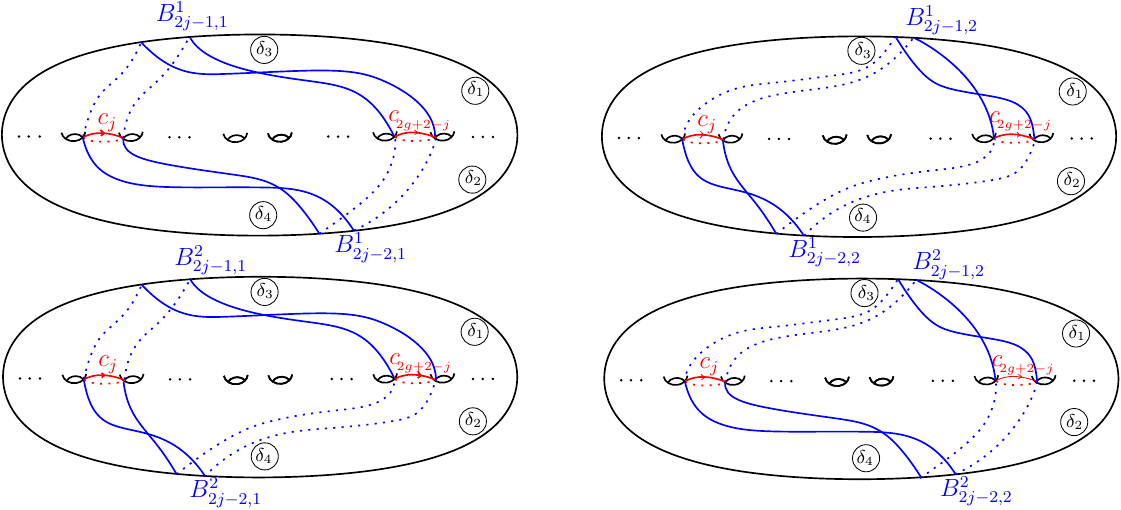}
\end{center}
Direct computations show that as isotopy classes of oriented curves,
\[
	T_{B_{2j-2, 2}^i} T_{B_{2j-1, 2}^i}(c_j) = c_{2g+2-j}, \qquad T_{B_{2j-2, 1}^i} T_{B_{2j-1, 1}^i}(c_{2g+2-j}) = c_{j}.
\]
Note that the curves $c_j$ and $c_{2g+2-j}$ are disjoint from the curves $B_{\ell, k}^i$ and $C_1^i, C_2^i$ if $\ell \neq 2j-2, 2j-1$ and compute as isotopy classes of oriented curves:
\begin{align*}
T_{B_{0,2}^i} T_{B_{1,2}^i} \dots T_{B_{2g,2}^i} T_{C_2^i}(c_j) &= c_{2g+2-j} \\
T_{B_{0,1}^i} T_{B_{1,1}^i} \dots T_{B_{2g,1}^i} T_{C_1^i}  T_{B_{0,2}^i} T_{B_{1,2}^i} \dots T_{B_{2g,2}^i}T_{C_2^i}(c_j) &= T_{B_{0,1}^i} T_{B_{1,1}^i} \dots T_{B_{2g,1}^i} T_{C_1^i} (c_{2g+2-j}) = c_j.
\end{align*}

The case $j = g+1$ follows from the lantern relation \cite[Proposition 5.1]{farb-margalit}. The curves $B_{2g, 2}^i$ and $C_2^i$ are shown below:
\begin{center}
\includegraphics[width=.9\textwidth]{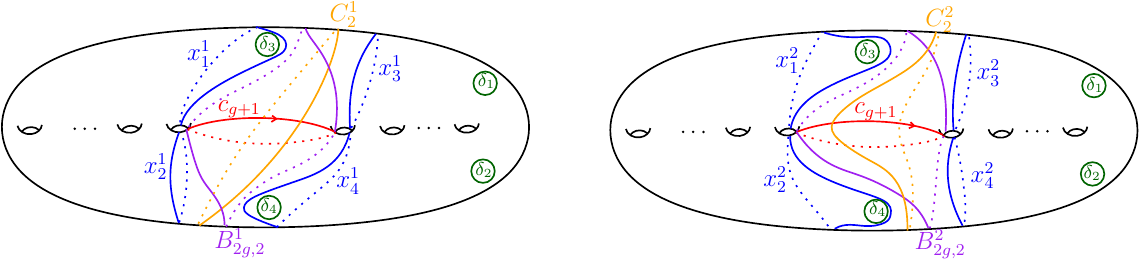}
\end{center}
For both $i = 1,2$, compute with the depicted curves $x_1^i, x_2^i, x_3^i, x_4^i$ that
\[
	T_{B_{2g, 2}^i} T_{C_2^i}(c_{g+1}) = T_{B_{2g, 2}^i} T_{C_2^i} T_{c_{g+1}}(c_{g+1}) = T_{x_1^i} T_{x_2^i} T_{x_3^i} T_{x_4^i}(c_{g+1}) = c_{g+1}
\]
where the second equality follows from the lantern relation $T_{B_{2g,2}^i}T_{C_2^i}T_{c_{g+1}^i} = T_{x_1^i} T_{x_2^i} T_{x_3^i} T_{x_4^i}$.

The curves $B_{2g,1}^i$ and $C_1^i$ are shown below:
\begin{center}
\includegraphics[width=.9\textwidth]{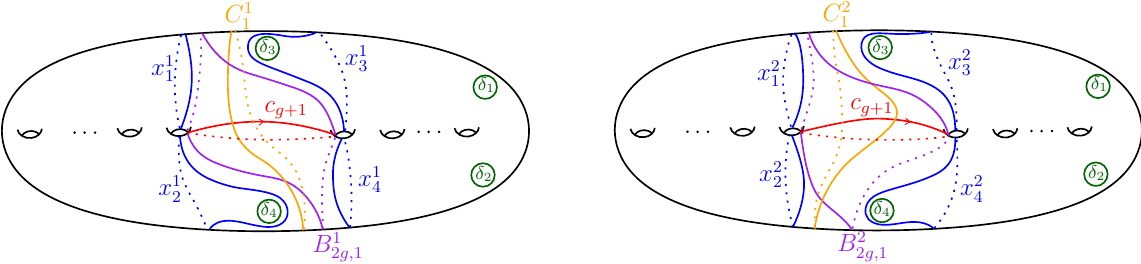}
\end{center}
For both $i = 1,2$, compute with the depicted curves $x_1^i, x_2^i, x_3^i, x_4^i$ that
\[
	T_{B_{2g, 1}^i} T_{C_1^i}(c_{g+1}) = T_{B_{2g, 1}^i} T_{C_1^i} T_{c_{g+1}}(c_{g+1}) = T_{x_1^i} T_{x_2^i} T_{x_3^i} T_{x_4^i}(c_{g+1}) = c_{g+1}
\]
where the second equality follows from the lantern relation $T_{B_{2g,1}^i}T_{C_1^i}T_{c_{g+1}} = T_{x_1^i} T_{x_2^i} T_{x_3^i} T_{x_4^i}$.

Finally, note that $c_{g+1}$ and $B_{\ell, k}^i$ are disjoint if $\ell \neq 2g$ and compute
\begin{align*}
T_{B_{0,2}^i} T_{B_{1,2}^i} \dots T_{B_{2g,2}^i} T_{C_2^i} (c_{g+1}) &= c_{g+1} \\
T_{B_{0,1}^i} T_{B_{1,1}^i} \dots T_{B_{2g,1}^i} T_{C_1^i} T_{B_{0,2}^i} T_{B_{1,2}^i} \dots T_{B_{2g,2}^i} T_{C_2^i} (c_{g+1}) &= T_{B_{0,1}^i} T_{B_{1,1}^i} \dots T_{B_{2g,1}^i} T_{C_1^i} (c_{g+1}) = c_{g+1}. \qedhere
\end{align*}
\end{proof}
\begin{lem}\label{lem:eta-di}
For $1 \leq j \leq 2g$ and $i = 1, 2$, there are equalities of isotopy classes of oriented curves 
\[
	(T_{B_{0,1}^i} T_{B_{1,1}^i} \dots T_{B_{2g,1}^i} T_{C_1^i})(T_{B_{0,2}^i} T_{B_{1,2}^i} \dots T_{B_{2g,2}^i} T_{C_2^i})(d_j) = d_j, \qquad T_{B_{0,2}^i} T_{B_{1,2}^i} \dots T_{B_{2g,2}^i}T_{C_2^i}(d_j) = \bar d_{2g+1-j}.
\]
\end{lem}
\begin{proof}
Consider the curves $B_{2j-1, k}^i$, $B_{2j, k}^i$ for $k, i = 1, 2$ and $d_j$, $\bar d_{2g+1-j}$ shown below:
\begin{center}
\includegraphics[width=0.9\textwidth]{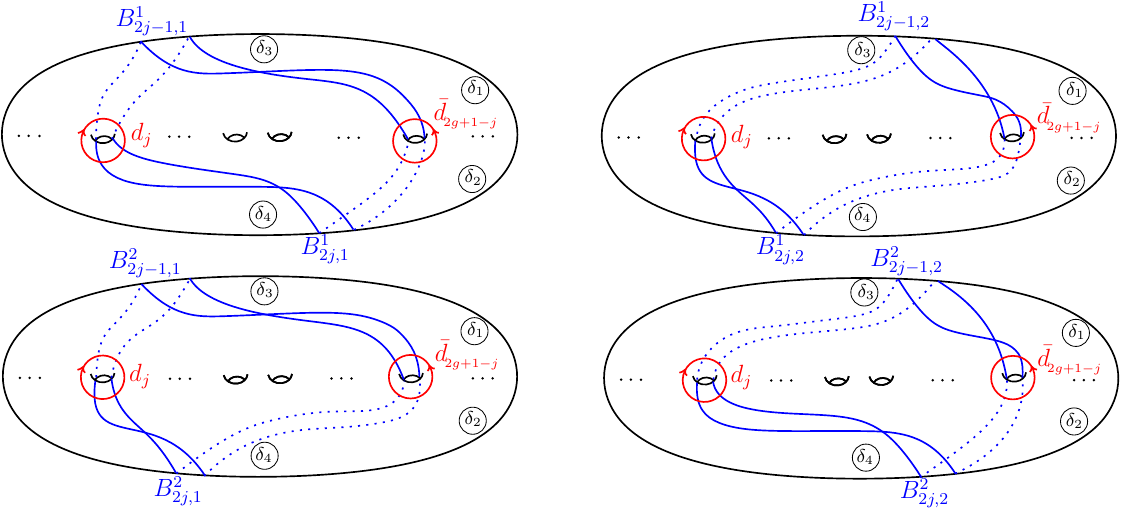}
\end{center}
Direct computations show that
\[
	T_{B_{2j-1, 2}^i} T_{B_{2j, 2}^i}(d_j) = \bar d_{2g+1-j}, \qquad T_{B_{2j-1, 1}^i} T_{B_{2j, 1}^i}(\bar d_{2g+1-j}) = d_{j}.
\]
Note that the curves $d_j$ and $d_{2g+1-j}$ are disjoint from the curves $B_{\ell, k}^i$ if $\ell \neq 2j-1, 2j$ and compute 
\begin{align*}
T_{B_{0,2}^i} T_{B_{1,2}^i} \dots T_{B_{2g,2}^i} T_{C_2^i}(d_j) &= \bar d_{2g+1-j} \\
T_{B_{0,1}^i} T_{B_{1,1}^i} \dots T_{B_{2g,1}^i} T_{C_1^i}  T_{B_{0,2}^i} T_{B_{1,2}^i} \dots T_{B_{2g,2}^i} T_{C_2^i}(d_j) &= T_{B_{0,1}^i} T_{B_{1,1}^i} \dots T_{B_{2g,1}^i} T_{C_1^i}(\bar d_{2g+1-j}) = d_{j}. \qedhere
\end{align*}
\end{proof}

We now compute $\tilde h_i(\alpha_1)$ and $\tilde h_i(\alpha_2)$.
\begin{lem}\label{lem:alpha-12}
For $i = 1, 2$ and $j = 1, 2$,
\[
	(T_{B_{0,1}^i} T_{B_{1,1}^i} \dots T_{B_{2g,1}^i} T_{C_1^i})(T_{B_{0,2}^i} T_{B_{1,2}^i} \dots T_{B_{2g,2}^i} T_{C_2^i})(\alpha_j) = T_{\delta_1} T_{\delta_2}T_{\delta_3} T_{\delta_4}(\alpha_j)
\]
as isotopy classes of arcs in $\Sigma_{2g}^4$.
\end{lem}
\begin{proof}
We first consider the arc $\alpha_1$. For each $i = 1, 2$ and any $0 \leq \ell \leq g$, consider the arc $\gamma_\ell^i$ depicted below:
\begin{center}
\includegraphics[width=0.9\textwidth]{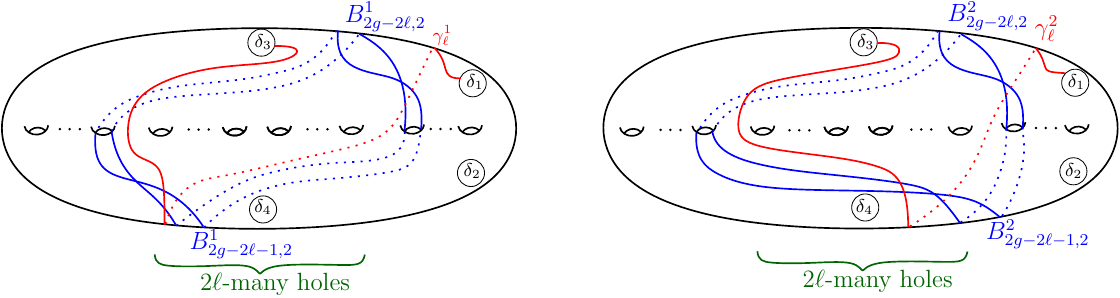}
\end{center}
Direct computations show that the following equalities of isotopy classes of arcs in $\Sigma_{2g}^4$ hold:
\begin{enumerate}[(a)]
\item $\gamma_0^i = T_{C_2^i}(\alpha_1)$, and
\item $T_{B_{2g-2\ell -1}^i} T_{B_{2g-2\ell}^i}(\gamma_\ell^i) = \gamma_{\ell+1}^i$ for all $0 \leq \ell\leq g-1$.
\end{enumerate}
In particular, $\gamma_g^i = T_{B_{1,2}^i} T_{B_{2,2}^i} \dots T_{B_{2g,2}^i} T_{C_2^i}(\alpha_1)$ by induction; the arc $\gamma_g^i$ and the curve $B_{0, 2}^i$ are depicted below.
\begin{center}
\includegraphics[width=0.9\textwidth]{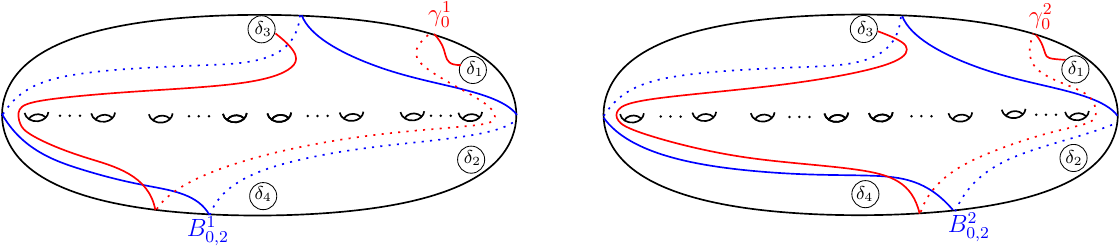}
\end{center}
Using the above, direct computation shows that
\[
	T_{B_{0,2}^i}(\gamma_g^i) = T_{B_{0,2}^i} T_{B_{1,2}^i} \dots T_{B_{2g,2}^i} T_{C_2^i}(\alpha_1) = T_{\delta_1} T_{\delta_2} T_{\delta_3} T_{\delta_4}(\alpha_1).
\]
Because $T_{\delta_1} T_{\delta_2} T_{\delta_3} T_{\delta_4}(\alpha_1)$ is disjoint from $B_{j, 1}^i$ and $C_1^i$ for all $0 \leq j \leq 2g$,
\[
	(T_{B_{0,1}^i} T_{B_{1,1}^i} \dots T_{B_{2g,1}^i} T_{C_1^i})(T_{B_{0,2}^i} T_{B_{1,2}^i} \dots T_{B_{2g,2}^i} T_{C_2^i})(\alpha_1) = T_{\delta_1} T_{\delta_2} T_{\delta_3} T_{\delta_4}(\alpha_1).
\]

For the arc $\alpha_2$, consider the hyperelliptic involution $\iota \in \Diff^+(\Sigma_{2g}, \delta_1\cup\delta_2\cup\delta_3\cup\delta_4)$ with $\iota(\delta_1) = \delta_2$ and $\iota(\delta_3) = \delta_4$ as shown in Figure~\ref{fig:filling-curves}. Then $\iota(\alpha_1) = \alpha_2$.
\begin{enumerate}[(a)]
	\item If $i = 1$ then for any $0 \leq j \leq 2g$, there are equalities of isotopy classes $\iota(B_{j,1}^1) = B_{j,2}^1$ and $\iota(C_1^1) = C_2^1$. Therefore,
	\begin{align*}
		T_{B_{0,1}^1} \dots T_{B_{2g,1}^1 T_{C_1^1}}T_{B_{0,2}^1} \dots T_{B_{2g,2}^1 T_{C_2^1}}(\alpha_2) &= \iota^2(T_{B_{0,1}^1} \dots T_{B_{2g,1}^1 T_{C_1^1}}T_{B_{0,2}^1} \dots T_{B_{2g,2}^1 T_{C_2^1}})(\iota\alpha_1) \\
		&= \iota(T_{B_{0,2}^1} \dots T_{B_{2g,2}^1 T_{C_2^1}}T_{B_{0,1}^1} \dots T_{B_{2g,1}^1 T_{C_1^1}})(\alpha_1) \\
		&= \iota(T_{B_{0,2}^1} \dots T_{B_{2g,2}^1 T_{C_2^1}})(\alpha_1) \\
		&= T_{\delta_1} T_{\delta_2} T_{\delta_3} T_{\delta_4}(\alpha_2).
	\end{align*}
	\item If $i = 2$ then for any $0 \leq j \leq 2g$, there are equalities of isotopy classes $\iota(B_{j,1}^2) = B_{j,1}^2$, $\iota(B_{j,2}^2) = B_{j,2}^2$, $\iota(C_1^2) = C_1^2$, and $\iota(C_2^2) = C_2^2$. Therefore,
	\begin{align*}
	T_{B_{0,1}^2} \dots T_{B_{2g,1}^2 T_{C_1^2}}T_{B_{0,2}^2} \dots T_{B_{2g,2}^2 T_{C_2^2}}(\alpha_2) &= \iota^2(T_{B_{0,1}^2} \dots T_{B_{2g,1}^2 T_{C_1^2}}T_{B_{0,2}^2} \dots T_{B_{2g,2}^2 T_{C_2^2}})(\iota\alpha_1) \\
	&= \iota(T_{B_{0,1}^2} \dots T_{B_{2g,1}^2 T_{C_1^2}}T_{B_{0,2}^2} \dots T_{B_{2g,2}^2 T_{C_2^2}}(\alpha_1)) \\
	&= T_{\delta_1} T_{\delta_2} T_{\delta_3} T_{\delta_4}(\alpha_2). \qedhere
	\end{align*}
\end{enumerate}
\end{proof}

\begin{lem}\label{lem:alpha-3}
For $i = 1, 2$,
\[
	(T_{B_{0,1}^i} T_{B_{1,1}^i} \dots T_{B_{2g,1}^i} T_{C_1^i})(T_{B_{0,2}^i} T_{B_{1,2}^i} \dots T_{B_{2g,2}^i} T_{C_2^i})(\alpha_3) = T_{\delta_1} T_{\delta_2} T_{\delta_3 } T_{\delta_4} (\alpha_3)
\]
as isotopy classes of arcs in $\Sigma_{2g}^4$.
\end{lem}
\begin{proof}
For each $i = 1, 2$ and any $0 \leq \ell \leq g-1$, consider the arc $\gamma_\ell^i$ in $\Sigma_{2g}^4$ depicted below:
\begin{center}
\includegraphics[width=.9\textwidth]{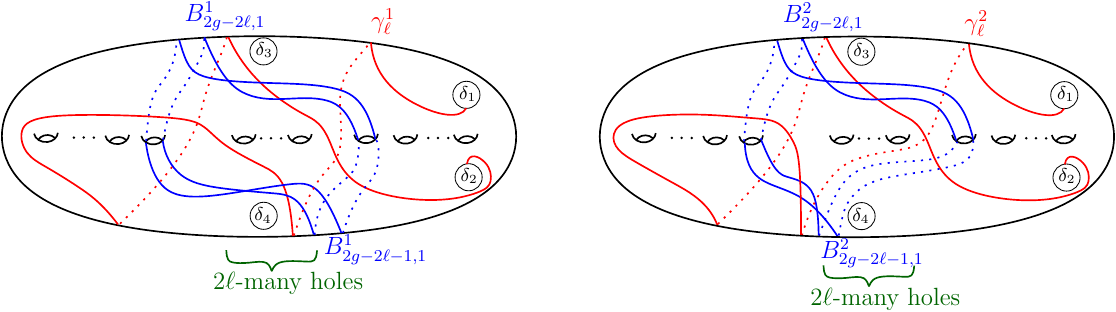}
\end{center}
Direct computation shows that the following equalities of isotopy classes of arcs hold: 
\begin{enumerate}[(a)]
\item $T_{C_1^i}T_{B_{0,2}^i}(\alpha_3) = \gamma_0^i$,
\item $T_{B_{2g-2\ell-1,1}^i} T_{B_{2g-2\ell,1}^i}(\gamma_\ell^i) = \gamma_{\ell+1}^i$ for every $0 \leq \ell \leq g-1$, and
\item $T_{B_{0,1}^i}(\gamma_{g}^i) = T_{\delta_1} T_{\delta_2}(\alpha_3)$.
\end{enumerate}
Combining these facts (and the fact that $\alpha_3$ is disjoint from the curves $B_{1,2}^i, \dots, B_{2g,2}^i, C_2^i$) together shows that
\[
	T_{B_{0, 1}^i} (T_{B_{1,1}^i} T_{B_{2,1}^i} \dots T_{B_{2g,1}^i})(T_{C_1^i}T_{B_{0,2}^i}) (T_{B_{1,2}^i} \dots T_{B_{2g,2}^i} T_{C_2^i})(\alpha_3) = T_{\delta_1} T_{\delta_2} (\alpha_3).
\]
Noting that $\alpha_3$ is disjoint from $\delta_3$ and $\delta_4$ concludes the proof.
\end{proof}

Below, we prove Theorem \ref{thm:hamada-sections}. Our proof is different from Hamada's original proof but instead is similar to Gurtas' proof of \cite[Theorem 2.0.1]{gurtas}. Note, however, that our version of Theorem \ref{thm:hamada-sections} does not recover the full computations of Hamada's work \cite{hamada}.
\begin{proof}[{Proof of Theorem \ref{thm:hamada-sections}}]
Consider the curves $c_1, \dots, c_{2g+1}$, $d_1, \dots, d_{2g}$ and arcs $\alpha_1, \alpha_2, \alpha_3$ shown in Figure \ref{fig:filling-curves}; these curves and arcs fill $\Sigma_{2g}^4$ and satisfy the conditions of the Alexander method \cite[Proposition 2.8]{farb-margalit}. Let
\[
	k_i := T_{\delta_1}^{-1} T_{\delta_2}^{-1} T_{\delta_3}^{-1} T_{\delta_4}^{-1}T_{B_{0,1}^i} T_{B_{1,1}^i} \dots T_{B_{2g,1}^i} T_{C_1^i}  T_{B_{0,2}^i} T_{B_{1,2}^i} \dots T_{B_{2g,2}^i} T_{C_2^i} \in \Mod(\Sigma_{2g}^4).
\]
By Lemmas \ref{lem:eta-ci} and \ref{lem:eta-di}, there are equalities of oriented isotopy classes of curves
\[
	k_i(c_j) = c_j, \qquad k_i(d_j) = d_j
\]
for all $j$. By Lemmas \ref{lem:alpha-12} and \ref{lem:alpha-3}, $k_i(\alpha_j) = \alpha_j$ for $j = 1, 2, 3$. These two facts combined imply that the mapping class $k_i$ fixes each vertex and each edge (with orientations) of the graph determined by $\bigcup_{j=1}^{2g+1} c_j \cup \bigcup_{j=1}^{2g} d_j \cup \bigcup_{j=1}^3 \alpha_j$. Therefore, the Alexander method \cite[Proposition 2.8]{farb-margalit} shows that $k_i = 1 \in \Mod(\Sigma_{2g}^4)$ for both $i = 1, 2$.
\end{proof}

Similarly using the Alexander method, we deduce Theorem \ref{thm:eta-factorization}.
\begin{proof}[{Proof of Theorem \ref{thm:eta-factorization}}]
Consider the curves $c_1, \dots, c_{2g+1}, d_1, \dots, d_{2g}$ as shown in Figure \ref{fig:filling-curves}, under the inclusion $\Sigma_{2g}^4 \hookrightarrow \Sigma_{2g}$ given by capping off the four boundary components with disks. Under this inclusion, the images of the curves $B_{j, 1}^i$ and $B_{j, 2}^i$ are both isotopic to the curve $B_j$ in $\Sigma_{2g}$ for all $0 \leq j \leq 2g$ and $i = 1, 2$ and the images of the curves $C_1^i$ and $C_2^i$ are both isotopic to the curve $C$ in $\Sigma_{2g}$ for $i = 1, 2$. By Lemmas \ref{lem:eta-ci} and \ref{lem:eta-di},
\[
	T_{B_0} \dots T_{B_{2g}} T_C(c_j) = c_{2g+2-j}, \qquad T_{B_0} \dots T_{B_{2g}} T_C(d_j) = \bar{d}_{2g+1-j}
\]
for all $j$ as oriented isotopy classes of curves. Therefore, the mapping class
\[
	[\eta] T_{B_0} \dots T_{B_{2g}} T_C \in \Mod(\Sigma_{2g})
\]
acts as the identity on the graph determined by $\bigcup_{j=1}^{2g+1} c_j \cup \bigcup_{j=1}^{2g}d_j$. Because the curves $c_1, \dots, c_{2g+1}$, $d_1, \dots, d_{2g}$ fill $\Sigma_{2g}$ and satisfy the the conditions of the Alexander method \cite[Proposition 2.8]{farb-margalit}, it applies to show that
\[
	[\eta] T_{B_0} \dots T_{B_{2g}} T_C = 1 \in \Mod(\Sigma_{2g}).
\]
Finally, note that $[\eta]$ has order $2$ because $\eta$ has order $2$ and is not isotopic to the identity.
\end{proof}
We also record a corollary of Theorem \ref{thm:hamada-sections} which studies the mapping class $\hat h \in \Mod(\Sigma_{2g, 2})$ first defined in (\ref{eqn:hat-h}).
\begin{cor}\label{cor:puncture-involution}
For $i = 1, 2$, recall the mapping classes
\[
	\tilde h_i = T_{B_{0,2}^i} T_{B_{1,2}^i} \dots T_{B_{2g,2}^i} T_{C_2^i} \in \Mod(\Sigma_{2g}^4)
\]
as defined in Lemma \ref{lem:bounding-pair-sections}. Let
\[
\hat h_1, \hat h_2 \in \Mod(\Sigma_{2g, 2})\]
denote the images of $\tilde h_1, \tilde h_2 \in \Mod(\Sigma_{2g}^4)$ respectively under the capping and forgetful homomorphisms
\[
	\Mod(\Sigma_{2g}^4) \to \Mod(\Sigma_{2g, 2})
\]
which caps each boundary component $\delta_1, \delta_2, \delta_3, \delta_4$ of $\Sigma_{2g}^4$ with a disk with a marked point $p_1, p_2, p_3, p_4$ respectively, and then forgets two marked points $p_3, p_4$. Then $\hat h_1 = \hat h_2 \in \Mod(\Sigma_{2g, 2})$ and has order $2$.
\end{cor}
\begin{proof}
Consider the inclusion of surfaces
\[
	\Sigma_{2g}^4 \hookrightarrow \Sigma_{2g, 2}
\]
given by capping each boundary component $\delta_1$, $\delta_2$ with a disk with a marked point $p_1$, $p_2$ respectively and each boundary component $\delta_3$, $\delta_4$ of $\Sigma_{2g}^4$ with a disk. We now take isotopy classes of curves in $\Sigma_{2g}^4$ and consider their images in $\Sigma_{2g, 2}$ under this inclusion. From Figure \ref{fig:MCK-lifts}, observe that there are equalities of isotopy classes of curves
\[
	B_{j, 1}^1 = B_{j, 2}^1 = B_{j, 1}^2 = B_{j, 2}^2 \subseteq \Sigma_{2g, 2}
\]
for all $0 \leq j \leq 2g$ and
\[
	C_1^1 = C_2^1 = C_1^2 = C_2^2 \subseteq \Sigma_{2g, 2}.
\]
In other words,
\[
	T_{B_{j,1}^1} = T_{B_{j,2}^1} = T_{B_{j,1}^2} = T_{B_{j,2}^2} \in \Mod(\Sigma_{2g, 2}), \qquad T_{C_1^1} = T_{C_2^1} = T_{C_1^2} = T_{C_2^2} \in \Mod(\Sigma_{2g, 2})
\]
for all $0 \leq j \leq 2g$. Therefore, there is an equality in $\Mod(\Sigma_{2g, 2})$
\[
	\hat h_1 = T_{B_{0,2}^1} T_{B_{1,2}^1} \dots T_{B_{2g,2}^1} T_{C_2^1} = T_{B_{0,2}^2} T_{B_{1,2}^2} \dots T_{B_{2g,2}^2} T_{C_2^2} = \hat h_2\in \Mod(\Sigma_{2g, 2}).
\]
The equalities of isotopy classes of curves above also show that
\[
	\hat h_1 = T_{B_{0, 1}^1} T_{B_{1,1}^1} \dots T_{B_{2g,1}^1} T_{C_1^1} \in \Mod(\Sigma_{2g, 2}).
\]
Now to see that $\hat h_1$ (and hence $\hat h_2$) has order $2$, compute using Theorem \ref{thm:hamada-sections} that 
\[
	\hat h_1^2 = \left(T_{B_{0,1}^1} T_{B_{1,1}^1} \dots T_{B_{2g,1}^1} T_{C_1^1}\right)\left(T_{B_{0, 2}^1} T_{B_{1,2}^1} \dots T_{B_{2g,2}^1} T_{C_2^1}\right) = 1 \in \Mod(\Sigma_{2g, 2}). \qedhere
\]
\end{proof}

Finally, the following lemma is crucial in the construction of the partial conjugations of the MCK Lefschetz fibration in Section \ref{sec:twisted-mck}.
\begin{lem}\label{lem:tilde-f-lemma}
For $i = 1, 2$, recall the mapping classes
\[
	\tilde h_i = T_{B_{0,2}^i} T_{B_{1,2}^i} \dots T_{B_{2g,2}^i} T_{C_2^i} \in \Mod(\Sigma_{2g}^4)
\]
as defined in Lemma \ref{lem:bounding-pair-sections}.
\begin{enumerate}[(a)]
\item \label{lem:tilde-f-lemma-disjoint} The curves $\tilde x$, $\tilde y$, $\tilde h_i(\tilde x)$, and $\tilde h_i(\tilde y)$ are pairwise disjoint in $\Sigma_{2g}^4$ and are as shown in Figure \ref{fig:bounding-pair}. In particular,
  \[ \tilde h_1(\tilde x) = \tilde h_2(\tilde x), \qquad \tilde h_1(\tilde y) = \tilde h_2(\tilde y). \]
\item \label{lem:tilde-f-lemma-double} There are equalities of isotopy classes of curves $\tilde h_i(\tilde x) = \tilde h_i^{-1}(\tilde x)$ and $\tilde h_i(\tilde y) = \tilde h_i^{-1} (\tilde y)$.
\end{enumerate}
\end{lem}
\begin{proof}
  Let $\tilde z$ and $\tilde w$ be the isotopy classes of curves depicted in \Cref{fig:bounding-pair} labelled $\tilde h_i(\tilde x)$ and $\tilde h_i(\tilde y)$ respectively. First, observe that $B_{j,2}^i$ and $C_2^i$ are disjoint from $\tilde x$ and $\tilde z$ for any $j>3$ and both $i=1,2$. Therefore,
  \[
  	\tilde h_i(\tilde x) = T_{B_{0,2}^i} T_{B_{1,2}^i} T_{B_{2,2}^i} T_{B_{3,2}^i}(\tilde x), \qquad \tilde h_i(\tilde z) = T_{B_{0,2}^i} T_{B_{1,2}^i} T_{B_{2,2}^i} T_{B_{3,2}^i}(\tilde z)
  \]
  for both $i = 1, 2$. A direct computation shows that
  \[
  	T_{B_{0,2}^i} T_{B_{1,2}^i} T_{B_{2,2}^i} T_{B_{3,2}^i}(\tilde x) = \tilde z, \qquad T_{B_{0,2}^i} T_{B_{1,2}^i} T_{B_{2,2}^i} T_{B_{3,2}^i}(\tilde z) = \tilde x
  \]
  for both $i = 1, 2$. This proves \ref{lem:tilde-f-lemma-disjoint} and \ref{lem:tilde-f-lemma-double} for $\tilde x$.

  To simplify the computations for $\tilde y$, note that the hyperelliptic involution $\iota$ (as shown on the right side of \Cref{fig:bounding-pair}) acts by
  \[
  	\iota(\tilde x) = \tilde y, \qquad \iota(\tilde z) = \tilde w.
  \]
  Letting $\tilde g_i := \iota^{-1} \tilde h_i \iota$, compute that
  \[
  	\tilde h_i(\tilde y) = \iota \tilde g_i (\tilde x), \qquad \tilde h_i(\tilde w) = \iota \tilde g_i(\tilde z).
	\]

	We now compute for $i = 1$ and $i = 2$ separately. For $i = 1$, observe in Figure \ref{fig:MCK-lifts} that
	\[
		\iota(B_{j,2}^1) = B_{j,1}^1, \qquad \iota(C_2^1) = C_1^1
	\]
	for all $0 \leq j \leq 2g$. Therefore,
	\[
			\tilde g_1 = T_{B_{0,1}^1} T_{B_{1,1}^1} \dots T_{B_{2g,1}^1} T_{C_1^1} \in \Mod(\Sigma_{2g}^4).
	\]
	Theorem \ref{thm:hamada-sections} shows that
	\[
		\tilde g_1 \tilde h_1(\tilde x) = T_{\delta_1} T_{\delta_2} T_{\delta_3} T_{\delta_4}(\tilde x) = \tilde x, \qquad \tilde g_1 \tilde h_1(\tilde z) = T_{\delta_1} T_{\delta_2} T_{\delta_3} T_{\delta_4}(\tilde z) = \tilde z.
	\]
	Now \ref{lem:tilde-f-lemma-disjoint} for $\tilde y$ follows because
	\[
		\tilde h_1(\tilde y) = \iota \tilde g_1(\tilde x) = \iota \tilde g_1(\tilde h_1^2(\tilde x)) = \iota \tilde g_1(\tilde h_1(\tilde z)) = \iota(\tilde z) = \tilde w,
	\]
	where the second and third equalities follow from \ref{lem:tilde-f-lemma-double} and \ref{lem:tilde-f-lemma-disjoint} for $\tilde x$ respectively. To see \ref{lem:tilde-f-lemma-double} for $\tilde y$, apply the same computation as above for $\tilde z$:
	\[
		\tilde h_1^2(\tilde y) = \tilde h_1(\tilde w) = \iota \tilde g_1(\tilde z) = \iota \tilde g_1(\tilde h_1^2(\tilde z)) = \iota \tilde g_1(\tilde h_1(\tilde x)) = \iota(\tilde x) = \tilde y,
	\]
	and hence $\tilde h_1
	^{-1}(\tilde y) = \tilde h_1(\tilde y)$.

	For $i = 2$, observe in Figure \ref{fig:MCK-lifts} that
	\[
		\iota(B_{j,2}^2) = B_{j,2}^2, \qquad \iota(C_2^2) = C_2^2
	\]
	for all $0 \leq j \leq 2g$, and so
	\[
		\tilde g_2 = \tilde h_2 \in \Mod(\Sigma_{2g}^4).
	\]
	Now \ref{lem:tilde-f-lemma-disjoint} for $\tilde y$ follows because
	\[
		\tilde h_2(\tilde y) = \tilde h_2(\iota(\tilde x)) = \iota \tilde g_2(\tilde x) = \iota \tilde h_2(\tilde x) = \iota(\tilde z) = \tilde w,
	\]
	where the second equality follows from the definition of $\tilde g_2$, the third follows from the fact that $\tilde g_2 = \tilde h_2$, and the fourth follows from \ref{lem:tilde-f-lemma-disjoint} for $\tilde x$. To see \ref{lem:tilde-f-lemma-double} for $\tilde y$, apply the same computation as above for $\tilde z$:
	\[
		\tilde h_2^2(\tilde y) = \tilde h_2(\tilde w) = \tilde h_2(\iota(\tilde z)) = \iota \tilde g_2(\tilde z) = \iota \tilde h_2(\tilde z) = \iota (\tilde x) = \tilde y,
	\]
	and hence $\tilde h_2^{-1}(\tilde y ) = \tilde h_2(\tilde y)$. 
\end{proof}


\section{Proof of Proposition \ref{prop:isotoping-hateta}} \label{appendix-b}

The goal of this appendix is to prove Proposition \ref{prop:isotoping-hateta}. Throughout this section
we follow the notation of Remark~\ref{capping}. By Lemma \ref{lem:tilde-f-lemma}, there are equalities of isotopy classes of curves in $\Sigma_{2g, 4}$ 
\[
	\tilde h^2(\tilde x) = \tilde x, \qquad \tilde h^2(\tilde y) = \tilde y.
\]
Therefore there exists a representative $\tilde\eta_0 \in \Diff^+(\Sigma_{2g, 4})$ of $\tilde h \in \Mod(\Sigma_{2g, 4})$ that preserves the set $\{\alpha, \beta, \gamma, \delta\}$ of curves in $\Sigma_{2g, 4}$ (cf. \cite[Section 13.2.2]{farb-margalit}).

Let $T_1, S, T_2$ denote the closures of the three connected components of
\[
	\Sigma_{2g, 4} - \left(\alpha \sqcup \beta \sqcup \tilde\eta_0(\alpha) \sqcup \tilde\eta_0(\beta)\right)
\]
as shown in Figure \ref{fig:surface-decomp}, where we regard $S \cong \Sigma_{2g-2, 4}^4$ as a surface with four marked points and four boundary components. Observe that $\tilde\eta_0$ permutes the subsets $T_1$ and $T_2$ and preserves $S$. Moreover, $\tilde\eta_0^2$ preserves each subset $T_1$, $S$, and $T_2$.

\begin{figure}
\includegraphics[width=0.5\textwidth]{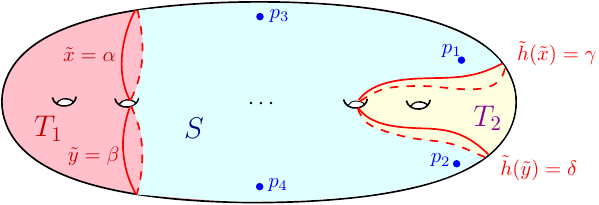}
\includegraphics[width=0.4\textwidth]{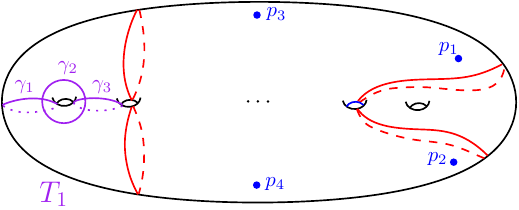}
\caption{Left: Curves $\alpha$, $\beta$, $\gamma$, $\delta$ in $\Sigma_{2g, 4}$ decomposing the surface as a union of compact subsurfaces $T_1$, $S$, $T_2$; Right: Curves $\gamma_1$, $\gamma_2$, $\gamma_3$ that fill $T^\circ_1$ and are fixed up to isotopy by $\tilde h$.} \label{fig:surface-decomp}
\end{figure}

In the following lemma, we modify $\tilde\eta_0$ to a new diffeomorphism $\tilde\eta_1 \in \Diff^+(\Sigma_{2g, 4})$ by an isotopy supported compactly in $\mathring T_1$ such that $\tilde\eta_1$ has order $2$ away from some collar neighborhoods in $T_1 \cup T_2$. The isotopies and diffeomorphisms of the proof of Lemma \ref{lem:modifying-t1-t2} are depicted in Figure \ref{fig:t1-cartoon}.
\begin{figure}
\includegraphics[width=0.5\textwidth]{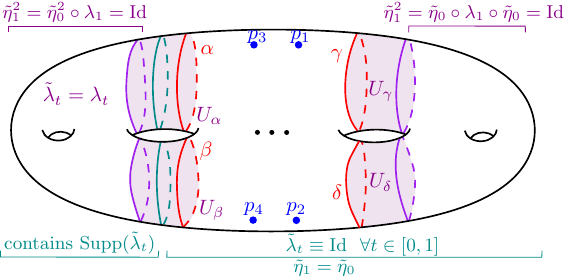}
\caption{A schematic of the constructed diffeomorphism $\tilde\eta_1$ and the support of the isotopy $\tilde\lambda_t$ in the proof of Lemma \ref{lem:modifying-t1-t2}}\label{fig:t1-cartoon}
\end{figure}

\begin{lem}[{Modification in $T_1 \cup T_2$}]\label{lem:modifying-t1-t2}
There exist collar neighborhoods $U_\alpha \subseteq T_1$ and $U_\beta \subseteq T_1$ of $\alpha$ and $\beta$ respectively and a diffeomorphism $\tilde\eta_1 \in \Diff^+(\Sigma_{2g, 4})$ with $[\tilde \eta_1] = \tilde h \in \Mod(\Sigma_{2g, 4})$ such that
\[
	\tilde\eta_1^2|_{T_1 - (U_\alpha \sqcup U_\beta)} = \Id|_{T_1 - (U_\alpha \sqcup U_\beta)}
\]
and such that
\[
	\tilde\eta_1|_{S \cup T_2} = \tilde\eta_0|_{S \cup T_2}.
\]
In particular, $\tilde\eta_1$ permutes the subsets $T_1$, $S$, and $T_2$, and $\tilde\eta_1^2$ preserves each subset
\[
	T_1 - (U_\alpha \sqcup U_\beta), \quad U_\alpha, \quad U_\beta, \quad S, \quad T_2 - (U_\gamma \sqcup U_\delta), \quad U_\gamma, \quad U_\delta
\]
where $U_\gamma := \tilde\eta_1(U_\alpha)$ and $U_\delta := \tilde\eta_1(U_\beta)$.
\end{lem}
\begin{proof}
Consider the curves $\gamma_1, \gamma_2, \gamma_3$ in $T_1$ as shown in Figure \ref{fig:surface-decomp}. By Lemmas \ref{lem:eta-ci} (for the isotopy classes $[\gamma_1] = c_1$ and $[\gamma_2] = c_2$) and \ref{lem:eta-di} (for the isotopy class $[\gamma_2] = d_1$), the curve $\tilde\eta_0^2(\gamma_i)$ is isotopic to $\gamma_i$ in $\Sigma_{2g, 4}$ for each $i = 1, 2, 3$. By \cite[Lemma 3.16]{farb-margalit}, the curve $\tilde\eta_0^2(\gamma_i)$ is then also isotopic to $\gamma_i$ in $T_1$ for each $i = 1, 2, 3$.

Consider the interior $\mathring{T_1}$ of $T_1$. The curves $\gamma_1, \gamma_2, \gamma_3 \subseteq \mathring T_1$ together fill $\mathring T_1$. By the Alexander method \cite[Proposition 2.8]{farb-margalit},
\[
	[\tilde\eta_0^2|_{\mathring T_1}] = 1 \in \Mod(\mathring T_1).
\]
Let $\lambda_t: \mathring T_1 \times [0, 1] \to \mathring T_1$ denote an isotopy with $\lambda_0 = \Id_{\mathring T_1}$ and $\lambda_1 = \tilde\eta_0^{-2}|_{\mathring T_1}$. Fix any collar neighborhoods $U_\alpha$, $U_\beta$ of $\alpha$ and $\beta$ in $T_1$. By the isotopy extension theorem \cite[Theorem 8.1.4]{hirsch}, there exists an isotopy $\tilde\lambda _t: \Sigma_{2g, 4} \times [0, 1] \to \Sigma_{2g, 4}$ with $\tilde\lambda_0 = \Id_{\Sigma_{2g, 4}}$ such that
\[
	\tilde\lambda_t|_{T_1 - (U_\alpha \sqcup U_\beta)} = \lambda_t|_{T_1 - (U_\alpha \sqcup U_\beta)}
\]
and such that $\tilde\lambda_t$ is compactly supported in $\mathring T_1$. Finally, let
\[
	\tilde \eta_1 := \tilde\eta_0 \circ \tilde\lambda_1.
\]
Because $\tilde\lambda_t$ is supported in $\mathring T_1$,
\[
	\tilde\eta_1|_{S \cup T_2} = \tilde\eta_0 \circ \tilde\lambda_1|_{S \cup T_2} = \tilde\eta_0|_{S \cup T_2}.
\]
Finally,
\[
	(\tilde\eta_0 \circ \tilde\lambda_1)^2|_{T_1 - (U_\alpha \sqcup U_\beta)} = (\tilde\eta_0 \circ \tilde\lambda_1)|_{T_2}  \circ (\tilde\eta_0 \circ \lambda_1)|_{T_1 - (U_\alpha \sqcup U_\beta)} = \tilde\eta_0|_{T_2} \circ \tilde\eta_0^{-1}|_{T_1 - (U_\alpha \sqcup U_\beta)} = \Id_{T_1 - (U_\alpha \sqcup U_\beta)}
\]
as desired.
\end{proof}

In the following lemma, we view the subsurface $S \subseteq \Sigma_{2g,4}$ as the compact subsurface with four boundary components (coming from $\alpha, \beta, \gamma, \delta$) containing the four marked points $p_1, p_2, p_3, p_4$. Let $\mathring S$ denote the interior of $S$; in particular, $\mathring S$ has four punctures (coming from the four boundary components of $S$). We write $\Diff^+(\mathring S, p_1, p_2)$ or $\Mod(\mathring S, p_1, p_2)$ below to denote the group of diffeomorphisms or mapping classes respectively of $\mathring S$ that fix each point $p_1, p_2$.
\begin{lem}[{Modification in $S$}]\label{lem:nielsen-S}
There exists a diffeomorphism $\psi \in \Diff^+(\mathring S, p_1, p_2)$ of order $2$ such that
\[
	[\psi] = [\tilde\eta_1|_{\mathring S}] \in \Mod(\mathring S, p_1, p_2).
\]
\end{lem}
\begin{proof}
Recall that $\tilde\eta_1 \in \Diff^+(\Sigma_{2g, 4})$ preserves the subsurface $S \subseteq \Sigma_{2g}$ and fixes each marked point $p_1, p_2, p_3, p_4$ of $\Sigma_{2g, 4}$. We claim that the restriction $\tilde\eta_1^2|_{\mathring S} \in \Diff^+(\mathring S, p_1, p_2)$ satisfies
\[
	[\tilde\eta_1^2|_{\mathring S}] = 1 \in \Mod(\mathring S, p_1, p_2).
\]
Because $[\tilde\eta_1^2] = 1 \in \Mod(\Sigma_{2g, 2})$ by Corollary \ref{cor:puncture-involution}, the curve $\tilde\eta_1^2(\gamma_0)$ is isotopic to $\gamma_0$ in $\Sigma_{2g, 2}$ for any curve $\gamma_0 \subseteq \mathring S$ that is not isotopic to $\alpha$, $\beta$, $\gamma$, or $\delta$ in $\Sigma_{2g, 2}$. By \cite[Lemma 3.16]{farb-margalit}, $\tilde\eta_1^2(\gamma_0)$ and $\gamma_0$ are also isotopic through an isotopy supported in $\mathring S$ fixing $p_1, p_2$ for all time $t$. By applying the Alexander method \cite[Proposition 2.8]{farb-margalit} to a set of filling curves of $\mathring S$ and its image under $\tilde\eta_1^2$, we conclude that
\[
	[\tilde\eta_1^2|_{\mathring S}] = 1 \in \Mod(\mathring S, p_1, p_2)
\]

By the Nielsen realization theorem \cite[Theorem 7.1]{farb-margalit}, there exists a diffeomorphism $\psi \in \Diff^+(\mathring S, p_1, p_2)$ of order $2$ such that
\[
	[\psi] = [\tilde\eta_1|_{\mathring S}] \in \Mod(\mathring S, p_1, p_2). \qedhere
\]
\end{proof}

Using the previous two lemmas, we now build an order-$2$ representative $\hat\eta$ of $\hat h \in \Mod(\Sigma_{2g, 2})$ with a controlled isotopy to a representative of $\tilde h \in \Mod(\Sigma_{2g, 4})$. See Figure \ref{fig:tubular-W-V} for a schematic of the resulting diffeomorphism $\hat \eta \in \Diff^+(\Sigma_{2g, 2})$ and isotopy $\kappa_t: \Sigma_{2g, 2} \times [0, 1] \to \Sigma_{2g, 2}$.
\begin{lem}[{Correction near $\alpha \cup \beta \cup \gamma \cup \delta$}]\label{lem:final-correction}
There exist
\begin{itemize}
\item a diffeomorphism $\hat \eta \in \Diff^+(\Sigma_{2g, 2})$ of order $2$ with $[\hat\eta] = \hat h \in \Mod(\Sigma_{2g, 2})$,
\item an isotopy $\kappa_t: \Sigma_{2g, 2} \times [0, 1] \to \Sigma_{2g, 2}$ with $\kappa_0 = \Id_{\Sigma_{2g, 2}}$,
\item disjoint tubular neighborhoods $W_\alpha$, $W_\beta$, $W_\gamma$, and $W_\delta$ of $\alpha$, $\beta$, $\gamma$, and $\delta$ in $\Sigma_{2g, 2}$, and
\item disjoint neighborhoods $O_i$ of $p_i$ for $i=1,2,3,4$
\end{itemize}
such that
\begin{enumerate}[(a)]
\item the diffeomorphism $\hat\eta \circ \kappa_1$ fixes pointwise $O_1,O_2,O_3,O_4$ and
  \[
    [\hat\eta \circ \kappa_1] = \tilde h \in \Mod(\Sigma_{2g, 4}),
  \]
\item the diffeomorphism $\hat\eta$ permutes the sets $W_\alpha$, $W_\beta$, $W_\gamma$, and $W_\delta$ by
\[
	\hat\eta(W_\alpha) = W_\gamma, \qquad \hat\eta(W_\beta) = W_\delta,
\]
and
\item for all $i=1,2,3,4$ and all $t \in [0,1]$,
  \[ \kappa_t(O_i) \cap (W_\alpha \sqcup W_\beta \sqcup W_\gamma \sqcup W_\delta) = \emptyset.\]
\end{enumerate}
\end{lem}
\begin{proof}
  In this proof we construct $\hat\eta$ by appropriately splicing together the order-$2$ pieces found in Lemmas \ref{lem:modifying-t1-t2} and \ref{lem:nielsen-S}.

  \medskip \noindent \emph{Step 0: Finding appropriate neighborhoods of $p_i$.}
  Since $\tilde\eta_1 \in \Diff^+(\Sigma_{2g,4})$, there exist (apply ~\cite[Theorem 8.3.1]{hirsch} to small enough neighborhoods of each marked point)
  \begin{enumerate}
  \item an isotopy $g_t:\Sigma_{2g,4} \to \Sigma_{2g,4}$ with $g_0=\Id_{\Sigma_{2g,4}}$ and compact support contained in $\mathring S$, and
  \item neighborhoods $O_1,O_2,O_3,O_4$ of $p_1,p_2,p_3,p_4$ such that $\overline{O_i} \subseteq \mathring S$, $\overline{O_i}$ is compact for all $i = 1, 2, 3, 4$, and $\overline{O_i} \cap \overline{O_j} = \emptyset$ for all $i \neq j$
  \end{enumerate}
  such that $\tilde\eta_1 \circ g_1$ fixes pointwise each $O_i$. In particular
  \[ [\tilde\eta_1 \circ g_1|_{\mathring S}] = [\tilde\eta_1] \in \Mod(\mathring S,p_1,p_2).\]
  Thus, replacing $\tilde\eta_1$ with $\tilde\eta_1 \circ g_1$ we can assume in Lemma~\ref{lem:nielsen-S} that
  $\tilde\eta_1$ acts trivially on each $O_i$ for $i=1,2,3,4$.

  Throughout, let $\lambda_t: (\mathring S, p_1, p_2) \times [0, 1] \to (\mathring S, p_1, p_2)$ be the isotopy with $\lambda_0 = \Id_{\mathring S}$ and $(\tilde\eta_1|_{\mathring S}) \circ \lambda_1 = \psi$ found in Lemma \ref{lem:nielsen-S}.

\medskip\noindent
\emph{Step 1: Fixing two tubular neighborhoods $V \subseteq W$ of $\alpha\cup \beta\cup \gamma\cup \delta$.} Recall the collar neighborhoods $U_\alpha$, $U_\beta$ of $\alpha$, $\beta$ in $T_1$ and $U_\gamma$, $U_\delta$ of $\gamma$, $\delta$ in $T_2$ found in Lemma \ref{lem:modifying-t1-t2}. Let $W_\alpha$ and $W_\beta$ be tubular neighborhoods of $\alpha$ and $\beta$ in $\Sigma_{2g, 4}$ respectively such that
\[
	\bar U_\alpha \subseteq W_\alpha, \qquad \bar U_\beta \subseteq W_\beta
\]
and such that the disks $\lambda_t^{-1}(O_i) \in S$ are disjoint from
\[
	(W_\alpha \cap \mathring S), \qquad (W_\beta \cap \mathring S), \qquad \psi(W_\alpha \cap \mathring S), \qquad \psi(W_\beta \cap \mathring S)
\]
for all $t \in [0, 1]$ and for all $i=1,2,3,4$. (Such choices of $W_\alpha$ and $W_\beta$ exist by compactness of the paths $\lambda_t^{-1}(\overline{O_i})$.) Then let
\[
	W_\gamma := \psi(W_\alpha \cap \mathring S) \cup \tilde\eta_1(W_\alpha \cap T_1), \qquad W_\delta := \psi(W_\beta \cap \mathring S) \cup \tilde\eta_1(W_\beta \cap T_1).
\]
By construction of $\psi$ and $\tilde\eta_1$, the open annuli $W_\gamma$ and $W_\delta$ are tubular neighborhoods of $\gamma$ and $\delta$ respectively in $\Sigma_{2g, 4}$ such that
\[
	\bar U_\gamma \subseteq W_\gamma, \qquad \bar U_\delta \subseteq W_\delta.
\]
Finally, let $W$ be the tubular neighborhood of $\alpha \cup \beta \cup \gamma \cup \delta$ in $\Sigma_{2g, 4}$ defined by
\[
	W := W_\alpha \cup W_\beta \cup W_\gamma \cup W_\delta.
\]

To specify $V \subseteq W$, choose collar neighborhoods $V_\alpha$, $V_\beta$ of $\alpha$ and $\beta$ in $S$ such that
\[
	\bar V_\alpha \subseteq W_\alpha, \qquad \bar V_\beta \subseteq W_\beta
\]
and let
\[
	V := (U_\alpha \cup V_\alpha) \cup (U_\beta \cup V_\beta) \cup (U_\gamma \cup \psi(V_\alpha \cap \mathring S)) \cup (U_\delta \cup \psi(V_\beta \cap \mathring S)) \subseteq W.
\]
Recall that $\tilde\eta_1(U_\alpha) = U_\gamma$ and $\tilde\eta_1(U_\beta) = U_\delta$ and that $\tilde\eta_1^2|_{T_1 \cup T_2}$ acts as the identity outside of $U_\alpha \sqcup U_\beta \sqcup U_\gamma \sqcup U_\delta$ by construction in Lemma \ref{lem:modifying-t1-t2}. Therefore, $\tilde\eta_1$ preserves both $W \cap (T_1 \cup T_2)$ and $V \cap (T_1 \cup T_2)$. On the other hand, $\psi$ has order $2$ and so $\psi$ preserves both $W \cap \mathring S$ and $V \cap \mathring S$. See Figure \ref{fig:tubular-W-V}.

\begin{figure}
\includegraphics[width=\textwidth]{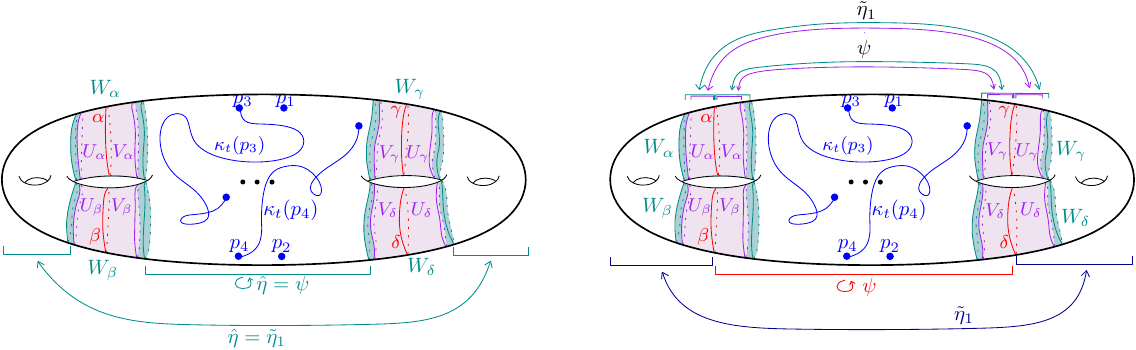}
\caption{Left: A schematic for the order-$2$ diffeomorphism $\hat\eta \in \Diff^+(\Sigma_{2g, 2})$ and isotopy $\kappa_t: \Sigma_{2g, 2} \times [0, 1] \to \Sigma_{2g, 2}$ in the conclusion of Lemma \ref{lem:final-correction}. Away from some tubular neighborhoods $W_\alpha, W_\beta, W_\gamma, W_\delta$ of $\alpha, \beta, \gamma, \delta$, the diffeomorphism $\hat\eta$ restricts to $\psi$ on $S$ and to $\tilde\eta_1$ on $T_1 \sqcup T_2$. Right: Choice of tubular neighborhoods $W$ and $V$ and the action of $\psi$ and $\tilde\eta_1$.}\label{fig:tubular-W-V}
\end{figure}

\medskip\noindent
\emph{Step 2: Isotoping $\tilde\eta_1$ in $S$ to have order $2$ away from $V$.} By the isotopy extension theorem \cite[Theorem 8.1.4]{hirsch}, there exists an isotopy $\tilde \lambda_t: \Sigma_{2g, 2} \times [0, 1] \to \Sigma_{2g, 2}$ with $\tilde \lambda_0 = \Id_{\Sigma_{2g, 2}}$ such that
\[
	\tilde \lambda_t|_{\mathring S - (S \cap V)} = \lambda_t|_{\mathring S - (S \cap V)}
\]
and such that $\tilde \lambda_t$ has compact support contained in $\mathring S \subseteq \Sigma_{2g, 2}$. In particular, the disks $\tilde \lambda_t^{-1}(O_i) = \lambda_t^{-1}(O_i)$ are contained in $S - (S \cap W)$ for all $t \in [0, 1]$ and for all $i=1,2,3,4$.

Let
\[
	\tilde \eta_2 := \tilde\eta_1 \circ \tilde\lambda_1 \in \Diff^+(\Sigma_{2g, 2}).
\]
We claim that $\tilde\eta_2$ preserves the two sets $V$, $W$ and that $\tilde\eta_2^2$ has support contained in $V$. To see that $\tilde\eta_2^2$ has support contained in $V$, we consider $T_1 \cup T_2$ and $S$ separately. In $T_1 \cup T_2$,
\[
	\tilde\eta_2^2|_{(T_1 \cup T_2) - V} = \tilde\eta_1^2|_{(T_1 \cup T_2) - V} = \Id_{(T_1 \cup T_2) - V}
\]
where first equality follows because $\tilde\lambda_1$ has compact support in $\mathring S$ and the second equality follows by Lemma \ref{lem:modifying-t1-t2} because $(T_1 \cup T_2)\cap V = U_\alpha \sqcup U_\beta \sqcup U_\gamma \sqcup U_\delta$. In $S$,
\[
	\tilde\eta_2^2|_{S-V} = (\tilde\eta_1 \circ \tilde\lambda_1)^2|_{S - V} = \psi^2|_{S - V} = \Id_{S - V},
\]
where the second equality follows because $\lambda_1|_{S - V} = \tilde\lambda_1|_{S - V}$ and the third equality follows because $\psi \in \Diff^+(\mathring S)$ has order $2$ by construction in Lemma \ref{lem:nielsen-S}.

To see that $\tilde\eta_2$ preserves $V$ and $W$, first recall that $W \cap (T_1 \cup T_2)$ and $V \cap (T_1 \cup T_2)$ are preserved by $\tilde\eta_1$, and that $\tilde\eta_1|_{T_1 \cup T_2} = \tilde\eta_2|_{T_1 \cup T_2}$. On the other hand, note that $\tilde\eta_2|_{S - V} = \psi|_{S - V}$, and that $W \cap \mathring S$ and $V \cap \mathring S$ are both preserved by $\psi$ by construction of $W$ and $V$.

\medskip\noindent
\emph{Step 3: Cutting and pasting in $W$.} By \cite[Proposition 2.4 and Lemma 3.17]{farb-margalit}, the restriction $\tilde \eta_2^2|_{\bar W}$ is topologically isotopic (rel $\partial \bar W$) to $\Id_{\bar W}$, because
\[
	[\tilde\eta_2^2] = [(\tilde\eta_1 \circ \tilde\lambda_1)^2] = [\tilde\eta_1^2] = \hat h^2 = 1 \in \Mod(\Sigma_{2g, 2}),
\]
where the last equality follows by Corollary \ref{cor:puncture-involution}. Therefore, there exists a topological isotopy $\mu_t: (\bar W_\alpha \sqcup \bar W_\beta) \times [0, 1] \to \bar W_\alpha \sqcup \bar W_\beta$ rel $\partial (\bar W_\alpha \sqcup \bar W_\beta)$ with $\mu_0 = \Id_{\bar W_\alpha \sqcup \bar W_\beta}$ such that
\[
	\tilde \eta_2^{-1}|_{\bar W_\alpha \sqcup \bar W_\beta} \circ \mu_1 = \tilde \eta_2|_{\bar W_\alpha \sqcup \bar W_\beta}: \bar W_\alpha \sqcup \bar W_\beta \to \bar W_\gamma \sqcup \bar W_\delta.
\]

With the isotopy $\mu_t$ in hand, consider the map
\[
	\hat\eta := \begin{cases}
	\tilde\eta_2 & \text{ on } \Sigma_{2g, 2} - (\bar W_\alpha \sqcup \bar W_\beta) \\
	\tilde\eta_2^{-1} & \text{ on } \bar W_\alpha \sqcup \bar W_\beta.
	\end{cases}
\]
The map $\hat\eta$ is smooth because $\tilde\eta_2^{-1}|_{\Sigma_{2g, 2} - V} = \tilde\eta_2|_{\Sigma_{2g, 2} - V}$ by construction in Step 2, since $\Sigma_{2g, 2} - V$ contains an open neighborhood of $\partial (\bar W_\alpha \sqcup \bar W_\beta)$. See Figure \ref{fig:tubular-W-V}.

The extension $\tilde \mu_t: \Sigma_{2g, 2} \times [0, 1] \to \Sigma_{2g, 2}$ of $\mu_t$ by the identity outside of $\bar W_\gamma \sqcup \bar W_\delta$ is a topological isotopy of $\Sigma_{2g, 2}$ with $\tilde\mu_0 = \Id_{\Sigma_{2g, 2}}$. Therefore,
\[
	[\hat \eta] = [\hat\eta \circ \tilde \mu_1] = [\tilde\eta_2] = \hat h \in \Mod(\Sigma_{2g, 2}).
\]

Furthermore, we claim that $\hat \eta$ has order $2$. Because $\tilde \eta_2^2$ has support contained in $V \subseteq W$, it suffices to check that $\hat \eta^2|_W = \Id_W$, which is true by construction. Moreover, $\hat\eta$ permutes the components of $W$ because $\tilde\eta_2$ does, and
\[
	\hat\eta(W_\alpha) = \tilde\eta_2(W_\alpha) = W_\gamma, \qquad \hat\eta(W_\beta) = \tilde\eta_2(W_\beta) = W_\delta.
\]
We have thus shown that $\hat\eta$ satisfies all the desired properties listed in the statement of the lemma.

\medskip\noindent
\emph{Step 4: Finding the isotopy $\kappa_t$.} Consider $\hat \eta \circ \tilde \lambda_1^{-1} \in \Diff^+(\Sigma_{2g, 2})$. Observe that $\hat \eta \circ \tilde \lambda_1^{-1}$ fixes pointwise the neighborhoods $O_i$ for $i=1,2,3,4$: For any $x \in O_i$,
\[
	(\hat\eta \circ \tilde \lambda_1^{-1})(x) = (\tilde\eta_2 \circ \tilde\lambda_1^{-1})(x) = (\tilde\eta_1 \circ \tilde \lambda_1)(\tilde \lambda_1^{-1}(x)) = \tilde\eta_1(x) = x,
\]
where the first equality holds because $\tilde \lambda_1^{-1}(O_i)$ is not contained in $\bar W$ by construction of $W$, and the last equality holds because $\tilde\eta_1$ fixes pointwise $O_i$ by Step 0. This shows that $\hat \eta \circ \tilde \lambda_1^{-1}$ is an element of $\Diff^+(\Sigma_{2g, 4})$ fixing each $O_i$ pointwise for $i=1,2,3,4$.

The topological isotopy $\tilde \lambda_1 \circ \tilde\mu_t \circ \tilde \lambda_1^{-1}: \Sigma_{2g, 2} \times [0, 1] \to \Sigma_{2g, 2}$ fixes the two points $p_3, p_4$ for all $t \in [0, 1]$ because $\tilde\mu_t$ is the identity outside $\bar W_\gamma \sqcup \bar W_\delta$. Therefore,
\[
	[\hat\eta \circ \tilde \lambda_1^{-1}] = [(\hat\eta \circ \tilde \lambda_1^{-1}) \circ (\tilde \lambda_1 \circ \tilde\mu_1 \circ \tilde \lambda_1^{-1})] = [(\hat\eta \circ \tilde\mu_1) \circ \tilde \lambda_1^{-1}] = [(\tilde\eta_1 \circ \tilde \lambda_1) \circ \tilde \lambda_1^{-1}] = [\tilde\eta_1] = \tilde h
\]
as mapping classes in $\Mod(\Sigma_{2g, 4})$. Finally, letting
\[
	\kappa_t := \tilde \lambda_t^{-1} : \Sigma_{2g, 2} \times [0, 1] \to \Sigma_{2g, 2}
\]
concludes the proof.
\end{proof}

We now collect the lemmas above to prove the main proposition of this appendix. 
\begin{proof}[Proof of Proposition \ref{prop:isotoping-hateta}]
  Fix the notation of the statement of Lemma \ref{lem:final-correction}. Let $W = W_\alpha \sqcup W_\beta\sqcup W_\gamma \sqcup W_\delta$ and $
  O = O_1\sqcup O_2 \sqcup O_3 \sqcup O_4$.
  The diffeomorphisms $\hat \eta$ and $\tilde\eta$ and the isotopy $\kappa_t$ found in Lemma \ref{lem:final-correction} satisfy the condition that $[\hat\eta] = \hat h \in \Mod(\Sigma_{2g, 2})$ and conditions \ref{lem:eta-construction-four-points} and \ref{lem:eta-construction-nice-iso}. It therefore suffices to construct the desired symplectic form $\theta$ and symplectomorphism $\varphi \in \Diff^+(\Sigma_{2g, 4})$ with $[\varphi] = \tilde f \in \Mod(\Sigma_{2g, 4})$ satisfying \ref{lem:eta-construction-symplectic} and \ref{lem:eta-construction-commutes}.

Let $\theta_\alpha$ be a symplectic form on $W_\alpha \subseteq \Sigma_{2g}$ and let $\varphi_\alpha\in \Symp(W_\alpha, \theta_\alpha)$ be a compactly supported, right-handed Dehn twist about $\alpha$. (For example, we can take $\theta_\alpha = d\theta \wedge dt$, where $\theta$, $t$ are the two coordinates of $W_\alpha \cong S^1 \times (0, 1)$, and $\varphi_\alpha(\theta, t) = (\theta - \rho(t), t)$ where $\rho: (0, 1) \to [0, 2\pi]$ is a smooth, increasing function with $\rho \equiv 0$ near $t = 0$ and $\rho \equiv 2\pi$ near $t = 2\pi$.) Similarly, let $\theta_\beta$ be a symplectic form on $W_\beta \subseteq \Sigma_{2g}$ and let $\varphi_\beta\in \Symp(W_\beta, \theta_\beta)$ be a compactly supported, right-handed Dehn twist about $\beta$.

Furthermore, we define similar symplectic forms and Dehn twists in $W_\delta$ and $W_\gamma$ by
\begin{align*}
	\varphi_\gamma := \hat\eta \circ \varphi_\alpha \circ \hat\eta, \qquad \theta_\gamma := \hat\eta^* \theta_\alpha, \\
	\varphi_\delta := \hat\eta \circ \varphi_\beta \circ \hat\eta, \qquad \theta_\delta := \hat\eta^*\theta_\beta.
\end{align*}
Here, recall from Lemma \ref{lem:final-correction} that $\hat\eta(W_\alpha) = W_\gamma$ and $\hat\eta(W_\beta) = W_\delta$ by construction.

With the above symplectic forms in hand, let $\theta$ be a symplectic form on $\Sigma_{2g}$ such that
\[
	\hat\eta^* \theta = \theta, \qquad \theta|_{ W_\alpha'} = \theta_\alpha, \quad \theta|_{ W_\beta'} = \theta_\beta, \quad \theta|_{ W_\gamma'} = \theta_\gamma, \quad \theta|_{ W_\delta'} = \theta_\delta.
      \]
Here $ W_\alpha', W_\beta',  W_\gamma', W_\delta'$ are suitable open subsets of $W_\alpha,W_\beta,W_\gamma,W_\delta$ respectively, each containing the support of $\varphi_\alpha,\varphi_\beta,\varphi_\gamma$ and $\varphi_\delta$ respectively. One way to find such a form $\theta$ is to first extend the forms $\theta_\alpha$, $\theta_\beta$, $\theta_\gamma$, and $\theta_\delta$ to all of $\Sigma_{2g}$ using a partition of unity argument, and then averaging this form under the action of $\langle \hat\eta \rangle \cong \Z/2\Z$. Now let
\[
	\varphi := \varphi_\alpha \circ \varphi_\beta^{-1} \circ \varphi_{\gamma} \circ \varphi_{\delta}^{-1}.
\]
Then $\varphi$ is supported in $W$ because $\varphi_\alpha$, $\varphi_\beta$, $\varphi_\gamma$, and $\varphi_\delta$ are, and hence $\varphi$ satisfies \ref{lem:eta-construction-symplectic}. Because as elements of $\Mod(\Sigma_{2g, 4})$,
\[
	[\varphi_\alpha] = T_x, \quad [\varphi_\beta] = T_y, \quad [\varphi_\gamma] = T_{\tilde h(x)}, \quad [\varphi_\delta] = T_{\tilde h(y)},
\]
there is also an equality of mapping classes $[\varphi] = \tilde f \in \Mod(\Sigma_{2g, 4})$ as desired.

To see that $\varphi$ preserves the form $\theta$, it suffices to check this on $W$, which contains the support of $\varphi$. On $W_\alpha \sqcup W_\beta$, $\varphi$ preserves $\theta$ because
\[
	(\varphi^*\theta)|_{W_\alpha} = \varphi_\alpha^*\theta_\alpha = \theta_\alpha, \qquad (\varphi^*\theta)|_{W_\beta} = (\varphi_\beta^{-1})^*\theta_\beta = \theta_\beta.
\]
On $W_\gamma \sqcup W_\delta$, $\varphi$ preserves $\theta$ because
\begin{align*}
	(\varphi^*\theta)|_{W_\gamma} &= \varphi_\gamma^*\theta_\gamma = (\hat\eta \varphi_\alpha \hat\eta)^*(\hat\eta^*\theta_\alpha) = \hat\eta^* \varphi_\alpha^*\theta_\alpha = \hat\eta^* \theta_\alpha = \theta_\gamma, \\
	(\varphi^*\theta)|_{W_\delta} &= (\varphi_\delta^{-1})^*\theta_\delta = (\hat\eta \varphi_\beta^{-1} \hat\eta)^*(\hat\eta^*\theta_\beta) = \hat\eta^* (\varphi_\beta^{-1})^*\theta_\beta = \hat\eta^* \theta_\beta = \theta_\delta.
\end{align*}

Finally, to see that $\hat\eta$ and $\varphi$ commute, compute
\begin{align*}
	\hat\eta \circ \varphi &= \hat\eta \circ (\varphi_\alpha \circ \varphi_\beta^{-1} \circ \varphi_\gamma \circ \varphi_\delta^{-1})= \hat\eta \circ (\varphi_\alpha \circ \varphi_\beta^{-1} \circ \varphi_\gamma \circ \varphi_\delta^{-1}) \circ \hat\eta^2= (\varphi_\gamma \circ \varphi_\delta^{-1} \circ \varphi_\alpha \circ \varphi_\beta^{-1}) \circ \hat\eta = \varphi \circ \hat\eta.
\end{align*}
This concludes the proof of \ref{lem:eta-construction-commutes}.
\end{proof}


\bibliographystyle{alphaurl}
\bibliography{lefschetz_ruled}

\end{document}